\documentclass[a4paper,reqno,11pt]{amsart}

\usepackage[labelfont={bf, footnotesize},textfont=footnotesize]{caption}

\usepackage{needspace}

\usepackage[normalem]{ulem}

\usepackage{amsmath, amsfonts, amssymb, amsthm, amscd}
\usepackage{graphicx}
\usepackage{psfrag}
\usepackage{perpage}
\usepackage{url}
\usepackage{color}
\usepackage{mathrsfs}
\usepackage{mathabx}
\usepackage{caption}
\usepackage{scalerel}

\usepackage{booktabs}

\usepackage{tikz}\usepackage{caption,subcaption}
\usetikzlibrary{arrows,decorations.pathmorphing,backgrounds,positioning,fit,petri}
\usepackage{subdepth}

\usetikzlibrary{shapes,backgrounds}

\newcommand{\cpow}{\mathord{\raisebox{-0.15ex}{\scalebox{1.15}{$\square$}}}}

\usepackage{dsfont} 

\usepackage[sort,compress]{cite}

\usepackage{bbm}
\usepackage[utf8]{inputenc}
\usepackage[T1]{fontenc}
\usepackage{lmodern}
\usepackage{microtype}

\definecolor{Sborder}{RGB}{80,227,194}
\definecolor{Sfill}  {RGB}{227,255,253}
\definecolor{Gedge}  {RGB}{80,207,21}
\definecolor{Redge}  {RGB}{239,68,33}

\usepackage{amsmath}
\usepackage{tikz}
\usepackage{mathdots}
\usepackage{color}
\usepackage{array}
\usepackage{multirow}
\usepackage{amssymb}
\usepackage{tabularx}
\usepackage{extarrows}
\usepackage{booktabs}
\usetikzlibrary{fadings}
\usetikzlibrary{patterns}
\usetikzlibrary{shadows.blur}
\usetikzlibrary{shapes}
\usepackage{hyperref}

\usepackage[a4paper,scale={0.72,0.74},marginratio={1:1},footskip=7mm,headsep=10mm]{geometry}

\usepackage{hyperref}

\makeatletter
\def\@secnumfont{\bfseries\scshape}

\def\section{\@startsection{section}{1}
	\z@{.9\linespacing\@plus\linespacing}{.5\linespacing}%
	{\normalfont\large\bfseries\scshape\centering}}

\def\subsection{\@startsection{subsection}{2}%
	\z@{.5\linespacing\@plus.7\linespacing}{-.5em}%
	{\normalfont\bfseries\scshape}}

\def\@secnumfont{\scshape}

\def\subsubsection{\@startsection{subsubsection}{3}%
	\z@{.5\linespacing\@plus.7\linespacing}{-.5em}%
	{\normalfont\scshape}}

\def\specialsection{\@startsection{section}{1}%
	\z@{\linespacing\@plus\linespacing}{.5\linespacing}%
	{\normalfont\centering\large\bfseries\scshape}}
\makeatother

\makeatletter

\renewenvironment{proof}[1][\proofname]{\par
	\pushQED{\qed}%
	\normalfont \topsep4\p@\@plus4\p@\relax
	\trivlist
	\item[\hskip\labelsep
	\bfseries
	#1\@addpunct{.}]\ignorespaces
}{%
	\popQED\endtrivlist\@endpefalse
}
\makeatother

\makeatletter
\newcommand \Dotfill {\leavevmode \leaders \hb@xt@ 6pt{\hss .\hss }\hfill \kern \z@}
\makeatother

\makeatletter
\def\@tocline#1#2#3#4#5#6#7{\relax
	\ifnum #1>\c@tocdepth 
	\else
	\par \addpenalty\@secpenalty\addvspace{#2}%
	\begingroup \hyphenpenalty\@M
	\@ifempty{#4}{%
		\@tempdima\csname r@tocindent\number#1\endcsname\relax
	}{%
		\@tempdima#4\relax
	}%
	\parindent\z@ \leftskip#3\relax \advance\leftskip\@tempdima\relax
	\rightskip\@pnumwidth plus4em \parfillskip-\@pnumwidth
	#5\leavevmode\hskip-\@tempdima
	\ifcase #1
	\or\or \hskip 1.65em \or \hskip 3.3em \else \hskip 4.95em \fi%
	#6\nobreak\relax
	\Dotfill
	\hbox to\@pnumwidth{\@tocpagenum{#7}}\par
	\nobreak
	\endgroup
	\fi}
\makeatother

\makeatletter
\def\l@section{\@tocline{1}{0pt}{1pc}{}{\scshape}}
\renewcommand{\tocsection}[3]{%
	\indentlabel{\@ifnotempty{#2}{\ignorespaces#1 #2.\hskip 0.7em}}#3}
\def\l@subsection{\@tocline{2}{0pt}{1pc}{5pc}{}}

\def\l@subsubsection{\@tocline{3}{0pt}{1pc}{7pc}{}}

\makeatother

\numberwithin{equation}{section}

\newtheoremstyle{mytheorem}{.7\linespacing\@plus.3\linespacing}{.7\linespacing\@plus.3\linespacing}%
{\itshape}
{}
{\bfseries}
{. }
{0.3ex}
{\thmname{{\bfseries #1}}\thmnumber{ {\bfseries #2}}\thmnote{ (#3)}}  

\theoremstyle{mytheorem}

\newtheorem{theorem}{Theorem}[section]
\newtheorem{lemma}[theorem]{Lemma}
\newtheorem{convention}[theorem]{Convention}
\newtheorem{proposition}[theorem]{Proposition}

\newtheorem{remark}[theorem]{Remark}
\newtheorem{definition}[theorem]{Definition}

\newtheorem{example}[theorem]{Example}

\newcommand{\bbE}{{\ensuremath{\mathbb E}} }

\newcommand{\bbS}{{\ensuremath{\mathbb S}} }

\newcommand{\cE}{{\ensuremath{\mathcal E}} }
\newcommand{\cF}{{\ensuremath{\mathcal F}} }
\newcommand{\cG}{{\ensuremath{\mathcal G}} }

\newcommand{\cL}{{\ensuremath{\mathcal L}} }

\newcommand{\cN}{{\ensuremath{\mathcal N}} }

\newcommand{\cR}{{\ensuremath{\mathcal R}} }

\DeclareMathSymbol{\leqslant}{\mathalpha}{AMSa}{"36} 
\DeclareMathSymbol{\geqslant}{\mathalpha}{AMSa}{"3E} 
\DeclareMathSymbol{\eset}{\mathalpha}{AMSb}{"3F}     

\newcommand{\be}{\begin{equation}}
	\newcommand{\ee}{\end{equation}}

\newcommand{\R}{\mathbb{R}}

\newcommand{\N}{\mathbb{N}}

\newcommand{\ind}{\mathds{1}}

\renewcommand{\epsilon}{\varepsilon}
\renewcommand{\phi}{\varphi}
\renewcommand{\theta}{\vartheta}
\renewcommand{\rho}{\varrho}

\newenvironment{myenumerate}{%
	\renewcommand{\theenumi}{\arabic{enumi}}%
	\renewcommand{\labelenumi}{{\rm(\theenumi)}}%
	\begin{list}{\labelenumi}
		{%
			\setlength{\itemsep}{0.4em}%
			\setlength{\topsep}{0.5em}%
			\setlength\leftmargin{2.45em}%
			\setlength\labelwidth{2.05em}%
			\setlength{\labelsep}{0.4em}%
			\usecounter{enumi}%
		}%
	}%
	{\end{list}
}

	{\end{list}
}

	{\end{list}
}

\renewenvironment{enumerate}{
	\begin{myenumerate}}%
	{\end{myenumerate}}

\newenvironment{myitemize}{%
	\begin{list}{$\bullet$}%
		{%
			\setlength{\itemsep}{0.4em}%
			\setlength{\topsep}{0.5em}%
			\setlength\leftmargin{2.65em}%
			\setlength\labelwidth{2.65em}%
			\setlength{\labelsep}{0.4em}%
		}%
	}%
	{\end{list}}

\renewenvironment{itemize}{
	\begin{myitemize}}%
	{\end{myitemize}}

\MakePerPage[2]{footnote} 

\usepackage{cleveref}

\newcommand{\rme}{\mathrm{e}}

\usepackage{enumitem}  

\newlist{conditions}{enumerate}{1}
\setlist[conditions,1]{label=\textbf{Condition (R)}, ref=R}

\title{\bf On irreducible central limit theorems}

\author[F. Caravenna]{Francesco Caravenna}
\address{Dipartimento di Matematica e Applicazioni\\
 Universit\`a degli Studi di Milano-Bicocca\\
 via Cozzi 55, 20125 Milano, Italy}
\email{francesco.caravenna@unimib.it}

\author[F. Cottini]{Francesca Cottini}
\address{LPSM, Sorbonne Université, 4 place Jussieu, 75005 Paris, France}
\email{francesca.cottini@sorbonne-universite.fr}

\author[G. Peccati]{Giovanni Peccati}
\address{Department of Mathematics, University of Luxembourg, Maison du Nombre
6, avenue de la Fonte
L-4364 Esch-sur-Alzette, Luxembourg}
\email{giovanni.peccati@uni.lu}

\keywords{Central Limit Theorem; Combinatorial Dimension; Cheeger Inequalities; Connectivity; Graphs; Homogeneous Sum; Hypergraphs; Irreducibility;     Polynomial Chaos}

\subjclass{60F05; 60C05; {05C40}}

\begin{document}
	
\maketitle

\begin{abstract}
    We consider sequences of homogeneous sums based on independent random variables and satisfying a central limit theorem (CLT). We address the following question:
\begin{center}
``{\it In which cases is it \emph{not possible} to reduce such an asymptotic result to the classical \\ Lindeberg--Feller CLT through a restriction of the summation domain}?''
\end{center}
We provide several sufficient conditions for such irreducibility, expressed both in terms of (hyper)graphs Laplace eigenvalues, and of a certain notion of {\it combinatorial dimension}. 
Our analysis combines Cheeger-type inequalities with fourth moment theorems, showing that the irreducibility of a given CLT for homogeneous sums can be naturally encoded by the connectivity properties of the associated sequence of weighted hypergraphs. Several ad-hoc constructions are provided in the special case of quadratic forms.
\end{abstract}

\tableofcontents

\section{Introduction}

\subsection{Overview}
{This paper focuses on \emph{homogeneous sums} --- sometimes called {\em homogeneous polynomials} or {\em polynomial chaoses} --- based on independent random variables, see \eqref{eq:defpolchaos} for a definition. These objects are quintessential examples of {\em degenerate $U$-statistics} \cite{efronstein, Hoeffding, korolyuk, serfling, MandelbaumTaqqu83}, and play a pivotal role e.g. in the construction of \emph{Wiener chaoses} over Gaussian or Poisson random measures \cite{NP12, NuaBook, JansonBook, PecTaqBook, PecReitz}, in the study of directed polymers \cite{CSZ1, CSZ17b, CSZ3, CC22}, and in the analysis of Boolean (and more general) functions on discrete structures \cite{DeServedio1, DeServedio2, Garban2011, GPS10, MOO10, ODonnellBook, LPY}. 
 
\smallskip 

Over the past few decades, several central limit theorems (CLTs) have been established for sequences of homogeneous sums, ranging from the many ramifications of the fundamental contributions by P. de Jong \cite{dJ87, dJ90, dejongRSA ,DPKPTRF, KoikeJOTP, DoblerSPL, DoblerKrokowski} to the {\em universal CLTs} established in \cite{NPR10, NPPS}  (see also \cite[Chapter 11]{NP12}), whose proofs combine the discrete Fourier techniques from \cite{MOO10, R79} with classical {\it fourth moment theorems} on Gaussian Wiener chaoses \cite{NP05, NP12}.

\smallskip 

Our goal in this paper is to investigate some previously unexplored aspects of these results. To this end, let us first recall the {\it universal de Jong CLT} established in \cite[Theorem 1.10]{NPR10}, formally stated in Theorem~\ref{t:unidejong} below:

\smallskip 
{
\begin{quote}
\it
Let $\{Z_n({\bf X}) : n \geq 1\}$ be a sequence of unit variance homogeneous sums, of fixed order $d \geq 2$ and based on a sequence of i.i.d.\ standard Gaussian random variables ${\bf X} = \{X_i\}$. Then, $Z_n({\bf X})$ converges in distribution to a standard normal random variable $N \sim \mathcal{N}(0,1)$ if and only if $\mathbb{E}[Z_n({\bf X})^4] \to \mathbb{E}[N^4] = 3$. If  $Z_n({\bf X})$ converges in distribution to $N$, then this convergence continues to hold if ${\bf X}$ is replaced by an arbitrary sequence of i.i.d.\ random variables $\{Y_i\}$ with mean zero and unit variance.\footnote{ This statement is called a ``de Jong CLT'' since the characterization of a central limit theorem through a fourth-moment condition matches phenomena first identified in \cite{dJ87, dJ90}. We also recall that, since ${\bf X}$ is a Gaussian family, the relation
\begin{equation}\label{e:4thmom}
Z_n({\bf X})\xlongrightarrow{d} N \quad \Longleftrightarrow \quad \mathbb{E}[Z_n({\bf X})^4] \to 3
\end{equation}
(where $\xlongrightarrow{d}$ indicates convergence in distribution of random variables) follows from \cite{NP05}, together with the fact that each $Z_n({\bf X})$ lies in the $d$th Wiener chaos associated with ${\bf X}$.}
\end{quote}
}

\smallskip 

The results from \cite{NPR10} have led to notable applications to directed polymers \cite{CSZ1, CSZ17b}, as well as to Salem–Zygmund CLTs for random polynomials \cite{AngstPolyEJP}, and to spectral fluctuations of non-Hermitian Gaussian random matrices \cite{NouPecALEA}; see also \cite{DeyaNourdinBernoulli, KoikeJOTP, SimoneALEA}. 

An intriguing observation made recently in \cite{CC22} is that, in many cases relevant to the study of directed polymers and stochastic PDEs, the convergence in law of $Z_n({\bf X})$ to a Gaussian distribution can be deduced directly from the classical Lindeberg-Feller CLT (see, e.g., \cite[Theorem 4.12]{kallengerg}) by restricting the summation domain of $Z_n({\bf X})$ in such a way that \emph{the homogeneous sum becomes a sum of independent random elements, up to a negligible remainder}. In this way, one can bypass the need to estimate higher-order moments or cumulants, such as those of order four. Throughout this paper, we will refer to such a simplified scenario as that of a {\it reducible CLT}.

\smallskip 

The goal of our work is to understand to what extent this simplification applies, and in which cases it fails to hold. In particular, our results address the following question:

\begin{enumerate}[label=\textbf{(Q1)}]
\item \label{q}  \begin{center}
\emph{Assume that the homogeneous sums $\{Z_n({\bf X})\}$ verify a CLT.\
Under which conditions is such a CLT {\em irreducible}, in the sense that it cannot be deduced from the classical Lindeberg--Feller theorem through a restriction of the summation domain?}
\end{center}
\end{enumerate}

Motivated by Question~\ref{q}, we begin by introducing a rigorous notion of \emph{(ir)reducibility} for homogeneous sums (see Definition~\ref{def:irreducibility}). We then establish sufficient conditions for irreducibility and construct explicit families of sequences of homogeneous sums that satisfy \emph{irreducible CLTs}---that is, they converge in distribution to a Gaussian limit, yet this convergence cannot be deduced by restricting the summation domain and applying the classical Lindeberg--Feller theorem.

\smallskip 

The examples developed in this work are closely tied to the combinatorial structure of the summation sets $E_n$ associated with the homogeneous sums $Z_n({\bf X})$, with the property of irreducibility naturally arising from the connectivity features of $E_n$. Our main contributions, discussed
in Section~\ref{sec:generalsett}, rely in particular on tools from graph theory, including {\it Cheeger-type inequalities} \cite{Che71, AM85, Alo86, SJ89, LGT14} and their extensions to hypergraphs \cite{B21, SSPHypergraphs}. We also present examples grounded in the concepts of {\it combinatorial dimension} and {\it fractional Cartesian products} \cite{blei1979,BleiSRandom, bleibook, Blei2011Survey, BleiPeresScchmerl, bleijanson, DPKPTRF, NPR10}. The inherent challenges involved in establishing {\em necessary conditions} for irreducibility are discussed in Subsection~\ref{ss:further}.

\smallskip

As illustrated below, our results complement and refine the recent characterizations of the asymptotic behaviour of homogeneous sums derived in \cite{BDMM22, HeMaPolyLaw}, and provide a novel perspective on the {\it fourth moment phenomenon} for chaotic random variables \cite{NP05, NP12, NouPecSwanJFA}. We observe that fourth moment theorems have recently played a crucial role in the derivation of CLTs for level sets of Gaussian waves \cite{NouPecRossi19, DierickxNPRCMP, PecVidottoJMP}, and have non-trivial counterparts in a noncommutative setting \cite{KNPSaop, DeyaNourdinCMP}. See also \href{https://sites.google.com/site/malliavinstein/home}{I. Nourdin's dedicated webpage} for a comprehensive list of references around the fourth moment phenomenon.

\begin{remark}[Spectral vs. domain reducibility for quadratic forms]{\rm 
The asymptotic normality of homogeneous sums of order~$2$ based on Gaussian variables can \emph{always} be inferred from the Lindeberg--Feller theorem, via a spectral decomposition. Indeed,
for a sequence $\{Z_n :  n \geq 1 \}$ in the {\it second Wiener chaos} of a separable Gaussian field \footnote{Recall that homogeneous sums of order 2 based on a i.i.d.\ standard Gaussian family ${\bf X}$ are typical elements of the second Wiener chaos associated with ${\bf X}$ (see \cite[Sections 2.2 and 2.7.4]{NP12}).}, the following holds (see, e.g., \cite[Proposition 2.7.11]{NP12}):
\begin{enumerate}[label=(\roman*)]
\item Each $Z_n$ admits a spectral representation of the following form,
for  $N\in \mathbb{N}\cup \{+ \infty\}$:
\begin{equation}\label{e:chi2}
   Z_n = \sum_{j=1}^N \gamma_{j,n}\, \xi_{j,n}, 
\end{equation}
where $\{\xi_{j,n}\}$ are i.i.d.\ centered chi-square random variables with one degree of freedom and $\{\gamma_{j,n}\}$ are the eigenvalues of a suitable Hilbert--Schmidt operator.

\item A direct application of the Lindeberg--Feller CLT yields the following equivalence: if ${\bf Var}(Z_n)\to 1$, then
\begin{equation}\label{e:maxgen}
Z_n \stackrel{d}{\longrightarrow} N \sim \mathcal{N}(0,1)
\quad \Longleftrightarrow \quad
\max_{j \geq 1} |\gamma_{j,n}| \to 0.
\end{equation}
\end{enumerate}
The kind of {reducibility} studied in the present paper is of a stronger nature: we are interested in understanding whether a CLT for $\{Z_n\}$ can be deduced by the Lindeberg--Feller theorem via a \emph{restriction of the summation domain}, rather than through implicit spectral conditions.
 As we shall demonstrate, the spectral criterion \eqref{e:maxgen} does \emph{not always} imply reducibility in our domain-based sense. For a direct application of criterion \eqref{e:maxgen} to Gaussian quadratic forms with $0$--$1$ coefficients, see Remark~\ref{r:afterdejong}-\ref{item:rempoint2}. To the best of our knowledge, no analogue of the representation \eqref{e:chi2} exists for elements of Wiener chaoses of order $d \geq 3$. As a result, no systematic ``spectral'' reduction to the Lindeberg--Feller theorem is available for homogeneous sums of higher order; see also \cite{HeMaPolyLaw}.
}
\end{remark}%

\begin{remark}[Reducibility and $U$-statistics]{\rm The notions of reducibility and irreducibility introduced in Definition~\ref{def:irreducibility} below extend canonically from homogeneous sums to the broader class of degenerate $U$-statistics \cite{dJ90, serfling, DPejp, DPKPTRF, efronstein, korolyuk}. This observation suggests a natural generalization of Question~\ref{q}: one could ask which sequences of degenerate $U$-statistics satisfy a CLT that cannot be reduced to the Lindeberg–Feller framework. We view this as a distinct research direction, which we leave open for future investigation.
}\end{remark}
}

\subsection{Organization of the paper/tables}

In Section~\ref{sec:generalsett}, we outline our general setting, introduce a rigorous definition of (ir)reducibility (Definition~\ref{def:irreducibility}), and present some of our main results (Theorems~\ref{th:mainSPECTRUMgraph}, \ref{th:irreducombdim}, and~\ref{th:irreducibilityk=2}). We then discuss some of the challenges involved in identifying necessary conditions for irreducibility, see  Subsection~\ref{ss:further}.

\smallskip

In Section~\ref{sec:examples}, we present a wider discussion of our main results, illustrated with some explicit, non-trivial examples.
The remaining sections are devoted to the proof of our results. 

\smallskip

In 
Section~\ref{sec:GRAPHS}, we prove our spectral criteria in the simplest setting of homogeneous sums of order $d=2$, corresponding to \emph{graphs}. After reviewing in Subsection~\ref{ss:pregraph} basic notions, normalized Laplacians, and related connectivity estimates (Cheeger's inequalities),
Subsection~\ref{subsec:graphirreduc} contains the proof of Theorem~\ref{th:mainSPECTRUMgraph}, while Subsections~\ref{subsec:examples} and~\ref{ss:full} are devoted, respectively, to examples and to a strengthened notion of irreducibility.

\smallskip

Our spectral results are then extended in
Section~\ref{sec:Hypergraphs} to homogeneous sums of generic order $d \geq 2$, relying on an extension of Cheeger-type estimates to the setting of weighted hypergraphs (see Proposition~\ref{prop:cheegerhyp} and Theorem \ref{th:mainresulthyper}).

\smallskip

Section~\ref{sec:sparsity} proves and illustrates Theorem~\ref{th:irreducombdim}, and Section~\ref{sec:d=2} offers a detailed analysis of the ad-hoc construction for homogeneous sums of order~2 that appears in Theorem~\ref{th:irreducibilityk=2}. Appendix~\ref{appendix} gathers additional proofs omitted in the main text, while Appendix~\ref{appendixB} collects some preliminaries on Cartesian products of graphs.

\medskip

For the reader’s convenience, all irreducible CLTs discussed in this paper are summarized in {\bf Table~A}, organized by the criterion used to establish irreducibility and the order of the corresponding homogeneous sums. Similarly, {\bf Table B} displays a list of the reducible CLTs analyzed below, complete with their location.

{\small  
\begin{figure}[ht]
\begin{tikzpicture}
\node (table) {
\begin{tabular}{ccc}
\toprule
\textbf{\small Criterion} $\backslash$ \textbf{\small Order} & $d=2$ & $d\geq 3$ \\
\midrule
{\it Spectral} &
\begin{tabular}{c}
{\small Theorem \ref{th:mainSPECTRUMgraph}}, Proposition \ref{p:bella} \\
Examples \ref{item:expanders}, \ref{item:completegraph}-\ref{item:ci}, \ref{item:bipartite}, \\ \ref{item:cartprod(e)}-\ref{item:eii}, \ref{item:cartprod(e)}-\ref{item:eiii} in Section \ref{subsec:examples} 
\end{tabular} &
\begin{tabular}{c}
Theorem \ref{th:mainresulthyper}, 
Example \ref{ex:rooklike}, \\ Example \ref{e:3uni}
\end{tabular} \\
\midrule
{\it Combinatorial dimension} &
\begin{tabular}{c}
Theorem \ref{t:probamethod}
\end{tabular} &
\begin{tabular}{c}
Theorem \ref{th:irreducombdim}, \\
Example \ref{e:fcp},
Proposition \ref{cor:fracprod}
\end{tabular} \\
\midrule
{\it Ad-hoc construction} &
\begin{tabular}{c}
Theorem \ref{th:irreducibilityk=2}, 
Section \ref{sec:d=2}
\end{tabular} &
\begin{tabular}{c}
 ---
\end{tabular} \\
\bottomrule
\end{tabular}
};
\node[below=0.3cm of table] {%
{{\bf Table A}: Irreducible CLTs and where to find them.}
};
\end{tikzpicture}
\end{figure}
}

{\small
\begin{table}[ht]
\centering
\begin{tabular}{cc}
\toprule
\textbf{Reducible CLT} & \textbf{Location} \\
\midrule
{\it Trivial} & Example \ref{e:ridux}-\eqref{it:ridux-trivial} \\
\midrule
{\it Disjoint unions} & Example \ref {e:ridux}-\eqref{it:ridux-rook}  \\
\midrule
{\it Variations of the Rook's graph} & Example \ref{e:ridux}-\eqref{it:ridux-rook} and Section \ref{subsec:examples}-\ref{item:cartprod(e)} \\
\midrule
{\it The (generalized) hypercube} & Example \ref{e:ridux}-\eqref{it:ridux-cube} and Section \ref{subsec:examples}-\ref{item:completegraph}  \\
\bottomrule
\end{tabular}
\caption*{\textbf{Table B}: Reducible CLTs.}
\end{table}
}

\medskip

From now on, every random object is assumed to be defined on a common probability space $(\Omega, \mathcal{F}, \mathbb{P})$, with $\mathbb{E}$ denoting an expectation with respect to $\mathbb{P}$.

\subsection{Acknowledgments} This research is supported by the Luxembourg National Research Fund (AFR/22/17170047/Bilateral-GRAALS). The first two authors acknowledge the support of INdAM/GNAMPA.

\subsection{Data availability statement} Data sharing is not applicable to this article as no datasets were generated or analyzed during the current study.

\section{General setting and main results}\label{sec:generalsett}
\subsection{Preliminaries}\label{ss:prelim}
{
 We consider a sequence of nonempty finite sets $\{V_n : n\geq 1\}$ with the property that 
\begin{equation}\label{eq:propV}
\quad  |V_n | 
\to \infty,\quad \mbox{as } \,n \to \infty.
\end{equation}
We define $V := \bigcup_n V_n$ and let ${\bf X} = \{X_v : v\in V\}$ be a collection of i.i.d. standard normal random variables. For a fixed $d\geq 2$, we let $\{E_n : n\geq 1\}$ denote a sequence of sets such that
\begin{enumerate} 
\item \label{item:VE1} $$E_n\subset \underbrace{ \, V_n\times \cdots \times V_n \, }_\text{$d$ times}\,, \,\, n\geq 1,\quad \mbox{and}\quad |E_n|\to \infty;$$

\item \label{item:VE2} each $E_n$ is \emph{symmetric}: if $(v_1,...,v_d)\in E_n$, then $(v_{\sigma(1)},...,v_{\sigma(d)} )\in E_n$ for all permutations $\sigma$ of $[d]:=\{1,\ldots,d\}$;

\item \label{item:VE3} each $E_n$ is \emph{non-diagonal}: if $(v_1,...,v_d)\in E_n$, then $v_i\neq v_j$ for all $i\neq j$.

\end{enumerate}

\smallskip

\noindent For every $n\geq 1$, we also consider a mapping $$q_n : E_n\to \mathbb{R} : (v_1,...,v_d) \mapsto q_n(v_1,...,v_d),$$ that we assume to satisfy the following properties:
\begin{enumerate}[label=(\alph*)]
    \item $q_n$ is symmetric: $q_n(v_1,...,v_d) =q_n(v_{\sigma(1)},...,v_{\sigma(d)}) $ for every permutation $\sigma$ of $[d]$;
    \item as $n\to \infty$,
\begin{equation}\label{e:varinfty}
\sum_{(v_1,...,v_d)\in E_n} q_n(v_1,...,v_d)^2 :=\| q_n\|^2\to \infty\,.     
\end{equation}
\end{enumerate} 
Plainly, one can identify each $q_n$ with a symmetric mapping on $V_n\times \cdots \times V_n$, with support equal to $E_n$. We will sometimes refer to the quantities $q_n(v_1,...,v_d)$ as the {\em coefficients} of the homogeneous sum, and to the squared coefficients 
\begin{equation}\label{e:weights}
w_n(v_1,...,v_d):= q_n(v_1,...,v_d)^2
\end{equation}
as the {\em weights}.

\medskip
As anticipated, in this work we will focus on sequences $\{Z_n : n\geq 1\}$ of \emph{homogeneous sums} of order $d$, based on the Gaussian family ${\bf X}$ and defined starting from the triplets $(V_n, E_n, q_n)$ introduced above; for every $n\geq 1$, these are defined as follows:
\begin{equation}\label{eq:defpolchaos}
    Z_n = Z_n({\bf X}) :=  \sum_{(v_1,\ldots,v_d)\, \in E_n } q_n(v_1,\ldots,v_d) \, \prod_{i=1}^d X_{v_i} \,.
\end{equation}
One particularly interesting case considered below (studied, for instance, in \cite{BDMM22} when $d = 2$) is that of constant coefficients, $q_n \equiv 1$. In this setting, the distribution of the random variable $Z_n$ is entirely determined by the combinatorial structure of the set $E_n$.

\smallskip

Since the random variables $X_v$'s entering the definition \eqref{eq:defpolchaos} are independent, centered with unit variance, one has trivially that (using the notation \eqref{e:varinfty})
\begin{equation}\label{eq:secmomZn}
    \bbE [ Z_n]=0 \qquad \text{and}\qquad  \bbE [ Z_n^2] =d! \| q_n\|^2 \,,
\end{equation}
and we observe that, if $q_n^2\equiv 1$ (constant weights), the second relation in \eqref{eq:secmomZn} reduces to $\bbE [ Z_n^2] = d!|E_n|$. We are particularly interested in those sequences of homogeneous sums satisfying a {\em central limit theorem} (CLT), that is, such that
\begin{equation}\label{eq:CLT}
    \widetilde{Z}_n =\widetilde{Z}_n({\bf X})  \coloneq \frac{Z_n({\bf X})}{\sqrt{d!\|q_n\|^2}} \, \xlongrightarrow[n\to\infty]{d} \, N\sim  \cN \big(0,1)\,,
\end{equation}
where $\xlongrightarrow{d}$ indicates, as before, convergence in distribution. 

\smallskip

It is by now a classical result (see the Introduction, as well as\cite[Chapter 5 and Chapter 11]{NP12}) that the convergence \eqref{eq:CLT} can be established by a substantial simplification of the method of moments --- via so-called \emph{fourth moment theorems} --- and has moreover a {\em universal nature} that is grounded in the polynomial chaos estimates from \cite{MOO10, R79}. 
For the reader's convenience, we discuss these results in Subsection~\ref{sec:4th} below (see, in particular, Theorem \ref{t:unidejong}).

}
\subsection{Irreducibility}\label{ss:irred}
 { Sometimes the CLT \eqref{eq:CLT} can hold as a direct consequence of the classical \emph{Lindeberg–Feller CLT for triangular arrays} \cite[Theorem 4.12]{kallengerg}. For instance, for $d \geq 1$ and $V = \mathbb{N}$, one easily verifies that the sequence
\begin{equation}\label{eq:exFL}
\tilde{Z}_n := \frac{1}{\sqrt{n}} \sum_{i=0}^{n-1} \prod_{\ell = 1}^d X_{id+\ell} 
= \frac{(X_1 \cdots X_d) + (X_{d+1} \cdots X_{2d}) + \ldots + (X_{(n-1)d + 1} \cdots X_{nd})}{\sqrt{n}}
\end{equation}
with $\{X_i\}$ an i.i.d.\ standard Gaussian family, converges in distribution to a standard Gaussian law as a direct consequence of the usual CLT. As anticipated, the initial impetus for our work comes from reference \cite{CC22}, where the authors identified a class of highly non-trivial examples for which a {\it domain-wise} reduction of Theorem~\ref{t:unidejong} to the Lindeberg--Feller setting remains possible. Their approach, motivated by applications to directed polymers, also extends to superpositions of multiple chaos orders, relies solely on second-moment computations, and provides natural criteria for proving Gaussian fluctuations in contexts such as singular stochastic PDEs \cite{CSZ1}.

\smallskip

As discussed in the Introduction, this paper aims to identify sequences of homogeneous sums whose Gaussian convergence \emph{cannot} be deduced from a Lindeberg--Feller-type central limit theorem via a restriction of the summation domains; we call such sequences \emph{irreducible}.  
Building on \cite[Theorem~2.1]{CC22}, we now provide a rigorous definition of (ir)reducibility adapted to our setting.

\smallskip

Let the notation and assumptions of Section \ref{ss:prelim} prevail; for any subset $B \subset V_n$, we denote by $\sigma^2_n(B)$ the ``contribution of $B$ to the variance'' defined by
\begin{eqnarray}\label{eq:sigmaBdef}
    \sigma^2_n(B) &\coloneq & d! \sum_{v_1,\ldots,v_d \in B} \ind_{E_n}(v_1,\ldots,v_d) q_n(v_1,...,v_d)^2 \\ \notag 
    &=&  d! \!\!\! \sum_{\substack{(v_1,\ldots,v_d)\in \\  E_n\cap (B\times \cdots \times B)}  }\!\!w_n(v_1,...,v_d)  \, ,
\end{eqnarray}
(where we have used the notation \eqref{e:weights}) and observe that, when $w_n\equiv 1$, one has simply that
\begin{equation}\label{eq:sigmaBdef2}
\sigma^2_n(B)= d!\, \big| \, E_n \cap (\underbrace{B \times \cdots \times B}_\text{$d$ times})\, \big| \, . 
\end{equation}

\begin{definition}[Reducibility/Irreducibility]\label{def:irreducibility}
    Fix $d \ge 2$ and consider the sequence of homogeneous sum $\{ Z_n({\bf X}): n \geq 1\} = \{Z_n \} $ defined in \eqref{eq:defpolchaos}.  We say that $\{Z_n\}$ is \emph{reducible} if for any $n \in \N$ there exist subsets (``boxes'') $B_1,\ldots,B_{m_n} \subset V_n$ (where $B_i = B_i{(n)}$ may depend on $n$), such that $B_i \cap B_j = \emptyset$ for $i\ne j$ (that is, the boxes $B_i$ are disjoint) and
\begin{enumerate}[label={\rm (\roman*)}]
    \item \label{redCLTcond1} $m_n \to \infty \,;$
    \item \label{redCLTcond2} $\lim_{n\to\infty}\frac{1}{d!\|q_n\|^2} \sum_{i=1}^{m_n} \sigma_n^2(B_i) = 1\,;$
    \item \label{redCLTcond3} $\lim_{n\to\infty}  \frac{1}{d!\|q_n\|^2} \max_{i=1,\ldots,m_n} \sigma_n^2(B_i) =0 \,.$ 
\end{enumerate}
We say that $\{Z_n\}$ is \emph{irreducible} if it is not possible to find disjoint sets $B_1,\ldots,B_{m_n} \subset V_n$ such that \ref{redCLTcond1}, \ref{redCLTcond2} and \ref{redCLTcond3} hold. If $\{Z_n\}$ is irreducible and verifies \eqref{eq:CLT}, we say $\{Z_n\}$ {\em verifies an irreducible CLT}. \end{definition}

\begin{remark}{\rm
The fact that reducible sequences always verify a CLT was a key observation in \cite{CC22}. For the sake of completeness, the argument is recalled in Subsection~\ref{sec:reducible-CLT} below.}
\end{remark}

We refer to Section~\ref{sec:examples}, see in particular Subsections~\ref{sec:re} and~\ref{subsec:examples}, for a discussion of some explicit but non-trivial examples of reducible and irreducible sequences.
Before proceeding, let us quickly compare the cases with constant vs.\ non-constant weights.

\begin{remark}\label{r:reducible}{\rm 
\begin{enumerate}[label=(\arabic*)]

\item \label{item:point1} ({\it Constant weights}) In the special case of \emph{constant coefficients} $q_n \equiv 1$, or more generally constant weights $q_n^2 \equiv 1$, the conditions for irreducibility are simpler to state. Indeed,
by virtue of \eqref{eq:sigmaBdef2}, Properties~\ref{redCLTcond1}, \ref{redCLTcond2} and~\ref{redCLTcond3} of Definition~\ref{def:irreducibility} can be equivalently expressed in terms of $\{E_n\}$ as follows: as $n\to \infty$,
\begin{enumerate}[label=(\alph*)]
    \item \label{redCLTconda} $m_n \to \infty \,;$
    \item \label{redCLTcondb} $\sum_{i=1}^{m_n} \big| \, E_n \cap (B_i \times \cdots \times B_i)\, \big| \sim |E_n|\,;$
    \item \label{redCLTcondc} $  \max_{i=1,\ldots,m_n} \big| \, E_n \cap (B_i \times \cdots \times B_i)\, \big|  = o\big(|E_n|\big)\,.$ 
\end{enumerate}
By extension, if a sequence $\{E_n\}$ verifies Properties \ref{redCLTconda}---\ref{redCLTcondc}, we say that $\{E_n\}$ is {\em reducible}.

\item ({\it From constant to non-constant weights}) Fix $d\geq 2$, set $q_n\equiv 1$, and consider a sequence $\{E_n\}$ verifying Properties (\ref{item:VE1})---\eqref{item:VE3}, stated at the beginning of Section \ref{ss:prelim}. Consider symmetric non-diagonal mappings $q_n : E_n \to \R : (v_1,...,v_d) \mapsto q_n(v_1,..., v_d)$ such that, for some $0<\epsilon < \eta <\infty$, one has that
$$
\epsilon \leq  \big| q_n(v_1,...,v_d)\big| \leq \eta, \quad (v_1,...,v_d)\in E_n, \,\, n\geq 1.
$$
Then, it is easily seen that the sequence $\{E_n\}$ verifies Properties \ref{redCLTconda}---\ref{redCLTcondc} of Point \ref{item:point1} if and only if the sequence $\{Z_n\}$ defined in \eqref{eq:defpolchaos} is irreducible in the sense of Definition \ref{def:irreducibility}. In view of this transfer principle, and despite a moderate loss of generality, we have chosen to center our analysis on the constant coefficient case throughout most of the paper—with the notable exception of Section~\ref{sec:Hypergraphs}.

\end{enumerate}
}
\end{remark}

In line with our initial objective, this paper aims to characterize those sequences $\{Z_n\}$ as in \eqref{eq:defpolchaos} that satisfy the CLT \eqref{eq:CLT}, but for which a domain-wise reduction is not possible. Our main contributions, presented in the forthcoming Subsections~\ref{subsec:irredspectra} and~\ref{subsec:sparsity}, revolve around two distinct sets of techniques:

\begin{enumerate}[label=(\Roman*)]

\item \underline{\it Spectral graph theory}. In the next Subsection~\ref{subsec:irredspectra}, we will identify homogeneous sums associated with (hyper)graphs (possibly weighted) and use notable spectral estimates associated with their {\em Laplacian} --- generally known as {\it Cheeger-type inequalities} \cite{Che71} --- in order to \emph{characterize their irreducibility}. The proof of these results are presented in Sections~\ref{sec:GRAPHS} and~\ref{sec:Hypergraphs}. Our main references for spectral graph theory and Cheeger inequalities are the outstanding lecture notes by Luca Trevisan \cite{trevisanLN}, as well as the classical texts \cite{BrHaBook, ChungSpectralGraph, GodsilRoy}; see also \cite{aroraFlowEmbedding, LuboSurvey, hlw, LGT14, Vempala12}. The main tools allowing us to gauge the connectivity of hypergraphs are substantially inspired by the theory developed in \cite{B21, SSPHypergraphs}.

\smallskip

\item \underline{\it Combinatorial dimensions and sparsity}. In Subsection~\ref{subsec:sparsity},
we we will exploit the notion of \emph{combinatorial dimension} (following the framework developed in \cite{bleibook, bleijanson}, see Definition~\ref{d:cdm}) in order to derive a \emph{general irreducibility} criterion for homogeneous sums whose support sets $\{E_n\}$ exhibit a \emph{sparse structure}. The proof is presented in Section~\ref{sec:sparsity}. These constructions build on ideas from \cite{blei1979,BleiSRandom, Blei2011Survey, BleiPeresScchmerl, DPKPTRF, NPR10}, and represent a departure from the spectral approaches. In the case $d=2$, where no canonical notion of fractional product exists \cite{BleiSRandom, BleiPeresScchmerl}, we study in Section~\ref{sec:d=2} an explicit sparse construction yielding irreducible homogeneous sums (Theorem~\ref{th:irreducibilityk=2}) that verify a CLT. 

\bigskip
    
\end{enumerate}

}

\needspace{3\baselineskip}
\subsection{Main results: spectral conditions}
\label{subsec:irredspectra}

We present our main results linking irreducibility to spectral properties. For ease of exposition, we first consider the simpler setting of \emph{graphs}, then we discuss the case of \emph{hypergraphs}.

{
\subsubsection{Graphs}\label{ss:graphsresults} Let us consider the case of homogeneous sums of order $d=2$ with $q_n \equiv 1$, in such a way that $Z_n$ in \eqref{eq:defpolchaos} boils down to a quadratic form with $0$-$1$ coefficients: 
\begin{equation}\label{eq:bilformsubsec}
    Z_n = \sum_{v_1,v_2 \in V_n} \ind_{E_n}(v_1,v_2)\, X_{v_1} X_{v_2}\,,
\end{equation}
where $\{E_n\}$ satisfy the requirements \eqref{item:VE1}--\eqref{item:VE3}, as stated at the beginning of Section \ref{ss:prelim}. 

\smallskip

We associate with $Z_n$ in \eqref{eq:bilformsubsec} the finite undirected graph $G_n = (V_n, \mathcal{E}_n)$ with adjacency matrix $A_n(v,w)\coloneq \ind_{E_n}(v,w)$, that is $\mathcal{E}_n$ consists of unordered pairs $\{v,w\}$ with $(v,w)\in E_n$. Letting $\{v_1,...,v_{N_n}\}$ be an enumeration of $V_n$, we let $D_n$ be the diagonal matrix with entries $d(v_1),\ldots,d(v_{N_n})$, i.e.\ the degrees of each vertex. We also consider the {\em normalized Laplacian} matrix $\mathcal{L}_n\coloneq I_n - D_n^{-1/2}A_nD_n^{-1/2}$ associated with $G_n$, whose eigenvalues are noted $0 = \mu_1^{(n)}\le \cdots \le \mu_{N_n}^{(n)}\leq 2$. 

\smallskip

\begin{convention}\label{conv;simple}\rm
Unless otherwise specified, every graph considered in the present and forthcoming sections is implicitly assumed to be undirected, simple (i.e.\ containing neither loops nor multiple edges), and with no isolated vertices.
\end{convention}

\smallskip 

 \begin{definition}\label{def:irreducibilegraphs}{\rm Let $G_n = (V_n, \mathcal{E}_n)$, $n\geq 1$, be a sequence of graphs such that $|V_n|, |\cE_n|\to \infty$, and write $A_n$ for the adjacency matrix of $G_n$. Consider the sequence $\{Z_n\}$ defined as in \eqref{eq:bilformsubsec} by identifying $A_n$ with the indicator $\ind_{E_n}$ of a symmetric and non-diagonal set $E_n \subset V_n\times V_n$. We refer to $\{Z_n\}$ as the {\em sequence of homogeneous sums generated by} $\{G_n\}$. We say that $\{G_n\}$ {\em generates an irreducible CLT}, if $\{Z_n\}$ verifies an irreducible CLT in the sense of Definition \ref{def:irreducibility} --- that is, if $\{Z_n\}$ is irreducible and 
 \begin{equation}\label{e:normadue}
\tilde{Z}_n:= \frac{Z_n}{\sqrt{2|E_n|}} \xlongrightarrow{d} \, N\sim  \cN \big(0,1).
\end{equation}
By extension, if the sequence $\{Z_n\}$ is reducible, we will say that $\{G_n\}$ {\em is reducible} or, more precisely, that $\{G_n\}$ {\em generates a reducible CLT}. Similar notions for hypergraphs and associated sequences of homogeneous sums will be introduced in Definition \ref{def:irreducibileHG}.
 }
\end{definition}
}

As discussed in Section~\ref{sec:GRAPHS}, it is a classical fact that, for $k \geq 2$, the eigenvalue $\mu_k^{(n)}$ is closely related to the connectivity properties of the graph $G_n$ generating $Z_n$, notably through (multiway) Cheeger-type inequalities \cite{trevisanLN, LGT14, Vempala12}. Our main result in this setting, proved in Section~\ref{sec:GRAPHS}, shows that these connectivity properties can be naturally related to the notions of reducibility and irreducibility introduced in Definition~\ref{def:irreducibility}.

\begin{theorem}\label{th:mainSPECTRUMgraph}
    Let $Z_n$ be as in \eqref{eq:bilformsubsec}, and let $G_n=(V_n,\mathcal{E}_n)$ be its associated graph. Suppose that there exists $k\ge 2$ such that, as $n\to \infty$,
\begin{equation}\label{e:mucca}
\liminf_n \mu_k^{(n)} >0.
\end{equation}
Then $\{Z_n\}$ is irreducible, in the sense of Definition \ref{def:irreducibility}. 
\end{theorem}

A large collection of irreducible CLTs for quadratic forms of the type \eqref{eq:bilformsubsec} --- spanning in particular sequences $\{Z_n\}$ generated by {\it expanders} and by generic Cartesian products of regular graphs --- is discussed in Subsection~\ref{subsec:examples}.

\begin{remark}[Full irreducibility]
{\rm Theorem~\ref{th:mainSPECTRUMgraph} implies a stronger form of irreducibility, that we call \emph{full irreducibility}, see Subsection~\ref{ss:full}.}
\end{remark}

\smallskip

\subsubsection{Hypergraphs} 
We next consider generic homogeneous sums of order $d \ge 2$, as defined in \eqref{eq:defpolchaos} (where the real coefficients $q_n(v_1,\ldots,v_d)$ are arbitrary). In Section~\ref{sec:Hypergraphs}, we will show that the content of Theorem \ref{th:mainSPECTRUMgraph} can be extended to this general setting. To this end, we associate with each $Z_n$ a finite \emph{weighted hypergraph} $\mathcal{G}_n=(V_n, \mathcal{E}_n,w_n)$. Here, each \emph{hyperedge} $ e = \{ v_1,\ldots, v_d \}\in \mathcal{E}_n$ is a subset of the vertices $V_n$ having cardinality $d$. Each hyperedge is also paired with the weight $w_n(v_1,\ldots,v_d) \coloneq q_n(v_1,\ldots,v_d)^2$ introduced above. In this context, it is therefore natural to introduce a general definition of the normalized Laplacian $\mathcal{L}_n$ associated with $\mathcal{G}_n$ and obtain an analogous version of Theorem \ref{th:mainSPECTRUMgraph} (see the discussion contained in Section \ref{sec:Hypergraphs}, as well as Theorem \ref{th:mainresulthyper} therein).

We refer to Section~\ref{sec:Hypergraphs} for a precise statement of our results.
As anticipated, we rely on an extension of Cheeger-type estimates for graphs, directly inspired by \cite{B21, SSPHypergraphs}.

\subsection{Main results: sparsity conditions}\label{subsec:sparsity}

We next present our main results which link irreducibility to a notion of sparsity, encoded by the concept of \emph{combinatorial dimension}.

\subsubsection{General results}
{
The following definition --- which was originally motivated by problems in harmonic analysis \cite{bleiharmonic2, bleiHarmonic1} --- is lifted from \cite{blei1979, bleibook, bleijanson} and is meant to characterize those subsets of Cartesian products displaying a certain form of \emph{sparsity}, resulting in the fact that they behave, in a precise sense, like Cartesian products of lower (and possibly non-integer) order.

\begin{definition}[Combinatorial dimension]\label{d:cdm} {\rm Let $\{V_n\}_{n\ge 1}$ be sets with $|V_n| \to \infty$. Fix an integer $d\geq 2$, as well as a real number $\alpha\in (0,d]$. We consider subsets $$ J_n\subseteq \underbrace{ \, V_n\times \cdots \times V_n \, }_\text{$d$ times}, \quad n\geq 1,$$ (each $J_n$ is not necessarily symmetric and has possibly diagonal components) and say that the sequence $\{J_n\}$ has \emph{combinatorial dimension} equal to $\alpha$ if there exist finite constants $0<c'\leq c$ such that the following two properties are verified for $n$ sufficiently large: 
\begin{itemize}
    \item for all $A_1,...,A_d\subseteq V_n$,
\begin{equation}\label{e:req1}
| J_n \cap (A_1\times \cdots \times A_d)| \leq c \, \max_{i=1,...,d} |A_i|^\alpha\,;
\end{equation}
\item moreover, 
\begin{equation}\label{e:req2}
| J_n | \geq c' \, |V_n |^\alpha\,, \quad n\geq 1.
\end{equation}

\end{itemize}
}
\end{definition}

\begin{remark}{\rm 
\begin{enumerate}[label=(\arabic*)]

\item For $d\geq 2$, the existence of sequences $\{J_n\}$ with arbitrary combinatorial dimension $\alpha \in [1,d]$ is established in \cite{BleiPeresScchmerl, BleiSRandom} through a random construction (see Subsection~\ref{sss:randomconstructions} for details).

\item If $\{J_n\}$ has combinatorial dimension $\alpha>0$, then as $n\to\infty$
\begin{equation}\label{e:alphasymp}
|J_n|\asymp |V_n|^{\alpha}, \qquad\text{i.e.}\qquad
c_1 \, |V_n|^{\alpha} \le |J_n| \le c_2 \, |V_n|^{\alpha} 
\end{equation}
for some $0 < c_1 \le c_2 < \infty$.
In particular, when $\alpha <d$, then $|J_n| = o(|V_n|^d)$; more generally, for every $\{A^{(n)}\}$ such that $A^{(n)}\subset V_n $ and $|A^{(n)}|\to \infty$, one has that
$$
    | J_n  \cap (A^{(n)}\times \cdots \times A^{(n)})|\le c\, |A^{(n)}|^{\alpha} = o\big(|A^{(n)}|^d\big).
$$
Such an estimate aligns with our heuristic characterisation of non-trivial combinatorial dimensions as signatures of sparsity.

\item The combinatorial dimension of a given sequence of sets $\{J_n\}$ is not always well-defined, even when the cardinality $|J_n|$ grows as some power of $|V_n|$. To see this, consider the case where $J_n = W_n^d$, with $W_n\subset V_n$ such that $|W_n|\asymp |V_n|^\beta$ for some $\beta<1$. Then $|J_n|= |J_n \cap (W_n \times \cdots \times W_n)| = |W_n|^d \asymp |V_n|^{\beta d}$ and these relations are not compatible with both requirements \eqref{e:req1} and \eqref{e:req2} (which would entail respectively $d \le \alpha$ and $\alpha \le \beta d$).

\end{enumerate}
}   
\end{remark}

Let us provide some basic examples.

\begin{example}\label{ex:fractional}{\rm
\begin{enumerate}[label=(\alph*)]
\item ({\it Full Cartesian products}) The sequence $J_n = V_n\times\cdots \times V_n$, $\,n\geq 1$, has a trivial full combinatorial dimension equal to $d$, since conditions
\eqref{e:req1} and \eqref{e:req2} are satisfied with $c=c' = 1$. 

\item ({\it Sets with diagonal restrictions}) Fix $d\geq 2$ and consider a partition $\pi = \{b_1,...,b_{|\pi|}\}$ of $[d]$, where $|\pi|\in [d]$ is the number of blocks of $\pi$. For every $n\geq 1$, we define $V_n^\pi $ to be the subset of the $d$-th Cartesian product of $V_n$ composed of those vectors $(v_1,...,v_d)$ such that $v_i = v_j$ if and only if $i$ and $j$ belong to the same block $b$ of $\pi$. One easily sees that the sequence $\{V_n^\pi\}$ has combinatorial dimension $|\pi|$. Two straightforward cases correspond to the choices $\pi = \{[d]\}$ (fully diagonal set, in which case the combinatorial dimension is 1) and $\pi = \{\{1\},..., \{d\}\}$ (fully non-diagonal set, in which case the combinatorial dimension is $d$). It is readily seen that, when $|\pi| < d$, the sequence $\{V_n^\pi\}$ {\em cannot} satisfy both Properties \eqref{item:VE2} (symmetry) and \eqref{item:VE3} (non-diagonal structure) from the beginning of Section~\ref{ss:prelim}. This prevents the use of such sets in constructing examples of irreducible CLTs.

\item \label{item:regulargraphs} ({\it Regular graphs}) For $d=2$, we consider a sequence $\{E_n\}$ satisfying the properties \eqref{item:VE1}--\eqref{item:VE3} listed at the beginning of Section \ref{ss:prelim}. We also fix $m\geq 1$, and assume that the mappings $(v,w)\mapsto \ind_{E_n}(v,w)$, $v,w\in V_n$, correspond to the adjacency matrices of a sequence of $m$-regular graphs. Then, one has that $|E_n| = m |V_n|$, and a direct application of the {\em expander mixing lemma} (see e.g. \cite[Lemma 2.5]{hlw}) yields that, for all $n\geq 1$ and all $A_1,A_2\subseteq V_n$,
$$
|E_n\cap (A_1\times A_2)|\leq O(m) \max\{|A_1|, |A_2|\},
$$
implying that the sequence $\{E_n\}$ has combinatorial dimension 1.

\end{enumerate}

}
\end{example}

\medskip

The following result provides a sufficient condition for the irreducibility of a symmetric and non-diagonal sequence $\{E_n\}$, formulated in terms of its combinatorial dimension. The second part of the statement also establishes the existence of infinitely many non-trivial sequences $\{Z_n\}$ of homogeneous sums of order $d \geq 3$ that satisfy an irreducible CLT. (The case $d=2$ will be studied in Theorem~\ref{t:probamethod} through the use of a random construction.)

\begin{theorem}
    \label{th:irreducombdim}
    For $d \ge 2$, let $\{E_n\}$ satisfy Properties \eqref{item:VE1}---\eqref{item:VE3}, as listed at the beginning of Section \ref{ss:prelim}. 
    \begin{enumerate}[label={\rm (\alph*)}]
    \item \label{item:partA} Assume that $\{E_n\}$ has combinatorial dimension 
    $$1< \alpha \leq  d\,.$$
    Then, $\{E_n\}$ is irreducible --- recall Remark~\ref{r:reducible}-{\rm \ref{item:point1}}. 
    \item \label{item:partB} For every $d \ge 3$ and $b=2,\ldots,d-1$, there exists a sequence $\{E_n\}$ with combinatorial dimension $$\alpha=\frac{d}{b} \in (\,1,d\,)\,,$$ such that the corresponding sequence $\{Z_n\}$ of homogeneous sums, defined as in \eqref{eq:defpolchaos} for $q_n\equiv 1$, satisfies an irreducible CLT.

    \end{enumerate}
\end{theorem}

We will prove Theorem \ref{th:irreducombdim} in Section \ref{sec:sparsity}.
The construction of the irreducible sequences $\{E_n\}$ featured in the second part of the statement is detailed in Example~\ref{e:fcp}: for this, we will take $V=\N^{b}$, $V_n=\{1,\ldots,n\}^b$, and $E_n$ with a structure close to that of a \emph{fractional Cartesian product}, see \cite{blei1979, bleibook, Blei2011Survey, bleijanson, DKP22, NPR10}. {The exact computation of the combinatorial dimension of $\{E_n\}$ exploits a discrete Brascamp-Lieb-type inequality due to Fissler \cite{F92}}, while the asymptotic normality of the  homogeneous sums follows from arguments already rehearsed in \cite{NPR10}.

\subsubsection{An ad-hoc construction in dimension 2}\label{sss:adhoc} It is a well-known fact (discussed e.g. in \cite{BleiSRandom, BleiPeresScchmerl}; see also Subsection~\ref{sss:randomconstructions} below) that there is no canonical notion of fractional Cartesian products in dimension $d=2$. This explains why the second part of Theorem~\ref{th:irreducombdim} does not address the two-dimensional case.

To compensate for this limitation (and to probe the sharpness of our results), we study a class of irreducible sequences of quadratic forms ${Z_n}$ of the type \eqref{eq:bilformsubsec}, constructed to satisfy the following properties: (i) the combinatorial dimension of the corresponding sequence $\{E_n\}$ is not defined (see Remark \ref{r:nice1}); (ii) the spectra of the associated normalized Laplacians appear to be analytically intractable in their most general form.

\medskip 

\begin{figure}[!t]
  \centering
  \begin{subfigure}[b]{0.3\textwidth}
    \includegraphics[width=\textwidth]{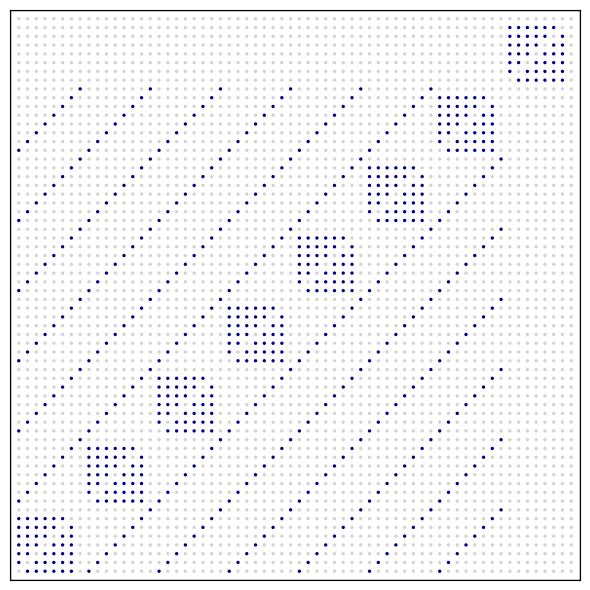}
    \caption{$n=8$} 
  \end{subfigure}
  \hfill
  \begin{subfigure}[b]{0.3\textwidth}
    \includegraphics[width=\textwidth]{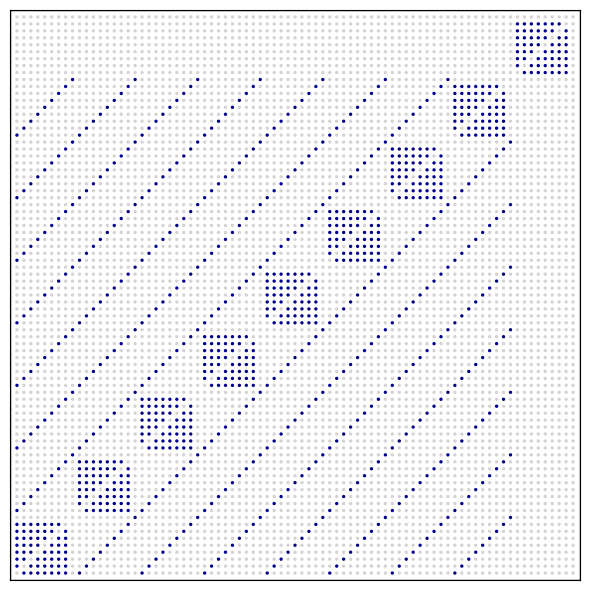}
    \caption{$n=9$} 
  \end{subfigure}
  \hfill
  \begin{subfigure}[b]{0.3\textwidth}
    \includegraphics[width=\textwidth]{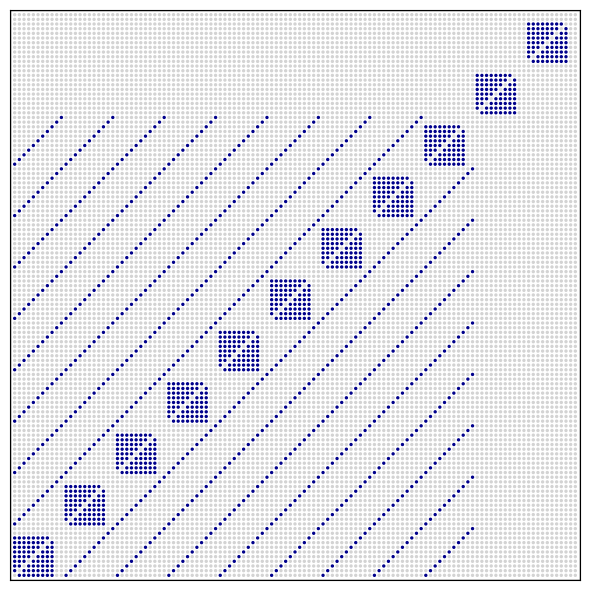}
    \caption{$n=11$} 
  \end{subfigure}
  \caption{A visualization of the set $E_n$ defined in \eqref{eq:E_ngriglia} for $\beta = 0.9$ and different values of $n$, when $\bbS_{\mathsf{v}}(a) = \{a\}\times \{1,..., \lfloor \beta n \rfloor \}$ and $\bbS_{\mathsf{h}}(b) = \{1,..., \lfloor \beta n \rfloor \} \times \{b\}$. Each element of $V_n = \{(i,j) : i,j\in [n]\}$ has been identified with an element of $[n^2]$ by ordering $V_n$ according to the lexicographic order. Pairs $(v_1,v_2)$ not belonging to $E_n$ are represented as grey dots. The fact that $\{E_n\}$ is irreducible implies that one cannot discard the diagonal bands framing the central blue tiles without drastically modifying the asymptotic behavior of $\{Z_n\}$.}
  \label{f:noncoveringblueboxes}
\end{figure}

Our explicit construction is realized as follows. For any $n \in \N$, we set
\begin{equation*}
    V_n = \lbrace 1,\ldots,n \rbrace \times \lbrace 1,\ldots,n \rbrace = [n]^2 \,.
\end{equation*}
Fix $\beta \in (0,1]$ and, for each $a,b \in \{1,\ldots, n\}$, consider subsets $\bbS_{\mathsf{v}}(a)
\subseteq \{a\} \times 
\{1,\ldots, n\}$
and $\bbS_{\mathsf{h}}(b) \subseteq \{1,\ldots, n\} \times \{b\}$
(the labels $\mathsf{v}, \mathsf{h}$ stand for \emph{vertical} and \emph{horizontal})
with
\begin{equation}\label{eq:sizebeta2}
    |\bbS_{\mathsf{v}}(a)| = |\bbS_{\mathsf{h}}(b)| = \lfloor \beta n \rfloor \,.
\end{equation}
\smallskip
We introduce the following equivalence relations on $V_n$: for any $v_1,v_2 \in V_n$,
\begin{align}
&	v_1 \stackrel{\mathsf{h}}{\sim} v_2 \quad  \iff \quad 
 \text{for some $b$ one has } \ v_1,v_2 \in \bbS_{\mathsf{h}}(b) \ \text{ with } \ v_1 \ne v_2 \,, \label{eq:relationH++2}\\
	&v_1 \stackrel{\mathsf{v}}{\sim} v_2  \quad\iff \quad \text{for some $a$ one has} \ v_1,v_2 \in \bbS_{\mathsf{v}}(a) \ \text{ with } \ v_1 \ne v_2 \,.\notag
\end{align}
We define the sequence $\{Z_n\}$ according to \eqref{eq:bilformsubsec}, where $E_n$ has the form
\begin{equation}
    \label{eq:E_ngriglia}
    E_n =\big\{ \, (v_1,v_2) \in V_n \times V_n \, : \text{ either }v_1\stackrel{\mathsf{h}}{\sim}v_2 \text{ or }v_1\stackrel{\mathsf{v}}{\sim}v_2 \, \big\}.
\end{equation}
In words, $v_1\stackrel{\mathsf{h}}{\sim}v_2$ means that $v_1, v_2$ are on the same horizontal line within a set $\mathbb{S}_h(\cdot)$, and similarly for $v_1\stackrel{\mathsf{v}}{\sim}v_2$. We refer to Figure~\ref{f:noncoveringblueboxes} for some graphical representations of~$E_n$.

\begin{remark}[$\{E_n\}$ has no combinatorial dimension]{\rm As discussed in Section \ref{sec:d=2}, one has that $|E_n| =2n \binom{\lfloor \beta n \rfloor}{2} \sim \beta^2 n^3 \asymp |V_n|^{3/2}$ as $n\to \infty$. As anticipated, for every $\beta\in (0,1]$ the sequence $\{E_n\}$ has no definite combinatorial dimension, because for every $n$ and $a \in [n]$
$$
|E_n\cap (\bbS_{\mathsf{v}}(a)\times \bbS_{\mathsf{v}}(a))| = |\bbS_{\mathsf{v}}(a)| \cdot \big(|\bbS_{\mathsf{v}}(a)|-1\big)
\sim |\bbS_{\mathsf{v}}(a)|^2 
$$
hence the two relations \eqref{e:req1} and \eqref{e:req2} cannot be fulfilled for the same~$\alpha$.
}   \label{r:nice1}
\end{remark}

The following result proves that, for $\beta> \frac12$, the sequence $\{Z_n\}$ is irreducibile, in fact it verifies an irreducible CLT.

\begin{theorem}\label{th:irreducibilityk=2}
    Let the above notation and conventions prevail, and fix $$ \frac{1}{2} < \beta \le 1\,.$$ 
    Then, $\{Z_n\}$ verifies an irreducible CLT, in the sense of Definition \ref{def:irreducibility}.
\end{theorem}

We prove this result in Section \ref{sec:d=2}, where we actually obtain a strengthened version of Theorem~\ref{th:irreducibilityk=2} (see, in particular, Theorem \ref{th:mainth2d}). 
We refer to Subsection~\ref{subsec:examples} --- see in particular Example~\ref{item:cartprod(e)} --- for a discussion of the irreducibility of $\{Z_n\}$ in some special cases.

 }

\subsubsection{Random constructions}
\label{sss:randomconstructions}

{
We now present a random construction from \cite{BleiSRandom, BleiPeresScchmerl}, yielding for every $d\geq 2$ the existence of a sequence $\{E_n\}$ with arbitrary combinatorial dimension $\alpha\in (1,d)$. To state this result, we fix $d$ and $\alpha$ as above and, for every $n\geq 1$, we consider an array 
\begin{equation}\label{e:array}
H^{(n)}  = \left\{    H^{(n)}_{j_1,...,j_d} : (j_1,...,j_d)\in [n]^d\right\}
\end{equation}
of i.i.d. Bernoulli random variables with parameter $n^{\alpha-d}$. We define the random sets
\begin{equation}\label{e:Hset}
{\bf E}_n(\alpha; d) := \left\{ (j_1,...,j_d)\in [n]^d : H^{(n)}_{j_1,...,j_d}=1\right\}, \quad n\geq 1,
\end{equation}
and, for every pair of constants $0<\epsilon <\eta$, we define the event $C_n(\epsilon,\eta; \alpha, d)$ to be the set of all $\omega\in \Omega$ such that 
$|{\bf E}_n(\alpha; d)(\omega)|\geq \epsilon\, n^{ \alpha}$ and $|{\bf E}_n(\alpha; d)(\omega)\cap(A_1\times \cdots \times A_d)|\leq \eta\, \max_i |A_i|^\alpha$, for all $A_1,...,A_d\subseteq [n]$.

\begin{figure}
  \centering
  \begin{subfigure}[b]{0.3\textwidth}
    \includegraphics[width=\textwidth]{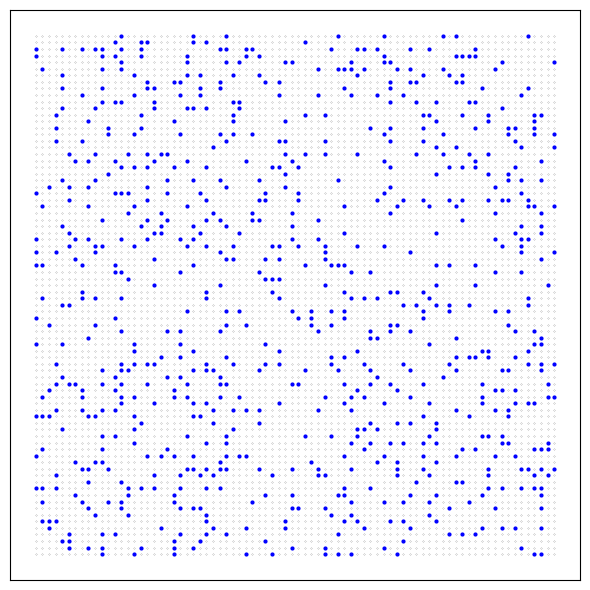}
    \caption{$\alpha=1.1$} 
  \end{subfigure}
  \hfill
  \begin{subfigure}[b]{0.3\textwidth}
    \includegraphics[width=\textwidth]{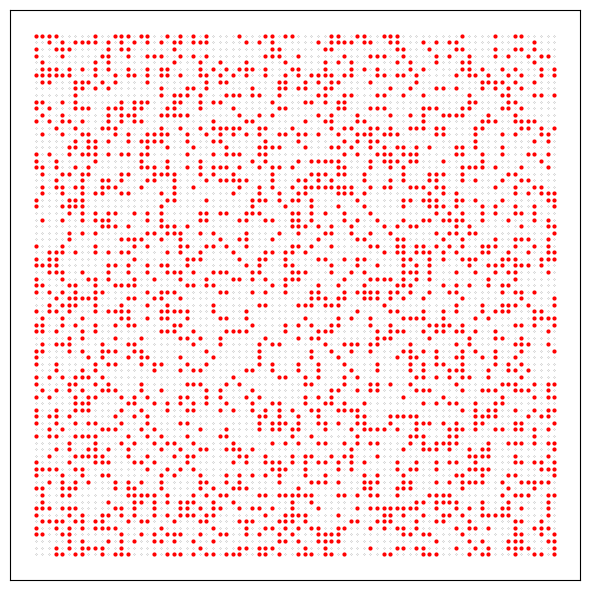}
    \caption{$\alpha=1.7$} 
  \end{subfigure}
  \hfill
  \begin{subfigure}[b]{0.3\textwidth}
    \includegraphics[width=\textwidth]{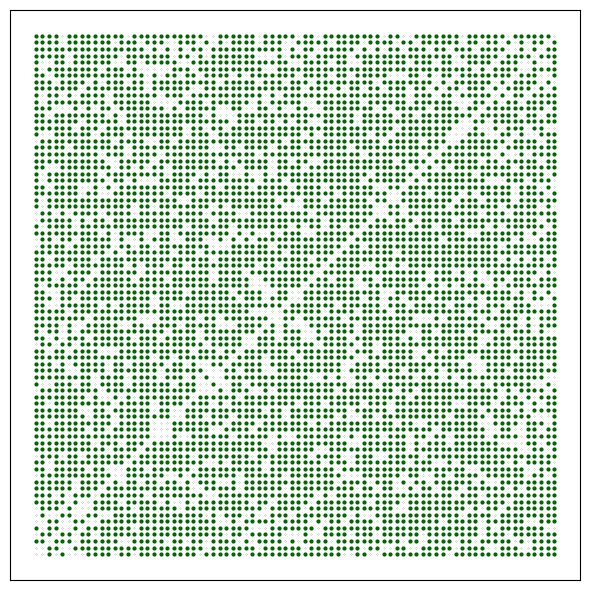}
    \caption{$\alpha=1.9$} 
  \end{subfigure}
  \caption{\footnotesize Three realisations of the (symmetric and non-diagonal) random point configuration $\widetilde{\bf E}_n(\alpha;2)$, as defined in Section \ref{sss:randomconstructions}, for $n=80$ and different values of $\alpha\in (1,2)$.}
  \label{f:randomplots}
\end{figure}

\begin{theorem}[Lemma 2 in  \cite{BleiPeresScchmerl}]\label{t:peres} Let the above notation and assumptions prevail. Then, for every $d\geq 2$ and every $\alpha\in (1,d)$ there exist constants $0<\epsilon_0 <\eta_0$ (depending on $d, \alpha$) such that
$$
\lim_{n\to\infty}\mathbb{P}\left\{ C_n(\epsilon_0,\eta_0; \alpha, d)\right\} = 1.
$$
    
\end{theorem}

It is easily seen that, for $d=2$, the conclusion of Theorem \ref{t:peres} continues to hold if one replaces the set ${\bf E}_n(\alpha;2)$ with the symmetrized and non-diagonal set $\widetilde{\bf E}_n(\alpha;2)$, obtained by modifying \eqref{e:Hset} as follows: (a) define $H^{(n)}_{j_1,j_1} = 0$ for all $j_1\in [n]$, and (b) force symmetry by setting $H^{(n)}_{j_1,j_2} = H^{(n)}_{j_2,j_1} $, for all $1\leq j_1<j_2\leq n$. In this way, the array $$\left\{{\ind}_{\widetilde{\bf E}_n(\alpha;2)}(i,j)\right\} = \left\{H^{(n)}_{i,j}\right\} $$ coincides with the (random) adjacency matrix of the Erd\"os-Renyi random graph $G(n,p_n)$, with $p_n = n^{\alpha-2}$; see e.g. \cite{bollobas}. Figure \ref{f:randomplots} displays several realizations of the random set $\widetilde{\bf E}_n(\alpha;2)$, for $n = 80$ and different values of $\alpha$. 

\smallskip 

The next statement complements Theorem \ref{th:irreducombdim}, by showing the existence of irreducible CLTs for homogeneous sums of order $2$, that are associated with sequences $\{E_n\}$ of arbitrary combinatorial dimension $\alpha\in (1,2)$. The (simple) proof is presented in Section \ref{ss:proofproba}, and uses the probabilistic method by exploiting the content of Theorem \ref{t:peres} in the case $d=2$. More refined results (immaterial for our statement) can be obtained by following the strategy outlined in \cite[Example 2.3]{BDMM22}, exploiting the classical estimates from \cite{largesudakov}.

\begin{theorem}\label{t:probamethod} Fix $d=2$ and $\alpha\in (1,2)$. Then, there exists a sequence $\{E_n\}$ such that: {\rm (i)} $\{E_n\}$ satisfies Properties \eqref{item:VE1}---\eqref{item:VE3}, as listed at the beginning of Section \ref{ss:prelim}, {\rm (ii)} $\{E_n\}$ has combinatorial dimension $\alpha$; {\rm (iii)} the associated sequence of homogeneous sums $\{Z_n\}$, defined as in \eqref{eq:bilformsubsec}, satisfies the irreducible CLT \eqref{e:normadue}.
    
\end{theorem}

\smallskip
 
The next question is left open for further research, and could be in principle attacked by refining the study of the sets ${\bf E}(\alpha;d)$ introduced in the present section.

\begin{enumerate}[label=\textbf{(Q2)}]
\item \label{q2}
\begin{center}
\emph{Fix \( d \geq 3 \) and \( \alpha \in (1,d) \). Is there a sequence \( \{Z_n\} \) as in \eqref{eq:defpolchaos} such that:}
\end{center}
\begin{enumerate}[label=(\alph*)]
  \item \( q_n \equiv 1 \),
  \item \( \{E_n\} \) has combinatorial dimension \( \alpha \),
  \item \( \{\tilde{Z}_n\} \) verifies the (irreducible) CLT \eqref{eq:CLT}?
\end{enumerate}
\end{enumerate}

\medskip 

Note that Theorem \ref{th:irreducombdim}-\ref{item:partB} provides a positive answer to Question \ref{q2} for all $d\geq 3 $ and $\alpha = d/b$, $b=2,...,d-1$.

\medskip

}

\subsection{On necessary conditions for irreducibility}\label{ss:further}

The main contribution of this paper is the identification of several sufficient conditions for irreducibility, stemming from both spectral (Theorem~\ref{th:mainSPECTRUMgraph} and its extensions to hypergraphs) and combinatorial (Theorems~\ref{th:irreducombdim} and \ref{th:irreducibilityk=2}) criteria. By contrast, we have not been able to establish corresponding necessary conditions. This limitation is briefly discussed in the following list.

\begin{enumerate}[label=(\roman*)]
\item The spectral condition~\eqref{e:mucca} is \emph{not} necessary for irreducibility,
as shown, for instance, in Example~\ref{item:cartprod(e)} of Section~\ref{subsec:examples}: one can construct sequences of irreducible homogeneous sums whose Laplace eigenvalues vanish at every order. We conjecture that necessary spectral conditions for irreducibility might be obtained by extending ``hard'' Cheeger-type inequalities (such as those presented in the forthcoming Theorem~\ref{t:hardcheeger}) to incorporate constraints on the relative sizes of the blocks $S_1,\ldots,S_k$. Some progress in this direction can be found, for instance, in~\cite{AnjosNeto}.

\item Similarly, Example~\ref{item:cartprod(e)} in Section \ref{subsec:examples} and Example~\ref{ex:fractional} will show, respectively, that (a) irreducible CLTs may also arise in cases where the combinatorial dimension is not well-defined, and (b) having a combinatorial dimension strictly greater than $1$ is not necessary for irreducibility, since there exist sequences of regular graphs with fixed degree that generate both irreducible and reducible CLTs (note that sequences of regular graphs with a fixed degree are always associated with summation domains having a combinatorial dimension equal to one; see Example~\ref{ex:fractional}-\ref{item:regulargraphs}). At the time of writing, we are not aware of any additional combinatorial characterizations of discrete symmetric sets that could help close this gap.

\end{enumerate}

\medskip

\begin{remark}{\rm An alternative approach to studying the CLTs considered in the present paper is based on the use of {\em dependency graphs} and Stein's method; see, for instance, the classical reference~\cite{baldirinott}. In the context of $0$-$1$ quadratic forms as in~\eqref{eq:bilformsubsec}, this would entail addressing questions of reducibility and irreducibility by means of the so-called {\em line graph} associated with $G_n$, that is, the graph whose vertices correspond to the edges of $G_n$ and where two edges are connected if and only if they are adjacent (see e.g.~\cite[Section~1.4.5]{BrHaBook}). While some preliminary computations have shown that this approach may lead to suboptimal results, we leave this direction open for future investigation.

}
\end{remark}

\section{Preliminaries, examples and applications}
\label{sec:examples}

In this section, we first recall some (by now) classical conditions which ensure the validity of the Central Limit Theorem for homogeneous sums. We then provide examples and applications which illustrate our setting and our main results. 

\subsection{Universal fourth moment theorems}
\label{sec:4th}

The following result, based on the material discussed in \cite[Chapter 5 and Chapter 11]{NP12}, corresponds to the {\em universal de Jong theorem} evoked in the Introduction.

\begin{theorem}[Theorem 1.10 in \cite{NPR10}]\label{t:unidejong} Let the assumptions and notation of Section \ref{ss:prelim} prevail. Then, the following three properties are equivalent, as $n\rightarrow \infty$:
\begin{enumerate}[label={\rm (\roman*)}]
\item The CLT \eqref{eq:CLT} holds;
\item \label{item:cond2thm3.1} $\mathbb{E}[\tilde{Z}_n({\bf X})^4]\to 3$;
\item \label{item:point3} For {\em any} collection ${\bf Y} = \{Y_v : v\in V\}$ of i.i.d. random variables with zero mean and unit variance,
$$ \widetilde{Z}_n({\bf Y})  \coloneq \frac{Z_n({\bf Y})}{\sqrt{d!\|q_n\|^2}} \, \xlongrightarrow{d} \, N\sim  \cN \big(0,1), $$
where $Z_n({\bf Y})$ is defined as in \eqref{eq:defpolchaos}.
\item For {\em any} collection ${\bf U} = \{U_v : v\in V\}$ of independent random variables with zero mean, unit variance and such that $\sup_{v\in V} \mathbb{E}[|U_v|^p]<\infty$ for some $p>2$,
$$ \widetilde{Z}_n({\bf U})   \, \xlongrightarrow{d} \, N\sim  \cN \big(0,1), $$
where the notation is analogous to that introduced at Point \ref{item:point3}. 
\end{enumerate}
    
\end{theorem}

\begin{remark}\label{r:afterdejong}{\rm
\begin{enumerate}[label=(\arabic*)]
\item Under the above notation and assumptions, define $\tilde{q}_n$ to be the mapping on $V_n\times \dots \times V_n$ such that $\tilde{q}_n(v_1,...,v_d) = q_n(v_1,...,v_d)/\sqrt{d!\|q_n\|^2}$ if $(v_1,...,v_d)\in E_n$, and $\tilde{q}_n(v_1,...,v_d) = 0$ otherwise. For every $r=1,...,d$, define the {\em contraction}
\begin{eqnarray*}
\tilde{q}_n\star_r \tilde{q}_n &:& \underbrace{ \, V_n\times \cdots \times V_n \, }_\text{$2d-2r$ times}\to \mathbb{R}\\  &: &(v_1,...,v_{2d-2r})\mapsto\!\!\! \sum_{a_1,...,a_r\in V_n} \tilde{q}_n(a_1,...,a_r, v_1,...,v_r)\tilde{q}_n(a_1,...,a_r, v_{r+1},...,v_{2d-2r})\\
&:=& \tilde{q}_n\star_r \tilde{q}_n(v_1,...,v_{2d-2r}).
\end{eqnarray*}
Then, \cite[Proposition 1.6]{NPR10} implies that Condition \ref{item:cond2thm3.1} in Theorem \ref{t:unidejong} is verified if and only if, for all $r=1,...,d-1$,
\begin{equation}\label{e:qn}
\sum_{v_1,...,v_{2d-2r}\in V_n} \tilde{q}_n\star_r \tilde{q}_n(v_1,...,v_{2d-2r})^2\longrightarrow 0.
\end{equation}
\item \label{item:rempoint2} For $d=2$, consider the case of constant coefficients $q_n\equiv 1$ and, for $n\geq 1$, denote by $\lambda_{n,1},...,\lambda_{n, |V_n|}$ the eigenvalues of the matrix $\{A_n(v,u) : u,v\in V_n\}$ such that $A_n(u,v)$ equals one or zero according as $(u,v)\in E_n$ or not (that is, $A_n$ is the {\it adjacency matrix} of the undirected graph with vertices $V_n$ induced by the symmetric set $E_n$). In this case, a direct computation, based on criterion \eqref{e:maxgen}, shows that Condition \ref{item:cond2thm3.1} of Theorem \ref{t:unidejong} holds if and only if, as $n\to \infty$,
\begin{equation}\label{e:degreeo}
\max_{u} |\lambda_{n,u}| = o\big( |E_n|^{1/2}\big).
\end{equation}
Now denote by $\Delta(n)$ the \emph{maximal degree} of the graph associated with the adjacency matrix $A_n$. It is easily seen that, since $\max_u |\lambda_{n,u}| \leq \Delta(n)$,  condition \eqref{e:degreeo} is implied by the stronger relation
\begin{equation}\label{e:maxdd}
\Delta(n) = o\big( |E_n|^{1/2}\big).
\end{equation}
Also, if such a graph is $d_n$-regular, then relation \eqref{e:degreeo} is equivalent to $d_n = o(|V_n|)$. See \cite[Proposition 1.11]{BDMM22} for a comprehensive statement applying to generic collections of i.i.d. random variables, and also \cite[Section 3.1]{Cha08}.

\end{enumerate}

}\end{remark}

\subsection{Reducible sequences verify a CLT}

\label{sec:reducible-CLT}

In order to provide some intuition, we now explain why \emph{reducible sequences satisfy a Lindeberg-Feller-type CLT}, following \cite{CC22}. Assume that the sequence $\{Z_n\}$ is reducible in the sense of Definition \ref{def:irreducibility}, and define 
$$
T_n := \sum_{i=1}^{m_n}\sum_{\substack{(v_1,\ldots,v_d)\in \\  E_n\cap (B_i\times \cdots \times B_i)}  }\!\!q_n(v_1,...,v_d)\prod_{\ell=1}^d X_{v_\ell} :=\sum_{i=1}^{m_n}Z_{n,i}.  
$$
Then, for fixed $n$, one has that the random variables $\{Z_{n,i} : i=1,...,m_n\}$ are stochastically independent and such that ${\bf Var}(Z_{n,i}) = \sigma^2_n(B_i)$. Moreover, Property \ref{redCLTcond2} yields the asymptotic relations (as $n\to \infty$)
\begin{equation}\label{e:thereduction}
{\bf Var}(T_n) = \sum_{i=1}^{m_n} \sigma_n^2(B_i)\sim d!\|q_n\|^2\quad \mbox{and} \quad \mathbb{E}[(\tilde{Z}_n - \tilde{T}_n)^2 ] = 
\mathbb{E}[\tilde{Z}_n^2] - \mathbb{E}[\tilde{T}_n^2 ]\to 0,
\end{equation}
where
$$
\tilde{T}_n :=\frac{1}{\sqrt{d! \|q_n\|^2}}   \sum_{i=1}^{m_n}Z_{n,i} :=\sum_{i=1}^{m_n}\tilde{Z}_{n,i}.
$$
Since every $\tilde{Z}_{n,i}$ is an element of the $d$th Wiener chaos associated with the Gaussian family ${\bf X}$, one can use standard hypercontractive estimates \cite[Corollary 2.8.14]{NP12} to deduce that, for all $p>2$ and for some absolute constant $c_{p,d}$,
\begin{equation}\label{eq:FellerLind}
    \sum_{i=1}^{m_n} \mathbb{E}\big[|\tilde{Z}_{n,i}|^p\big]\leq \frac{c_{p,d}}{(d! \|q_n\|^2)^{p/2}}\sum_{i=1}^{m_n} ( \sigma_n^2(B_i) )^{p/2}
    \le c_{p,d}\, \max_{i=1, \ldots, m_n}
    \sigma_n^2(B_i)^{p/2-1}
    = o(1) \, , \quad n\to \infty,
\end{equation}
where the last relation follows by combining Properties \ref{redCLTcond2} and \ref{redCLTcond3} of Definition \ref{def:irreducibility}. Applying Markov inequality as in \cite[formulae (4.12)---(4.14)]{CC22}, one sees that \eqref{eq:FellerLind} implies that $\tilde{T}_n,\tilde{Z}_n \stackrel{d}{\rightarrow} N\sim \cN (0,1)$ by the Lindeberg-Feller CLT, see e.g.\ \cite[Theorem~4.12]{kallengerg}.

\medskip

The next two sections present several examples of reducible and irreducible sequences, focusing mainly on the case $d=2$, which corresponds to the results of Section~\ref{ss:graphsresults} (except for Example~\ref{e:ridux}\,\eqref{it:ridux-trivial}).  
Throughout Sections~\ref{sec:re}---\ref{subsec:examples}, we consider sequences of simple graphs 
$\{ G_n = (V_n,\mathcal{E}_n) \}$, which may vary from one example to another, and adopt the notation
\[
0 = \mu^{(n)}_1 \leq \cdots \leq \mu^{(n)}_{|V_n|} \leq 2
\]
for the eigenvalues of the normalized Laplacian associated with $G_n$.  
More background can be found in Section \ref{sec:GRAPHS}.

Further illustrations in the case $d\geq 3$ will require some non-trivial notions associated with hypergraphs and their adjacency structure, and will be discussed in the forthcoming Section~\ref{sec:Hypergraphs}, see in particular Examples~\ref{ex:rooklike} and~\ref{e:3uni}.

\subsection{Examples of reducible sequences}
\label{sec:re}

We start by illustrating some explicit examples of \emph{reducible} sequences.

\begin{example}\label{e:ridux}{\rm

\begin{enumerate}
\renewcommand{\theenumi}{\alph{enumi}}
\renewcommand{\labelenumi}{(\theenumi)}

\item\label{it:ridux-trivial}
A trivial example of reducible sequence is obtained by considering the random variables $Z_n:= \sum_{i=0}^{n-1} \prod_{\ell = 1}^d X_{id+\ell}$, $n\geq 1$, appearing in \eqref{eq:exFL}. Indeed, in this case one has that $V_n = \{1,...,dn\} = [dn]$, and one can take $m_n = n$ and 
$$
B_i = B_i(n) = \{ (i-1)d+1,..., id\}, \quad i=1,...,n.
$$

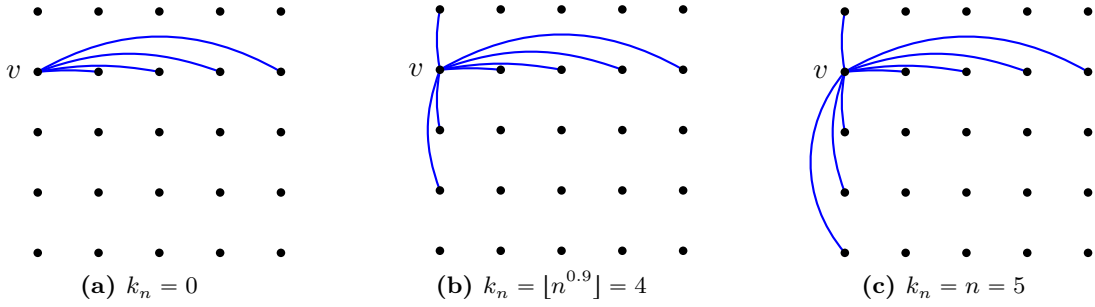
\begin{figure}
  \centering
  \begin{subfigure}[b]{0.3\textwidth}
    \centering
    \begin{tikzpicture}[scale=0.8]
      \draw[blue, thick] (0,3) to[out=5, in=175] (1,3);
      \draw[blue, thick] (0,3) to[out=10, in=170] (2,3);
      \draw[blue, thick] (0,3) to[out=20, in=160] (3,3);
      \draw[blue, thick] (0,3) to[out=30, in=150] (4,3);
      \foreach \x in {0,1,2,3,4}
        \foreach \y in {0,1,2,3,4}
          \fill (\x,\y) circle (2pt);
      \node[left] at (-0.1,3) {$v$};
    \end{tikzpicture}
    \caption{$k_n = 0$}
  \end{subfigure}
  \hfill
  \begin{subfigure}[b]{0.3\textwidth}
    \centering
    \begin{tikzpicture}[scale=0.8]
      \draw[blue, thick] (0,3) to[out=5, in=175] (1,3);
      \draw[blue, thick] (0,3) to[out=10, in=170] (2,3);
      \draw[blue, thick] (0,3) to[out=20, in=160] (3,3);
      \draw[blue, thick] (0,3) to[out=30, in=150] (4,3);
      \draw[blue, thick] (0,3) to[out=260, in=100] (0,2);
      \draw[blue, thick] (0,3) to[out=250, in=110] (0,1);
      \draw[blue, thick] (0,3) to[out=100, in=260] (0,4);
      \foreach \x in {0,1,2,3,4}
        \foreach \y in {0,1,2,3,4}
          \fill (\x,\y) circle (2pt);
      \node[left] at (-0.1,3) {$v$};
    \end{tikzpicture}
    \caption{$k_n = \lfloor n^{0.9}\rfloor = 4$}
  \end{subfigure}
  \hfill
  \begin{subfigure}[b]{0.3\textwidth}
    \centering
    \begin{tikzpicture}[scale=0.8]
      \draw[blue, thick] (0,3) to[out=5, in=175] (1,3);
      \draw[blue, thick] (0,3) to[out=10, in=170] (2,3);
      \draw[blue, thick] (0,3) to[out=20, in=160] (3,3);
      \draw[blue, thick] (0,3) to[out=30, in=150] (4,3);
      \draw[blue, thick] (0,3) to[out=260, in=100] (0,2);
      \draw[blue, thick] (0,3) to[out=250, in=110] (0,1);
      \draw[blue, thick] (0,3) to[out=230, in=130] (0,0);
      \draw[blue, thick] (0,3) to[out=100, in=260] (0,4);
      \foreach \x in {0,1,2,3,4}
        \foreach \y in {0,1,2,3,4}
          \fill (\x,\y) circle (2pt);
      \node[left] at (-0.1,3) {$v$};
    \end{tikzpicture}
    \caption{$k_n = n= 5$}
  \end{subfigure}

  \caption{\footnotesize The edges adjacent to the vertex $v = (2,1)$ in the graph underlying Example \ref{e:ridux}-\eqref{it:ridux-rook}, for $n=5$ and different choices of $k_n$. }
  \label{f:bluespiders}
\end{figure}

\item\label{it:ridux-rook} 
We now build a non-trivial example associated with the sequence $V_n := \{ (i,j) : i,j\in [n]\}$, $n\geq 1$. Let $k_n$ be an integer such that $0\leq k_n\leq n$ and define the set $E_n \subset V_n\times V_n$ as follows: given $v_1 = (a,b)\in V_n,\,  v_2=(i,j) \in V_n$, then $(v_1,v_2) \in E_n$ if and only if:
\begin{itemize}
    \item either $a=i$ with $b\neq j$;
    \item or if $a\neq i$ with $b= j$ and, additionally, $i\leq k_n$. 
\end{itemize}
\smallskip
\emph{Let us henceforth 
assume that $k_n = o(n)$.}
Then, it is an exercise to show that $|E_n|\sim n^3$. Moreover, since $k_n = o(n)$, one has that the family \emph{$\{E_n\}$ is reducible} in the sense of Remark~\ref{r:reducible}-\ref{item:point1}, with $m_n = n$ and 
$$
B_i = B_i(n) = \{(i,\ell) : \ell = 1,...,n\}, \quad i=1,...,n.
$$
As already observed in Remark \ref{r:afterdejong}-\ref{item:rempoint2}, one can identify the function $(v_1,v_2)\mapsto\ind_{E_n} (v_1, v_2)$ with the adjacency matrix $A_n$ of the undirected graph $G_n$ induced by $E_n$ on the vertex set $V_n$. An illustration of such a graph for $n=5$ and different choices of $k_n$ can be found in Fig. \ref{f:bluespiders}. A graphical representation of the reducible set $E_n$ for $n=8,9,11$ and $k_n = \lfloor n^{0.9} \rfloor$ is provided in Fig. \ref{f:blueboxes}: in such a picture, the blue tiles around the diagonal correspond to the disjoint sets $E_n\cap (B_i\times B_i)$, $i=1,...,n$, whereas the collection of all red dots corresponds to the set
$$
E_n\, \backslash\,\bigcup_{i=1}^{m_n} (B_i\times B_i).
$$
The sequence of graphs illustrated in Figure \ref{f:bluespiders}-(a) corresponds to the limiting case $k_n=0$, yielding that $G_n$ is the disjoint union of $n$ copies of $K_n$, that is, $G_n = K_n\sqcup \cdots \sqcup K_n$, $n\geq 1$. Now write $\{\mu_k^{(n)}\}$ for the Since each $G_n$ has exactly $n$ connected components, one has that $0 = \mu_1^{(n)} = \cdots = \mu_n^{(n)}$, $n\geq 1$, so that $\mu_k^{(n)}\to 0$ for every fixed $k$. This behavior is consistent with Theorem \ref{th:mainSPECTRUMgraph}.

\item\label{it:ridux-cube} For $n\geq 2$, the {\em hypercube} $Q_n$ is the $n$-regular graph with vertex set $V_n$ given by all strings of the type $x = (x_1,...,x_n)\in \{0,1\}^n$, and such that $x,y\in V_n$ are connected by an edge ($x\sim y$) if and only if they differ exactly by one coordinate, that is,  $\sum_{i=1}^n |x_i-y_i|=1$. One has that $|V_n| = 2^n$. Also, if one defines 
\begin{equation}\label{f:hyperchedge}
    E_n:= \{ (x,y)\in V_n^2 : x\sim y\}, \quad n\geq 1,
\end{equation}
one has that $|E_n| = n2^n$ (twice the number of edges in $Q_n$). We now show that \emph{the sequence $\{E_n\}$ is reducible} in the sense of Remark \ref{r:reducible}-\ref{item:point1}. To see this, for $h <n$, consider the set $V_h =\{0,1\}^h$ and, for $z \in V_h$, define $$B_n^{(z)} \coloneq \big\{ x \in V_n : x_i=z_i \ \forall i=1,\ldots,h \big\}\,,$$
in such a way that $\big| B_n^{(z)} \big|=2^k$ with $k=n-h$. Choosing sequences $h_n, k_n \to \infty$ as $n \to \infty$ such that  $h_n + k_n = n$, $\frac{k_n}{n}\to 1$ as $n\to\infty$\footnote{Choose for instance $h_n= \lfloor \log n \rfloor$ and $k_n=n-\lfloor \log n \rfloor$.}, one has that the family of partitions $\big\{ B_n^{(z)} : z \in V_{h_n}\big\}$, $n\geq 1$, satisfies the three requirements \ref{redCLTconda}---\ref{redCLTcondc} of Remark \ref{r:reducible}-\ref{item:point1}. Indeed, \ref{redCLTconda} follows since $m_n := | V_{h_n}| = 2^{h_n} \to \infty$. Moreover, for any $z \in V_{h_n}$ one has that $| E_n \cap (B_n^{(z)}\times B_n^{(z)})|= \big| \big\{ (x,y) \in E_n : x,y \in B_n^{(z)} \big\} \big| =k_n 2^{k_n}$, which implies that \ref{redCLTcondb} and \ref{redCLTcondc} are verified because 
\begin{equation*}
    \frac{| E_n \cap (B_n^{(z)}\times B_n^{(z)})|}{|E_n|} =\frac{k_n2^{k_n}}{n2^{n}}=\frac{k_n}{n} \frac{1}{2^{h_n}} \, \underset{n \to \infty}{\longrightarrow}\, 0\,, 
\end{equation*}
uniformly in $z \in V_{h_n}$, and
\begin{equation*}
    \frac{ \sum_{z \in V_{h_n}}| E_n \cap (B_n^{(z)}\times B_n^{(z)})|}{|E_n|} = \frac{2^{h_n}k_n2^{k_n}}{n2^{n}}=\frac{k_n}{n}\, \underset{n \to \infty}{\longrightarrow}\, 1\,.
\end{equation*}
We observe that, classically, $Q_n = K_2^{ {\cpow}  n }$, that is, $Q_n$ is the Cartesian product of $n$ copies of the complete graph over two vertices.
\end{enumerate}
}
\end{example}

\begin{figure}
  \centering
  \begin{subfigure}[b]{0.3\textwidth}
    \includegraphics[width=\textwidth]{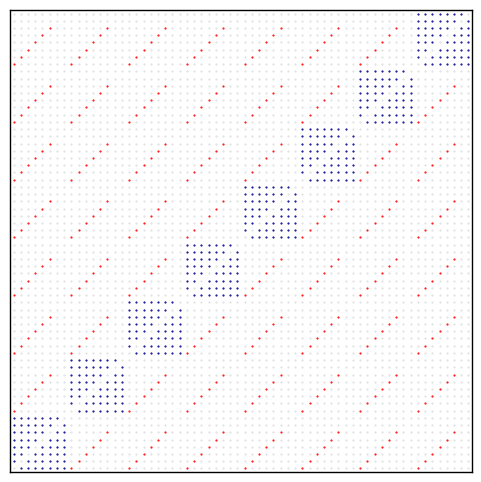}
    \caption{$n=8$}
  \end{subfigure}
  \hfill
  \begin{subfigure}[b]{0.3\textwidth}
    \includegraphics[width=\textwidth]{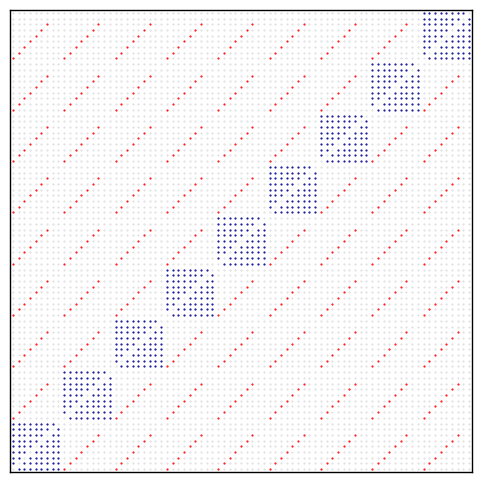}
    \caption{$n=9$} 
  \end{subfigure}
  \hfill
  \begin{subfigure}[b]{0.3\textwidth}
    \includegraphics[width=\textwidth]{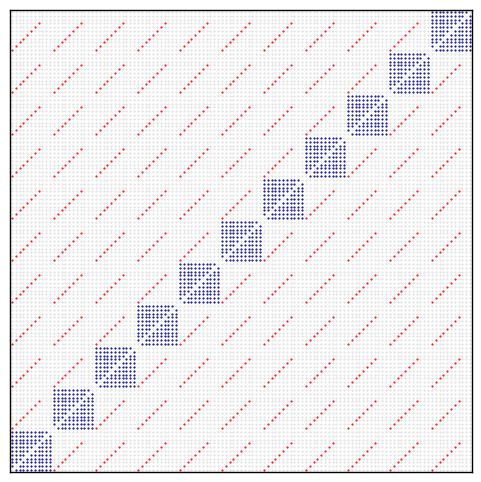}
    \caption{$n=11$} 
  \end{subfigure}
  \caption{\footnotesize The set $E_n\subset V_n\times V_n$ (represented as the union of red dots and blue dots) from Example \ref{e:ridux}-\eqref{it:ridux-rook}, for $n=8,9,11$ and $k_n = \lfloor n^{0.9}\rfloor$, where each element of $V_n = \{(i,j) : i,j\in [n]\}$ has been identified with an element of $[n^2]$ by ordering $V_n$ according to the lexicographic order. Pairs $(v_1,v_2)$ not belonging to $E_n$ are represented as grey dots. In this case, the reducibility of $\{E_n\}$ implies that the red dots can be removed from $E_n$ without affecting the asymptotic behavior of the homogeneous sum $Z_n$ defined in \eqref{eq:defpolchaos} for $q_n\equiv 1$. 
  \label{f:blueboxes}}
\end{figure}

\subsection{Examples of irreducible sequences}\label{subsec:examples}

In this section, we make use of the lexical conventions put forward in Definition \ref{def:irreducibilegraphs} and present several examples of irreducible sequences of graphs that generate irreducible CLTs.

\smallskip 

We start with a general fact showing that, under fairly general conditions, Cartesian products of regular graphs always generate irreducible CLTs. For definitions and properties of Cartesian products, we refer to Appendix \ref{appendixB}, as well as to \cite[Sections 4, 5 and 33]{ProductHandbook} or \cite[Section 1.4]{BrHaBook}. 

\begin{proposition}\label{p:bella} Consider a sequence $G^0_n = (V_n^0, \cE_n^0)$, $n\geq 1$, of connected graphs, with normalized Laplace spectra given by
$$
0 = \mu_1^{(0,n)} < \mu_2^{(0,n)}\leq \cdots \leq \mu_{|V^0_n|}^{(0,n)}\leq 2, \quad n\geq 1.
$$
 Assume that {\rm (i)} $G^0_n$ is $d_n$-regular for some $d_n\geq 1$, {\rm (ii)} $|V^0_n|\to\infty$, and {\rm (iii)} $\liminf_{n\to\infty} \mu_2^{(0,n)}  >0 $. Then, for any fixed integer $m\geq 2$, the sequence of product graphs
$$
G_n := \left(G^0_n\right)^{\cpow m}, \quad n\geq 1,
$$
generates an irreducible CLT, in the sense of Definition \ref{def:irreducibilegraphs}. 
\end{proposition}
\begin{proof} Writing $|V^0_n| := N_n$, one has that $G_n$ is a $md_n$-regular connected graph with $N_n^m$ vertices. Denote by $0 = \mu_1^{(n)} < \mu_2^{(n)}\leq \cdots \leq \mu_{N_n^m}^{(n)}\leq 2$, $ n\geq 1$, the normalized Laplace eigenvalues of $G_n$. Then, using the content of Appendix \ref{appendixB} (see also \cite[Section 1.4.6]{BrHaBook}) and the regularity of $G_n$ one has that $\mu_2^{(n)} = \mu_2^{(0,n)}/m $ (with multiplicity equal to $m$ times the multiplicity of $\mu_2^{(0,n)}$ in $G^0_n$), and the irreducibility of $\{G_n\}$ immediately follows from Theorem \ref{th:mainSPECTRUMgraph}. The fact that $\{G_n\}$ generates a CLT is a consequence of Remark \ref{r:afterdejong}-\ref{item:rempoint2}, and of the property that, trivially, 
$$
\frac{m d_n}{N_n^m}\leq \frac{m}{N_n^{m-1}} \to 0, \quad \mbox{as}\,\, n\to\infty.
$$
\end{proof}

\begin{remark}\label{r:mecha}{\rm We stress that Proposition~\ref{p:bella} holds without any assumption on the numerical sequence $\{d_n\}$. In particular, one may have $d_n \asymp |V_n|$, in which case (by virtue of Remark~\ref{r:afterdejong}-\ref{item:rempoint2}) the sequence $\{G_n^0\}$ {\em does not} generate a CLT (and is therefore irreducible, see Subsection~\ref{sec:reducible-CLT}). Heuristically, Proposition~\ref{p:bella} shows that, given such a $\{G_n^0\}$, the $m$-fold Cartesian power operation  
\[
G_n^0 \mapsto \big(G_n^0\big)^{\cpow m}
\]  
creates enough sparsity to induce Gaussian fluctuations, yet without altering the irreducible structure. See the forthcoming Examples \ref{item:completegraph} and \ref{item:bipartite} for two illustrations of this phenomenon.

}  
\end{remark}

To analyze some of the examples described below, we will use the following facts:
\begin{itemize}
\item[--] for $n\geq 2$, the complete graph $K_n$ over $n$ vertices has normalized Laplace spectrum given by
\begin{equation}\label{e:completespectrum}
\left\{0,\frac{n}{n-1}\right\}, \,\,\mbox{with respective multiplicities $1$ and $(n-1)$} \,; 
\end{equation}

\item[--] for $n\geq 2$, the complete bipartite graph $K_{n,n}$ over $2n$ vertices has normalized Laplace spectrum given by
\begin{equation}\label{e:bipartitespectrum}
\left\{0,1,2\right\}, \,\,\mbox{with respective multiplicities $1$, $2(n-1)$ and 1} \,. 
\end{equation}

\end{itemize} 

See e.g. \cite[Section 1.4]{BrHaBook}. For the rest of the section, and consistently with the notation adopted before, we write $\{X_i : i\geq 1\}$ to indicate a collection of i.i.d. $\mathcal{N}(0,1)$ random variables.

\medskip  

\noindent{\it Some explicit examples of irreducible sequences ($d=2$)}:

\smallskip

\begin{enumerate}[label=(\alph*)]

\item ({\em Expanders}) \label{item:expanders} Fix $m\geq 2$ and $\epsilon>0$, and consider a sequence of $m$-regular graphs $G_n = (V_n, \cE_n)$, $n\geq 1$, such that $|\cE_n|, \, |V_n| \to \infty$. According to the classical definition given e.g. in \cite[Definition 2.2]{hlw}, the family $\{G_n\}$ is said to be a $\epsilon$-{\em expander} if 
$$
\min_{0< |S|\leq |V_n|/2} \frac{E(S, \bar{S})}{{\rm Vol}(S)} := \tilde{\varphi}_2(G_n)\ge \frac{\epsilon}{m}
$$ 
for all $n$, where $E(S, \bar{S})$ indicates the number of edges of $G_n$ connecting $S$ with its complement $\bar{S}$, and ${\rm Vol}(S)$ is the sum of the degrees of the vertices in $S$ (see Section~\ref{sec:GRAPHS}, in particular the discussion around formula \eqref{e:tilde}, for more details). Using the Cheeger-type inequality stated below in \eqref{e:truecheeger}, one sees that, in this case, $\liminf_{n\to\infty} \mu_2^{(n)}\ge \epsilon^2/(2 m^2)>0$. Since $m$ is fized while $|V_n| \to \infty$, in particular $m/|V_n|\to 0$, one can now combine Remark~\ref{r:afterdejong}-\ref{item:rempoint2} with Theorem~\ref{th:mainSPECTRUMgraph} to infer that $\{G_n\}$ generates an irreducible CLT, in the sense of Definition \ref{def:irreducibilegraphs}. An analogous conclusion holds in the case of a sequence of $m_n$-regular graphs verifying $\tilde{\varphi}_2(G_n)\ge \epsilon_n/m_n$ for some $\epsilon_n$ such that $\liminf_{n\to\infty} \epsilon_n/m_n>0$, and such that $m_n = o(|V_n|)$. Canonical examples of expanders include {\em Ramanujan graphs} \cite{rammurtysurvey, hlw, LuboSurvey} and 
 {\em random $d$-regular graphs} (which are expanders with high probability,
 see \cite{PuderInventiones, LarryGraphs, FriedmanSecondEig}).

\medskip

\item \label{item:completegraph}({\em Cartesian products of complete graphs}) This example significantly expands the content of Example \eqref{it:ridux-cube} of Section~\ref{sec:re}. For integers $q,m \geq 2$, consider the the $m$-fold Cartesian product $G(q,m) := K_q^{{\cpow} m}$. Then, combining \eqref{e:completespectrum} (for $n=q$) with the content of Appendix \ref{appendixB} (see also \cite[Proposition 33.6]{ProductHandbook} or \cite[Section 1.4.6]{BrHaBook}), one sees that $G(q,m)$ is a $m(q-1)$ regular graph over $q^m$ vertices with a normalized Laplace spectrum given by
$$
\left\{ \frac{\ell q}{m(q-1)}  : \ell=0,1,...,m \right\},
$$
and respective multiplicities
$$
\binom{m}{\ell } (q-1)^\ell , \quad \ell =0,...,m.
$$
We can immediately draw the following conclusions:
\smallskip
\begin{enumerate}[label=(\roman*)]
\item \label{item:ci} For fixed $m\geq 2$, the sequence of graphs $G_n:= G(n,m) = K_n^{\cpow m}$, $n\geq 1$, generates an irreducible CLT as a direct consequence of Proposition \ref{p:bella} in the case $G_0^n = K_n$ (taking into account \eqref{e:completespectrum}). Note that $G(n,2) =K_n^{{\cpow} 2} $ corresponds to the ``Rook's graph'', studied e.g. in \cite{chessboardgraphs}. It is interesting to observe that the sequence $\{ G(n,2)\} $ can be identified with the sequence $(V_n, \cE_n)$, $n\geq 1$, from Example \ref{e:ridux}-\eqref{it:ridux-rook} \emph{in the special case $k_n =n$} (as illustrated in Figure \ref{f:bluespiders}-(c)). This framework also corresponds to the example introduced in Subsection~\ref{sss:adhoc}, with $\bbS_{\mathsf{v}}(a) = \{a\} \times \{1,\ldots, n\}$
and $\bbS_{\mathsf{h}}(b) = \{1,\ldots, n\} \times \{b\}$.
\item \label{item:cii} For fixed $q\geq 2$, the sequence of graphs $G_n := G(q,n)$ is such that, for all $k\geq 2$, $\lim_{n\to \infty} \mu_k^{(n)} = 0$ and, in principle, one cannot use Theorem \ref{th:mainSPECTRUMgraph} to determine whether $\{G_n\}$ is reducible or not. The reduciblity of $\{G_n\}$ can be established by using a direct argument. To see this, observe that $G_n$ is isomorphic to the graph whose vertices are given by all vectors of the form $v = (v_1,...,v_n)$, where $v_i\in \{0,1,...,q-1\}$ and $v\sim w$ if and only if $v_i\neq w_i$ for exactly one index $i=1,...,n$; in particular, $G(2,n) \simeq Q_n$, the hypercube considered in Example \ref{e:ridux}-\eqref{it:ridux-cube}. The reducibility of $\{G_n\}$ can now be deduced by a straightforward variation of the arguments rehearsed in Example \ref{e:ridux} (the details are left to the reader).
\end{enumerate}
We observe that, if $K_n$ is identified with the complete graph over $[n]$, then the sequence of homogeneous sums generated by $\{K_n\}$ (in the sense of Definition \ref{def:irreducibilegraphs}) is given for $n\in\N$ by
$$
Z_n = \sum_{1\leq i\neq j \leq n} X_iX_j
 = (X_1 + \ldots + X_n)^2
- (X_1^2 + \ldots + X_n^2)  \,.
$$
It is easily seen that ${\bf Var} (Z_n)= 2n(n-1)$ and that, as $n\to\infty$, $\tilde{Z}_n:= Z_n/\sqrt{2n(n-1)}$ converges in distribution towards $\frac{1}{\sqrt{2}}(N^2 - 1)$ with $N \sim \mathcal{N}(0,1)$, that is a multiple of a centered $\chi^2$ random variable with one degree of freedom. See Remark \ref{r:mecha}, as well as \cite[formula (1.4)]{BDMM22}. Thus, as announced, the sequence of graphs $K_n$ does not generate a CLT, while for any fixed $m \ge 2$ the cartesian products $K_n^{\cpow m}$ generate (an irreducible) CLT.

\medskip 

\item \label{item:bipartite} ({\em Cartesian products of bipartite graphs}) Fix $m\geq 2$. By virtue of \eqref{e:bipartitespectrum}, another direct application of Proposition \ref{p:bella} in the case $G^0_n = K_{n,n}$, shows that the sequence $G_n = K_{n,n}^{\cpow m}$, $n\geq 1$, generates an irreducible CLT. We observe that $G_n$ can be identified with the graph having vertex set $V_n = \{(v_1,...,v_m) : v_i\in [2n]\}$ and such that $v\sim w$ if and only if $v_i\neq w_i$ for exactly one index $i$, and either $v_i\in [n]$ and $w_i\in \{n+1,...,2n\}$, or $w_i\in [n]$ and $v_i\in \{n+1,...,2n\}$. See Figure \ref{f:star} for an illustration of this graph in the case $n=3$ and $m=2$. If $K_{n,n}$ is identified with the bipartite graph over $[2n]$, then the sequence of homogeneous sums generated by $\{K_{n,n}\}$ can be written as
$$
Z_n =2 \left(\sum_{i=1}^n X_i\right)\times \left( \sum_{j=n+1}^{2n} X_j \right) \stackrel{\rm law}{=} 2n\,X_1X_2, \quad n\geq 1.
$$
See again Remark \ref{r:mecha}, as well as \cite[Example 2.2]{BDMM22}.
Also in this case, the sequence of graphs $K_{n,n}$ does not generate a CLT, but for any fixed $m \ge 2$ the Cartesian products $K_{n,n}^{\cpow m}$ generate (an irreducible) CLT.

\medskip

\item \label{item:cartprod(e)}({\em Cartesian products of graphs with isolated vertices}) Let $\{m_n\}$ be an integer-valued sequence such that $m_n\to \infty$ and consider the framework of Subsection~\ref{sss:adhoc} for $\bbS_{\mathsf{v}}(a) = \{a\}\times \{1,..., m_n \}$ and $\bbS_{\mathsf{h}}(b) = \{1,..., m_n \} \times \{b\}$. In this case, it is easily seen that the indicator $\ind_{E_n}$ coincides with the adjacency matrix of the graph $$ U(n):= (K_{m_n}\sqcup \bar{K}_{n-m_n}) \raisebox{-0.25ex}{\scalebox{1.8}{$\square$}} (K_{m_n}\sqcup \bar{K}_{n-m_n}),$$ where $K_{m_n}$ is the complete graph over $m_n$ vertices and $\bar{K}_{n-m_n}$ is the trivial graph with $n-m_n$ isolated vertices. A direct inspection shows that the graph $U(n)$ has $(n-m_n)^2$ isolated vertices, $2(n-m_n)$ connected components isomorphic to $K_{m_n}$ and one ``giant'' component isomorphic to $K_{m_n} \raisebox{-0.25ex}{\scalebox{1.8}{$\square$}} K_{m_n}$ (some visualizations of $E_n$ in the special setting of the present example are provided in Figure \ref{f:noncoveringblueboxes}). Note that we already know that the sequence $\{K_{m_n}\}$ is associated with an irreducible CLT (see Point \ref{item:ci} in Example \ref{item:completegraph} of the present section).  

We write $G_n$, $n\geq 1$, to denote the graph obtained from $U(n)$ after removing all isolated vertices. Using once again \eqref{e:completespectrum}, one easily deduces that, for each $n$, the normalized Laplacian of $G_n$ has spectrum $\big\{0, m_n/2(m_n-1), m_n/(m_n-1)\big\}$, with respective multiplicities $1+2(n-m_n)$, $2(m_n-1)$ and $2(m_n-1)^2+ 2(n-m_n)(m_n-1)$. 
One can verify the following (details are left to the reader):
\begin{enumerate}[label=(\roman*)]
\item \label{item:ei} If $m_n = o(n)$ then $\{G_n\}$ generates a reducible CLT;
\item \label{item:eii} If $m_n/n \to \beta \in (0,1)$, then $\{G_n\}$ generates an irreducible CLT\footnote{Since the total number of edges in $K_{m_n} \raisebox{-0.25ex}{\scalebox{1.8}{$\square$}} K_{m_n}$ scales as $\asymp \beta^3 n^3$, one has that the corresponding sequence of quadratic form $\{Z_n\}$ (defined as in \eqref{eq:bilformsubsec}) is such that each $Z_n$ can be written as as the sum of two independent homogeneous sums of order $2$,
$$
Z_n = Z'_n + Z''_n,
$$
in such a way that, as $n\to \infty$, $\{Z'_n\}$ is irreducible and such that
$$
{\bf Var}(Z'_n)\sim \beta \, {\bf Var}(Z_n).
$$
This implies the irreducibility of $\{Z_n\}$.}, and $\lim_{n\to\infty}\mu_k^{(n)}= 0$ for all $k$ since the multiplicity of the zero Laplace eigenvalue is $1+2(n-m_n)\to\infty$;
\item \label{item:eiii} If $m_n/n \to 1$, then $\{G_n\}$ generates an irreducible CLT (by the same argument as for the case $\beta\in (0,1)$), and there exists an integer $k\geq 2$ such that $\liminf_{n}\mu_k^{(n)}> 0$ if and only if $\limsup_n (n-m_n)<\infty$ (and, in this case, $\liminf_{n}\mu_k^{(n)}= 1/2$).
\end{enumerate}

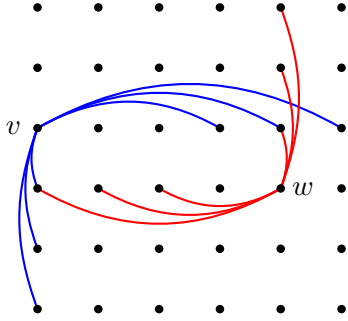
\begin{figure}[t]
  \centering
  \begin{tikzpicture}[scale=0.8]

    \draw[blue, thick] (0,3) to[out=30, in=150] (5,3);
\draw[red, thick] (0,2) to[out=-30, in=-150] (4,2);
\draw[red, thick] (1,2) to[out=-30, in=-150] (4,2);
\draw[red, thick] (2,2) to[out=-30, in=-150] (4,2);
\draw[red, thick] (4,3) to[out=290, in=70] (4,2);
\draw[red, thick] (4,4) to[out=290, in=70] (4,2);
\draw[red, thick] (4,5) to[out=290, in=70] (4,2);

    \draw[blue, thick] (0,3) to[out=30, in=150] (4,3);
\draw[blue, thick] (0,3) to[out=30, in=150] (3,3);
 
  \draw[blue, thick] (0,3) to[out=250, in=110] (0,2);
    
   \draw[blue, thick] (0,3) to[out=250, in=110] (0,1); 
   \draw[blue, thick] (0,3) to[out=250, in=110] (0,0);  

    \foreach \x in {0,1,2,3,4,5}
      \foreach \y in {0,1,2,3,4,5}
        \fill (\x,\y) circle (2pt);

    \node[left] at (-0.1,3) {$v$};
\node[right] at (4.02,2) {$w$};
  \end{tikzpicture}
  \caption{\footnotesize In blue: the edges adjacent to the vertex $v = (3,1)$ in the graph $(K_{3,3})^{\cpow 2}$. In red: the edges adjacent to the vertex $w = (4,5)$, in the same graph.}
  \label{f:star}
\end{figure}

\end{enumerate}

\needspace{3\baselineskip}
\section{Irreducibility and graph spectra: background and proofs}\label{sec:GRAPHS}
In this section we prove Theorem \ref{th:mainSPECTRUMgraph}, yielding a characterization of irreducible sequences $\{Z_n\}$ defined as in \eqref{eq:bilformsubsec}.

\subsection{Preliminaries: some connectivity estimates}\label{ss:pregraph}

\subsubsection{General setting}\label{ss:generalgraph}{ Let $G = (V, \mathcal{E})$ be a finite (simple, undirected, loop-free) graph with $|V|=N\ge 2$ vertices, each one having a possibly different {\em degree} $d(v) \coloneq | \{ e \in \mathcal{E} : v \in e \} |\,,$ $v \in V$. {As usual, we regard edges $e \in \mathcal{E}$ as {unordered pairs}
$e = \{v,w\}$ and make extended use of the basic {\em degree sum formula} $\sum_{v\in V} d(v) = 2 |\mathcal{E}|$. In resonance with Convention \ref{conv;simple}, for our purposes there is no loss of generality in assuming that the graph has no isolated vertices, that is, $d(v)>0$ for all $v\in V$. For future use, we also fix an (arbitrary) enumeration $\{v_1,...,v_N\}$ of $V$.}

We write $\mathbb{R}^V$ to indicate the set of all functions $V\to \R$, whose generic element is written ${\bf x} = \{ x_v : v\in V\}$. The space $\mathbb{R}^V$ is canonically identified with the Euclidean space $\R^N$, while the Euclidean scalar product and norm on $\R^V$ are denoted by $\langle \cdot , \cdot \rangle$ and $\|\cdot\|$, respectively. Given $S\subseteq V$, we use the symbol ${\bf 1}_S$ to indicate the element of $\R^V$ such that $x_v$ equals one or zero according as $v\in S$ or $v\in \widebar{S}$, where we used the notation
\begin{equation}\label{e:complement}
    \widebar{S}:= V\backslash S \,.
\end{equation}
Given $S,T\subseteq V$, we write
\begin{equation}\label{eq:E(S,T)}
    E(S,T) := \big| \big\{ e = \{s,t\}\in \mathcal{E} : s\in S,\, t\in T \big\}\big|,
\end{equation}
that is, $E(S,T)$ is the number of edges connecting $S$ and $T$. We denote by $A := \{A(v,w) = \ind_{\{v,w\} \in \mathcal{E}} : v,w\in V\}$ the $N\times N$ \emph{adjacency} matrix of $G$ 
and write 
$$
\lambda_1 \ge \lambda_2\geq \cdots \ge \lambda_N
$$
to indicate its eigenvalues, in such way that $$\sum_{i=1}^N \lambda_i^2= {\rm Trace}(A^2)= \sum_{v \in V} A^2(v,v) = \sum_{v \in V} d(v) = 2|\mathcal{E}|.$$ We observe that, for every $S \subseteq V$, one has that
\begin{equation*}
    E(S,S) = \frac{1}{2} \langle {\bf 1}_S, A \, {\bf 1}_S \rangle \,, \qquad
    E(S,\widebar{S}) = \langle {\bf 1}_S, A \, {\bf 1}_{\widebar{S}} \rangle \,,
\end{equation*}
where the factor $\frac{1}{2}$ is motivated by the fact that each edge $e = \{v,w\}$ with $v,w \in S$ is associated with two distinct ordered pairs $(v,w)$ and $(w,v)$.

The {\em Laplacian} matrix of $G$ is defined as
$$
L \coloneq D - A, \qquad \text{where} \qquad
D(v,w) \coloneq d(v) \, \ind_{v=w}\,,\,\,v,w\in V,
$$
that is: $D$ is the $N\times N$ diagonal {\em degree matrix}, with diagonal entries $d(v_1),\ldots,d(v_N)$. We also introduce the \emph{normalized Laplacian} of $G$ as
$$\cL \coloneq I - D^{-\frac{1}{2}}AD^{-\frac{1}{2}}\,,$$
whose eigenvalues (see \cite[Section 3.2]{trevisanLN}) satisfy
$$
0=\mu_1\le \mu_2\le \cdots \le \mu_N \le 2.
$$
The second eigenvalue $\mu_2$ is the so-called {\em spectral gap} of $G$. In the special case where $G$ is a $d$-{\em regular graph}, i.e.\ $d(v)=d$ for all $v \in V$ and for some $d \ge 1$, the normalized Laplacian simplifies to $\cL = I - \frac{1}{d}A$ and $\mu_i = 1 - \frac{\lambda_i}{d}, \,\, i=1,\ldots,N$.

\smallskip

For every nonzero vector ${\bf x}\in \mathbb{R}^V$, and writing ${\bf y}\coloneq D^{-\frac{1}{2}} {\bf x} $, we define the {\em Raleygh quotient} associated with $\cL$ and ${\bf y}$ to be the ratio
$$
\frac{\langle {\bf x}, \cL\,{\bf x}\rangle}{\|{\bf x}\|^2} = \frac{\langle {\bf x}, D^{-\frac{1}{2}} LD^{-\frac{1}{2}}\,{\bf x}\rangle}{\|{\bf x}\|^2}=\frac{\langle {\bf y},  L\,{\bf y}\rangle}{\big\|D^{\frac{1}{2}}{\bf y}\big\|^2}=\frac{\sum_{u \sim v} (y_u-y_v)^2}{\sum_{v \in V}d(v) \, y_v^2} =: R_\cL({\bf y})\,.
$$
\begin{remark}\label{r:minmax}{\rm In the sequel, we will use the following variational characterisation of the eigenvalues of $\cL$ (see e.g. \cite[Section 1.3]{trevisanLN}): for all $k=1,...,N$,
\begin{equation}\label{e:quotients}
\mu_k = \min_{\substack{\mathcal{M}\subseteq \R^V \\ {\rm dim}(\mathcal{M})=k }}\max_{{\bf z}\in \mathcal{M}\backslash \{0\}} R_\cL({\bf z}),
\end{equation}
where the minimum runs over all $k$-dimensional subspaces $\mathcal{M}$ of $\R^V$. It follows that $\mu_1=0$ and the multiplicity of $\mu_1$ is equal to the number of connected components of $G$. 
}
\end{remark}

\subsubsection{Cheeger's inequalities}\label{ss:cheeger} To introduce Cheeger's inequalities, we first define a quantity that measures how well a subset \( S \subset V \) is connected to its complement \( \bar{S} \), relative to the total number of edges incident to the vertices in \( S \). Given $\emptyset \neq S\subseteq V $, the {\em edge expansion} of $S$ is defined as
$$
\phi(S) := R_\cL({\bf 1}_S) = \frac{E(S,\widebar{S}) }{\text{vol}(S)}, \qquad \text{where we set} \quad\text{vol}(S)\coloneq \sum_{v \in S}d(v) \,.
$$
and the second equality follows from the relation $\langle{\bf 1}_S , (D - A) {\bf 1}_S\rangle = \text{vol}(S) - 2 E(S,S)=E(S,\widebar{S})$ (see Figure \ref{fig:edgeexp} for illustrations). In particular, one has trivially that $\phi(V) = 0$. The quantity $\text{vol}(S)$ is called the \emph{volume of $S$}, and it equals $d |S|$ when $G$ is $d$-regular. 
}

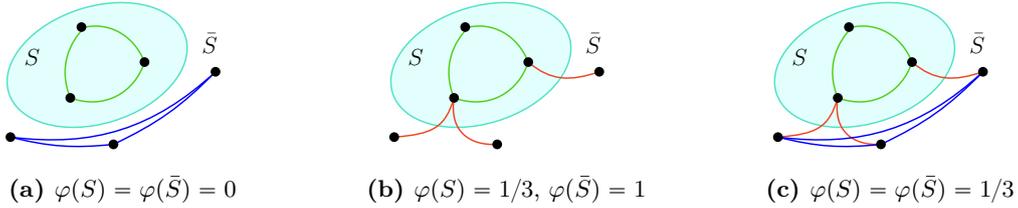
\begin{figure}[h]
  \centering

  \begin{subfigure}[b]{0.30\linewidth}
    \centering
    \resizebox{0.7\linewidth}{!}{%
    \begin{tikzpicture}[x=0.75pt,y=0.75pt,yscale=-1,xscale=1]
      \path[use as bounding box] (250,10) rectangle (410,120);

      \draw[color=Sborder, fill=Sfill, line width=0.75]
        (253.35,73.55) .. controls (247.43,52.17) and (269.75,26.53) ..
        (303.2,16.3) .. controls (336.65,6.06) and (368.56,15.09) ..
        (374.48,36.48) .. controls (380.4,57.86) and (358.08,83.5) ..
        (324.63,93.74) .. controls (291.18,103.97) and (259.27,94.94) ..
        (253.35,73.55) -- cycle;

      \draw[color=Gedge, line width=0.75] (295.57,77.68) .. controls (319.35,86.34) and (339.63,73.68) .. (346.29,53.02);
      \draw[color=Gedge, line width=0.75] (295.57,77.68) .. controls (287.02,56.35) and (295.26,39.02) .. (305.4,29.02);
      \draw[color=Gedge, line width=0.75] (346.29,53.02) .. controls (339,34.69) and (322.52,25.36) .. (305.4,29.02);

      \draw[blue, line width=0.75] (254.69,104.67) to[bend left=10] (325.05,109.67);
      \draw[blue, line width=0.75] (325.05,109.67) to[bend left=10] (394.78,59.68);
      \draw[blue, line width=0.75] (254.69,104.67) to[bend left=25] (394.78,59.68);

      \foreach \x/\y in {
        303.18/28.93, 295.57/77.68, 346.29/53.02,
        394.78/59.68, 254.69/104.67, 325.05/109.67}
        \fill (\x,\y) circle (2.5pt);

      \node at (269.44,49.25) [font=\large] {$S$};
      \node at (390.78,39.68) [font=\large] {$\bar{S}$};
    \end{tikzpicture}
    }
    \caption{$\varphi(S) = \varphi(\bar{S}) = 0$}
  \end{subfigure}
  \quad
  \begin{subfigure}[b]{0.30\linewidth}
    \centering
    \resizebox{0.7\linewidth}{!}{%
    \begin{tikzpicture}[x=0.75pt,y=0.75pt,yscale=-1,xscale=1]
      \path[use as bounding box] (250,10) rectangle (410,120);

      \draw[color=Sborder, fill=Sfill, line width=0.75]
        (253.35,73.55) .. controls (247.43,52.17) and (269.75,26.53) ..
        (303.2,16.3) .. controls (336.65,6.06) and (368.56,15.09) ..
        (374.48,36.48) .. controls (380.4,57.86) and (358.08,83.5) ..
        (324.63,93.74) .. controls (291.18,103.97) and (259.27,94.94) ..
        (253.35,73.55) -- cycle;

      \draw[color=Gedge, line width=0.75] (295.57,77.68) .. controls (319.35,86.34) and (339.63,73.68) .. (346.29,53.02);
      \draw[color=Gedge, line width=0.75] (295.57,77.68) .. controls (287.02,56.35) and (295.26,39.02) .. (305.4,29.02);
      \draw[color=Gedge, line width=0.75] (346.29,53.02) .. controls (339,34.69) and (322.52,25.36) .. (305.4,29.02);

      \draw[color=Redge, line width=0.75] (394.78,59.68) .. controls (375.13,66.35) and (361.82,66.35) .. (346.29,53.02);
      \draw[color=Redge, line width=0.75] (254.69,104.67) .. controls (281.31,102.67) and (289.55,98) .. (295.57,77.68);
      \draw[color=Redge, line width=0.75] (325.05,109.67) .. controls (306.03,110.67) and (292.09,100.67) .. (295.57,77.68);

      \foreach \x/\y in {
        303.18/28.93, 295.57/77.68, 346.29/53.02,
        394.78/59.68, 254.69/104.67, 325.05/109.67}
        \fill (\x,\y) circle (2.5pt);

      \node at (269.44,49.25) [font=\large] {$S$};
      \node at (390.78,39.68) [font=\large] {$\bar{S}$};
    \end{tikzpicture}
    }
    \caption{$\varphi(S)=1/3$, $\varphi(\bar{S})=1$}
  \end{subfigure}
  \quad
  \begin{subfigure}[b]{0.30\linewidth}
    \centering
    \resizebox{0.7\linewidth}{!}{%
    \begin{tikzpicture}[x=0.75pt,y=0.75pt,yscale=-1,xscale=1]
      \path[use as bounding box] (250,10) rectangle (410,120);

      \draw[color=Sborder, fill=Sfill, line width=0.75]
        (253.35,73.55) .. controls (247.43,52.17) and (269.75,26.53) ..
        (303.2,16.3) .. controls (336.65,6.06) and (368.56,15.09) ..
        (374.48,36.48) .. controls (380.4,57.86) and (358.08,83.5) ..
        (324.63,93.74) .. controls (291.18,103.97) and (259.27,94.94) ..
        (253.35,73.55) -- cycle;

      \draw[color=Gedge, line width=0.75] (295.57,77.68) .. controls (319.35,86.34) and (339.63,73.68) .. (346.29,53.02);
      \draw[color=Gedge, line width=0.75] (295.57,77.68) .. controls (287.02,56.35) and (295.26,39.02) .. (305.4,29.02);
      \draw[color=Gedge, line width=0.75] (346.29,53.02) .. controls (339,34.69) and (322.52,25.36) .. (305.4,29.02);

      \draw[color=Redge, line width=0.75] (394.78,59.68) .. controls (375.13,66.35) and (361.82,66.35) .. (346.29,53.02);
      \draw[color=Redge, line width=0.75] (254.69,104.67) .. controls (281.31,102.67) and (289.55,98) .. (295.57,77.68);
      \draw[color=Redge, line width=0.75] (325.05,109.67) .. controls (306.03,110.67) and (292.09,100.67) .. (295.57,77.68);

      \draw[blue, line width=0.75] (254.69,104.67) to[bend left=10] (325.05,109.67);
      \draw[blue, line width=0.75] (325.05,109.67) to[bend left=10] (394.78,59.68);
      \draw[blue, line width=0.75] (254.69,104.67) to[bend left=25] (394.78,59.68);

      \foreach \x/\y in {
        303.18/28.93, 295.57/77.68, 346.29/53.02,
        394.78/59.68, 254.69/104.67, 325.05/109.67}
        \fill (\x,\y) circle (2.5pt);

      \node at (269.44,49.25) [font=\large] {$S$};
      \node at (390.78,39.68) [font=\large] {$\bar{S}$};
    \end{tikzpicture}
    }
    \caption{$\varphi(S) = \varphi(\bar{S}) = 1/3$}
  \end{subfigure}

  \caption{Computing edge expansions in various graphs: in each picture, the number of red edges equals the quantity $E(S, \bar{S})$ entering the definition of $\varphi(S)$.}
  \label{fig:edgeexp}
\end{figure}
{For $k\geq 2$ we introduce the notation
\begin{equation}\label{eq:phi_k(G)}
    \phi_k(G) := \min_{\substack{S_1,...,S_k \subset V\\ \mbox{\tiny nonempty and disjoint}}} 
\max_{i=1,...,k}\phi(S_i)\,.
\end{equation}

\begin{remark}{\rm 
\begin{enumerate}[label=(\arabic*)]
 \item ({\it A probabilistic interpretation of edge expansions}) Fix a nonempty $S\subseteq V$, and sample an edge adjacent to one of the vertices in $S$ according to the following procedure: (i) sample a vertex $v_0$ in $S$ with a probability proportional to its degree, and (ii) sample an edge $e_0$ adjacent to $v_0$ uniformly at random. Then, $\varphi(S)$ equals the probability that $e_0$ connects $S$ and $\bar{S}$, that is, that $e_0 = \{a,b\}$ for some $a\in S$ and $b\in \bar{S}$. Note that, in the regular case, Step (i) above reduces to selecting one vertex uniformly at random within $S$.

\item For every $k\ge 2$ one has that
$$
\phi_k(G)\le \min_{\substack{S_1,...,S_k\\ \mbox{\tiny partition of $V$ with $S_i$ nonempty}}} \max_{i=1,...,k}\phi(S_i) ,
$$
so that, in particular,
\begin{equation}\label{e:tilde}
\phi_2(G) \le \min_{\emptyset\subset S\subset V} 
\max\big\lbrace\phi(S), \phi(\widebar{S})\big\rbrace = \min_{0< |S|\leq |V|/2} \phi(S) =: \widetilde{\varphi}_2(G)\,. 
\end{equation}
\end{enumerate}
}
\end{remark}

We now state the ``easy direction'' of \emph{Cheeger's inequalities} of order~$k$, and include a sketch of the proof --- not only for completeness, but also because the result plays a crucial role in the arguments developed below. The proof relies on the min-max characterization of the eigenvalues of~$\cL$; see Remark \ref{r:minmax}.

\begin{proposition}[Cheeger's inequalities, easy direction]\label{p:easycheeger} Let $G=(V,\mathcal{E})$ be a finite graph respecting Convention \ref{conv;simple} with $|V|=N$ and let $\mu_k$ be the $k$-th smallest eigenvalue of the normalized Laplacian $\cL$, for $k=2,\ldots,N$.
Then,
\begin{equation}\label{eq:cheegereasydir}
    \mu_k \le 2 \, \phi_k(G)\,.
\end{equation}

\end{proposition}
\begin{proof}[Sketch of the proof] Fix nonempty disjoint subsets $S_1,...,S_k \subset V$. A standard extension to irregular graphs of \cite[Lemma 7.1]{trevisanLN} yields that, for all vectors $(\alpha_1,...,\alpha_k) \in \R^k$ not identically zero, 
\begin{equation}
R_L\left(\sum_{i=1}^k \alpha_i{\bf 1}_{S_i}\right) \leq 2\, \frac{\sum_{i=1}^k\alpha_i^2 \, \text{vol}(S_i) \,\phi(S_i)}{\sum_{i=1}^k \alpha_i^2 \,\text{vol}(S_i)} \leq 2 \max_{i=1,...,k}\phi(S_i) .
\end{equation}
The conclusion follows from \eqref{e:quotients}, first by computing $\max_{{\bf z}\in \mathcal{M}\backslash \{0\}} R_\cL({\bf z})$ when $\mathcal{M}$ is the vector space generated by ${\bf 1}_{S_1},...,{\bf 1}_{S_k}$, and then by taking the minimum over all nonempty disjoint subsets $S_1,...,S_k \subset V$. 
\end{proof}

The following statement (of which we omit the proof) is taken from \cite[Theorem 3.8]{LGT14}, and is a version of the ``hard Cheeger inequality'' for graph partitions; see also \cite{Vempala12}.

\begin{theorem}[Cheeger's inequality, hard direction]\label{t:hardcheeger}
Let $G=(V,\mathcal{E})$ be a finite graph, respecting Convention \ref{conv;simple} and such that $|V|=N$, and let $\mu_k$ be the $k$-th smallest eigenvalue of its normalized Laplacian $\cL$, for $k=2,\ldots,N$. Then, there exists an absolute constant $C>0$ such that, for every $k=2,...,N$, there exists a partition $S_1,...,S_k$ of $V$ with nonempty subsets verifying
\begin{equation*}
\max_{i=1,\ldots, k} \phi(S_i)\leq Ck^4 \, \sqrt{\mu_k}\,.
\end{equation*}    
{In particular, one has that $\phi_k(G) \leq Ck^4 \, \sqrt{\mu_k}$.}
\end{theorem}

We also remark that, in the case $k=2$, one has the stronger estimate
\begin{equation}\label{e:truecheeger}
\widetilde{\phi}_2(G) \leq \sqrt{2\mu_2},
\end{equation}
where $\widetilde{\phi}_2(G)$ is defined according to \eqref{e:tilde} --- see e.g. \cite[Theorem 2.4]{hlw}.

\medskip

We conclude by stating an elementary bound, that will be exploited in the proof of Theorem \ref{th:mainSPECTRUMgraph}.

\begin{proposition} \label{prop:simple} Let $G=(V,\mathcal{E})$ be a finite graph as in Convention \ref{conv;simple}, with $|V|=N$. Fix an integer $1\leq m\leq N$. Consider a collection $B_1,...,B_m$ of nonempty disjoint subsets of $V$, and assume that
\begin{equation}\label{e:ordering}
\phi(B_1)\le \phi(B_2)\le \cdots \le \phi(B_m)\,. 
\end{equation}
Then, for all $k=1,...,m$,
\begin{equation}\label{e:simple}
 \mu_k \leq 2 \, \frac{\sum_{i=k}^m \mathrm{vol}(B_i) \,\phi(B_i)}{\sum_{i=k}^m \mathrm{vol}(B_i)} \,  . 
\end{equation}
\end{proposition}
\begin{proof}
Fix $k=1,...,m$. Then \eqref{e:ordering} implies that
$$
\frac{\sum_{i=k}^m \mathrm{vol}(B_i)\, \phi(B_i)}{\sum_{i=k}^m \mathrm{vol}(B_i)} \ge \phi(B_k).
$$
The conclusion follows by observing that, again because of \eqref{e:ordering},
$$
\phi(B_k) =
\min_{1\le j_1<\cdots < j_k\le m}\,\, \max_{\ell = 1,...,k} \phi(B_{j_\ell})
\ge \phi_k(G)\ge \frac{\mu_k}{2},
$$
where we have used Proposition \ref{p:easycheeger}.
\end{proof}
}
\subsection{Laplace spectra and irreducibility} \label{subsec:graphirreduc}
{This subsection focuses on the proof of Theorem \ref{th:mainSPECTRUMgraph}, which requires us to consider a sequence of quadratic sums $\{Z_n\}$ as in \eqref{eq:bilformsubsec}. We stress that such a framework corresponds to the case $d=2$ and $q_n\equiv 1$ in our general setting \eqref{eq:defpolchaos}, that the sequence $\{E_n\}$ is assumed to satisfy conditions \eqref{item:VE1}---\eqref{item:VE3} from the beginning of Section \ref{ss:prelim}, and that the indicator $\ind_{E_n}$ is identified with the adjacency matrix of an appropriate graph $G_n = (V_n, \mathcal{E}_n)$ such that $|E_n| = 2|\mathcal{E}_n|$. For $n\ge 1$, the normalized Laplace spectrum of $G_n$ is written $0 = \mu_1^{(n)}\leq \mu_2^{(n)}\le \cdots \le \mu_{N_n}^{(n)}\leq 2$.
\smallskip

The following elementary result (whose proof is left to the reader) uses the fact that, in the framework of the present section and using the notation $\eqref{eq:sigmaBdef}$  and $\eqref{eq:E(S,T)}$, one has that
$$
\sigma_n^2(B) = 2|E_n \cap (B\times B)| = 4E(B,B), \quad B\subseteq V.
$$
\begin{lemma}[Reducibility and graphs]\label{def:irreducibility22} Let the assumptions and conventions of the present section hold. Then, the sequence $\{Z_n\}$ is reducible in the sense of Definition \ref{def:irreducibility} (see also Remark \ref{r:reducible}-{\rm \ref{redCLTconda},\ref{redCLTcondb},\ref{redCLTcondc}}) if and only if there exist partitions $\{B_1,\ldots,B_{m_n}\} = \{B^{(n)}_1,\ldots,B^{(n)}_{m_n}\}$ of the vertex sets $V_n$, such that, as $n \to \infty$,
    \begin{enumerate}[label={\rm (\roman*')}]
    \item \label{item:irr'1} $m_n\to \infty$;
    \item \label{item:irr'2} $\sum_{i=1}^{m_n} E(B_i,B_i) \sim |\mathcal{E}_n|$;
    \item \label{item:irr'3} $\max_{i=1,...,m_n} E(B_i,B_i) = o\big(|\mathcal{E}_n|\big)$.  
\end{enumerate}
\end{lemma} 

As indicated in the statement, the definition of the elements partition $\{B_1,\ldots,B_{m_n}\}$ depends in principle on $n$; this dependence has been omitted in \ref{item:irr'1}, \ref{item:irr'2}, and \ref{item:irr'3} to avoid overloading the notation.

We can now prove one of the main results of the paper.

\begin{proof}[Proof of Theorem \ref{th:mainSPECTRUMgraph}] We reason by contradiction by assuming that \eqref{e:mucca} holds for some $k\geq 2$ and that there exists partitions $\{B_1,\ldots, B_{m_n}\}$, $n\geq 1$, such that the three properties \ref{item:irr'1}---\ref{item:irr'3} listed in Lemma \ref{def:irreducibility22} hold true. Without loss of generality we can always assume that the partition $B_1,\ldots, B_{m_n}$ satisfies the ordering relation \eqref{e:ordering} for all $n\ge 1$, so that \eqref{e:simple} yields
\begin{equation} \label{eq:bound-muk}
    \frac{\mu_k^{(n)}}{2}\leq \frac{\sum_{i=k}^{m_n} q^{(n)}_i  \phi(B_i)}{\sum_{i=k}^{m_n} q^{(n)}_i}\,,
\qquad \text{where } \
q^{(n)}_i := \frac{\mathrm{vol}\big(B_i\big)}{\mathrm{vol}(V_n)}
= \frac{\mathrm{vol}\big(B_i\big)}{2 \, |\cE_n|}\,.
\end{equation}
To achieve the desired contradiction and conclude the proof, it is now sufficient to show that the right-hand side of the previous inequality converges to zero. Note that
\begin{equation}\label{eq:parti}
\begin{split}
\frac{1}{2} 
\sum_{i=1}^{m_n} E(B_i, \widebar{B_i}) &
= \frac{1}{2} \sum_{i=1}^{m_n} 
 \mathrm{vol} \big( B_i\big) - \sum_{i=1}^{m_n} E(B_i, B_i) \\
&= \frac{1}{2} \mathrm{vol}(V_n) - \sum_{i=1}^{m_n} E(B_i, B_i) \\
&=|\cE_n| - \sum_{i=1}^{m_n}  E(B_i, B_i) 
= o\big(|\cE_n|\big)
\end{split}
\end{equation}
by property \ref{item:irr'2}. As a consequence,
\begin{equation} \label{eq:BBbar}
    \sum_{i=1}^{m_n} \frac{E(B_i, \widebar{B_i})}{2\,|\cE_n|}
= \sum_{i=1}^{m_n} q^{(n)}_i \phi(B_i)
= o(1) \,.
\end{equation}
This implies that the numerator on the right-hand side of the inequality
\eqref{eq:bound-muk} vanishes, and it remains to show that the
denominator is bounded away from zero. We conclude by observing that such a denominator converges indeed to to~$1$, which is equivalent to the relation $\sum_{i=1}^{k-1} q^{(n)}_i \to 0$
since $\sum_{i=1}^{m_n} q^{(n)}_i = 1$.
To see this, we simply write
\begin{equation} \label{eq:pbound-}
    q^{(n)}_i =
\frac{\mathrm{vol}(B_i)}{\mathrm{vol}(V_n)} =  \frac{\mathrm{vol}(B_i)}{2|\cE_n|} = \frac{ E(B_i, B_i)}{|\cE_n|} +\frac{E(B_i, \widebar{B_i})}{2\,|\cE_n|} = o(1) \,,
\end{equation}
uniformly for~$i \in \{1,\ldots, m_n\}$ (because of \ref{item:irr'3}).
\end{proof}
}

\subsection{Full irreducibility}\label{ss:full}
We conclude this section with a (slight) generalization of Theorem~\ref{th:mainSPECTRUMgraph}. Let $G_n = (V_n, \cE_n)$, $n\geq 1$, be a sequence of graphs, as those studied in the previous section, and consider the associated sequence $\{Z_n\}$. We say that the sequence $\{G_n\}$ is \emph{partially reducible} if the following conditions hold.
\begin{itemize}
\item There exist $\rho \in (0,1]$ and a subset $V'_n \subseteq V_n$, $n\ge 1$, such that (recall \eqref{e:complement})
\begin{equation}\label{eq:V'rho}
    \text{vol}(V'_n) = (\rho + o(1)) \text{vol}(V_n) \,,
    \qquad
    E\big(V'_n, \widebar{V'_n}\big) = o(|\cE_n|) \,,
\end{equation}
that is, $V'_n$ contains asymptotically a
fraction $\rho > 0$ of vertices and it is loosely connected to its complement.

\item There exist partitions $B_1,\ldots, B_{m_n}$ of $V'_n$, $n\geq 1$, where $B_i=B_i^{(n)}$ for $i=1,\ldots,m_n$, such that
\begin{enumerate}[label={(\roman*'')}]
\item \label{item:red1''} $m_n\to \infty$;
\item \label{item:red2''} $\sum_{i=1}^{m_n} E(B_i,B_i) = E(V'_n, V'_n) + o(|\cE_n|)$;
\item \label{item:red3''} $\max_{i=1,\ldots,m_n}E(B_i,B_i) = o(|\cE_n|)$.  
\end{enumerate}
\end{itemize}
We say that a sequence of graphs $\{G_n\}$ is \emph{fully irreducible} if it is not partially reducible.

\begin{remark}{\rm 
    The second condition in \eqref{eq:V'rho} is equivalent to $\phi(V'_n) = o(1)$, thanks to the first condition which yields $|\cE_n| = \frac{1}{2} d_n |V_n|
    \sim \frac{1}{2} \frac{1}{\rho} d_n |V'_n|$. Also note that the second condition in \eqref{eq:V'rho} follows by the first one in the extreme case $\rho=1$, simply because $E\big(V'_n, \widebar{V'_n}\big) \le d_n |\widebar{V'_n}| = d_n(|V_n| - |V'_n|) = d_n|V_n| (1-\rho+o(1))  = o(d_n|V_n|) = o(|\cE_n|)$.}
\end{remark}

With little effort, we can deduce a slightly improved version of Theorem~\ref{th:mainSPECTRUMgraph}, whose proof is presented in Appendix \ref{appendix}.

\begin{theorem}\label{th:k-irr+}
If condition \eqref{e:mucca} holds,
then $\{G_n\}$ is fully irreducible.
\end{theorem}

\section{Irreducibility and hypergraph spectra}\label{sec:Hypergraphs}
{In this section we derive a generalization of Theorem~\ref{th:mainSPECTRUMgraph}, extending our results to the setting of homogeneous sums of arbitrary order $d \geq 2$, possibly involving non-trivial coefficients $q_n$. To this end, we broaden the spectral framework developed in Section~\ref{sec:GRAPHS} by moving to the setting of weighted hypergraphs, partially following the approach of \cite{SSPHypergraphs, B21}. The main result of this section, Theorem~\ref{th:mainresulthyper}, strictly contains Theorem~\ref{th:mainSPECTRUMgraph} as a particular case. Nevertheless, we chose to treat the case $d = 2$ separately in order to highlight its specific features and maintain clarity.}

\subsection{Hypergraphs and extended Cheeger's inequalities}
{A \emph{hypergraph} $\cG$ is a pair $\cG = (V, \cE)$ consisting of a finite vertex set $V$ and a collection $\cE$ of non-empty subsets of $V$, referred to as \emph{hyperedges} (or simply \emph{edges} when the context is clear). To accommodate the most general setting, we consider \emph{weighted} hypergraphs, where a non-negative weight function $w = \{ w(e) : e \in \cE \}$ is assigned to the edges of $\cG$. We exclude the presence of loops, and therefore assume that $|e| \geq 2$ for all $e \in E$. For $d\geq 2$, we say that $\cG$ is $d$-{\em uniform} if $|e| = d$ for every $e\in \cE$.

It is evident that the notion of a hypergraph generalizes that of a standard graph, which corresponds to the special case of 2-uniform hypergraphs, that is, hypergraphs where each edge connects exactly two vertices. As in the case of graphs, two vertices $v, w \in V$ are said to be \emph{adjacent}, written $v \sim w$, if there exists an edge $e \in \cE$ such that $\{v, w\} \subseteq e$.

Consider a weighted hypergraph $\cG = (V, \cE, w)$ with $V = \{v_1, \ldots, v_N\}$. Our goal is to define an extended version of the adjacency matrix associated with a graph, capable of capturing the connectivity structure of $\cG$. To this end, we follow the approach of \cite{B21}, which deals with the unweighted case where $w \equiv {1}$. Writing $\cE_{ij}\coloneq \{ e \in \cE : v_i,v_j \in e \}$, $i,j=1,...,N$, the \emph{adjacency matrix} $A_\cG$ associated with $\cG$ is the $N \times N$ matrix that is zero on the diagonal, and otherwise defined as
\begin{equation}\label{eq:adjacencymatrix}
    (A_\cG)_{ij} \coloneq \sum_{e \in \cE_{ij}} \frac{w(e)}{|e|-1}\,, \quad i \neq j\in [N]\,.
\end{equation}
{\it Heuristic interpretation of} \eqref{eq:adjacencymatrix}: Fix $v_i \in V$ and assign mass $1$ to every edge containing it. Assuming $w \equiv {1}$, the adjacency matrix $A_\cG$ quantifies the connectivity of $v_i$ by distributing the unit mass of each edge $e \ni v_i$ uniformly among the other vertices $v_j \in e \setminus \{v_i\}$. 

\smallskip 

As a consequence of the choice of normalization in \eqref{eq:adjacencymatrix}, if one sums the elements of the $i$th row (or column) of $A_\cG$, one obtains exactly the total number of edges that include $v_i$ as an element, a quantity that corresponds to the {\em degree} (or {\em weighted degree}, if $w\ne {1}$) $d(v_i)$ of $v_i$. Such a quantity is defined as
\begin{equation}\label{e:hyperdegree}
    d(v_i)= \sum_{j=1}^N  (A_\cG)_{ij} = \sum_{\substack{j=1\\ j\ne i}}^N \sum_{e \in \cE_{ij}} \frac{w(e)}{|e|-1}
     = \sum_{e \in \cE \colon e \ni v_i} w(e) \,,\qquad v_i \in V\,.
\end{equation}
See Figure \ref{fig:hyper} for some examples. Notice that, when $\cG = G$ is a graph, the expression \eqref{eq:adjacencymatrix} coincides with the standard definition of the adjacency matrix associated with $G$.

\begin{figure}[t]
\centering

\begin{minipage}[t]{0.31\textwidth}
\centering
\begin{tikzpicture}[scale=0.7, yscale=1.2]

    \foreach \i/\x in {1/0, 2/1.5, 3/3, 4/4.5} {
        \filldraw[black] (\x,0) circle (2pt);
        \node at (\x,-0.4) {\small \i};
    }

    \draw[red, thick] (0,0) to[bend left=35] (1.5,0);
    \draw[red, thick] (1.5,0) to[bend right=35] (3,0);

    \draw[blue, thick] (1.5,0) to[bend left=35] (3,0);
    \draw[blue, thick] (3,0) to[bend right=35] (4.5,0);

\end{tikzpicture}

\vspace{2pt}
\tiny {\bf (a)} \[ A_\cG = 
\begin{bmatrix}
0 & 1/2 & 1/2 & 0 \\
1/2 & 0 & 1 & 1/2 \\
1/2 & 1 & 0 & 1/2 \\
0 & 1/2 & 1/2 & 0
\end{bmatrix}
\]
\end{minipage}%
\hfill
\begin{minipage}[t]{0.31\textwidth}
\centering
\begin{tikzpicture}[scale=0.7, yscale=1.2]

    \foreach \i/\x in {1/0, 2/1.5, 3/3, 4/4.5} {
        \filldraw[black] (\x,0) circle (2pt);
        \node at (\x,-0.4) {\small \i};
    }

    \draw[orange, thick] (0,0) to[bend right=30] (1.5,0);
    \draw[orange, thick] (1.5,0) to[bend left=30] (3,0);
    \draw[orange, thick] (3,0) to[bend right=30] (4.5,0);

    \draw[blue, thick] (0,0) to[bend left=35] (1.5,0);

    \draw[red, thick] (1.5,0) to[bend right=35] (3,0);
    \draw[red, thick] (3,0) to[bend left=35] (4.5,0);

\end{tikzpicture}

\vspace{2pt}
\tiny {\bf (b)} \[ A_\cG = 
\begin{bmatrix}
0 & 4/3 & 1/3 & 1/3 \\
4/3 & 0 & 5/6 & 5/6 \\
1/3 & 5/6 & 0 & 5/6 \\
1/3 & 5/6 & 5/6 & 0
\end{bmatrix}
\]
\end{minipage}%
\hfill
\begin{minipage}[t]{0.31\textwidth}
\centering
\begin{tikzpicture}[scale=0.7, yscale=1.2]

    \foreach \i/\x in {1/0, 2/1.5, 3/3, 4/4.5} {
        \filldraw[black] (\x,0) circle (2pt);
        \node at (\x,-0.4) {\small \i};
    }
\draw[orange, thick] (0,0) to[bend left=35] (3,0);
    \draw[orange, thick] (3,0) to[bend right=47] (4.5,0);

    \draw[blue, thick] (0,0) to[bend right=47] (1.5,0);
    \draw[blue, thick] (1.5,0) to[bend left=35] (4.5,0);

    \draw[red, thick] (0,0) to[bend right=20] (1.5,0);
    \draw[red, thick] (1.5,0) to[bend left=20] (3,0);
    \draw[red, thick] (3,0) to[bend right=20] (4.5,0);

\end{tikzpicture}

\vspace{2pt}
\tiny {\bf (c)} \[ A_\cG = 
\begin{bmatrix}
0 & 5/6 & 5/6 & 4/3 \\
5/6 & 0 & 1/3 & 5/6 \\
5/6 & 1/3 & 0 & 5/6 \\
4/3 & 5/6 & 5/6 & 0
\end{bmatrix}
\]
\end{minipage}

\caption{\footnotesize Three non-weighted hypergraphs ($w\equiv1$) over the vertex set $V = [4]$, endowed with their adjacency matrices. Each hyperedge is represented as a bundle of 2-edges with the same colour, so that, noting ${ d} = (d(1),d(2), d(3), d(4))$ the corresponding degree vector, one has the following configurations: {\bf (a)} $\cE = \{\{1,2,3\}, \{2,3,4\}\}$, ${ d} = (1,2,2,1)$; {\bf (b)} $\cE = \{\{1,2\},\{1,2,3,4\}, \{2,3,4\}\}$, ${ d} = (2,3,2,2)$; {\bf (c)} $\cE = \{\{1,2,4\}, \{1,2,3,4\}, \{1,3,4\}\}$,  ${d} = (3,2,2,3)$. The hypergraph in {\bf (a)} is 3-uniform.  }
\label{fig:hyper}
\end{figure}
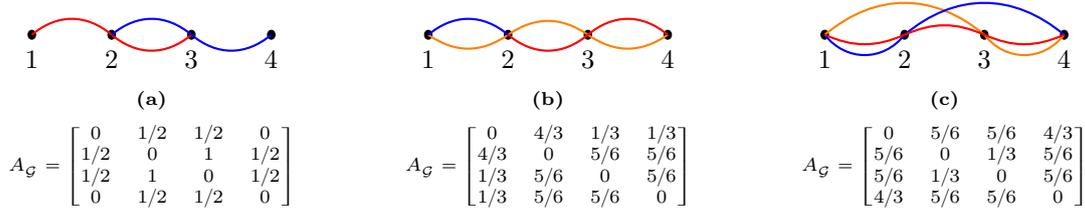

\smallskip 

As announced, we can now generalize the definitions introduced in Section \ref{sec:GRAPHS}. 
Denoting by $D$ the diagonal matrix with degrees $d(v_1),\ldots,d(v_N)$ as entries, we define the {\em Laplacian matrix} $L_{\cG} := D - A_{\cG}$ and the {\em normalized Laplacian matrix} associated with the hypergraph $\cG$ as $\cL_\cG := I - D^{-\frac{1}{2}} A_{\cG} D^{-\frac{1}{2}}$.
We observe that, for non-zero functions ${\bf g}=(g(v_1),\ldots, g(v_N)) \in \R^V$ and ${ \bf f}=(f(v_1),\ldots, f(v_N)):= D^{-\frac{1}{2}}{ \bf g}$, the \emph{Raleygh quotient} associated with $\cL_\cG$ at ${\bf f}$ is given by the ratio
\begin{equation*}
\begin{split}
    \frac{\langle {\bf g}, \cL_{\cG} {\bf g}\rangle }{\langle {\bf g},  {\bf g} \rangle }&=\frac{ \langle {\bf g},D^{-\frac{1}{2}}L_{\cG} D^{-\frac{1}{2}}{\bf g}\rangle}{\langle {\bf g},{\bf g}\rangle}=\frac{\langle {\bf f}, L_{\cG} {\bf f} \rangle }{\langle D^{\frac{1}{2}} {\bf f},  D^{\frac{1}{2}} {\bf f} \rangle }\\
    &=\frac{\sum_{v_i \sim v_j} \sum_{e \in \cE_{ij}} \frac{w(e)}{|e|-1} \,\big( f(v_i)-f(v_j) \big)^2}{\sum_{v_i\in V}d(v_i)f(v_i)^2} =: \cR_\cG( {\bf f})\,.
    \end{split}
\end{equation*}
In particular, recalling the variational characterization of the eigenvalues $\mu_1\leq \cdots \leq \mu_N$ of $\cL_{\cG}$ (which is valid in our case since $\cL_{\cG}$ is symmetric; see Remark \ref{r:minmax}), one has that
\begin{equation}\label{eq:eingenLhyper}
    \mu_k = \min_{\substack{\mathcal{M}\subseteq \R^V \\ {\rm dim}(\mathcal{M})=k }}\max_{{\bf f}\in \mathcal{M}\backslash \{0\}} R_{\cL_\cG}({\bf f})\,, \quad \forall k=1,\ldots,N,
\end{equation}
from which one infers that $0=\mu_1\leq \mu_2\leq \cdots \leq \mu_N\leq 2$.
\medskip

Given $\emptyset \neq S\subseteq V $, we define 
\begin{equation}
    \label{eq:E}
\partial S \coloneq \big\lbrace \, e \in \cE : \exists \, v_i,v_j \in e \text{ with }   v_i \in S, \ v_j \in \bar S \, \big\rbrace
\end{equation}
to be the set of edges with at least one vertex in $S$ and one in its complement $\bar S \coloneq V \setminus S$. Moreover, we define the {\em volume} of $S$ as
\begin{equation}\label{eq:volumeS}
    \text{vol}(S)=w(S)\coloneq \sum_{v \in S}d(v)=\sum_{v \in S} \sum_{e \in \cE_v} w(e)\,,
\end{equation}
where $\cE_v \coloneq \{ e \in \cE : v \in e\}$, while for $\cF \subseteq \cE$ we set \begin{equation}\label{eq:volumeF}
    w(\cF) \coloneq \sum_{e \in \cF} w(e)\,.
\end{equation}
The {\em edge expansion} of a nonempty subset $S \subseteq V$ is given by
\begin{equation}\label{eq:edgeexp}
\phi(S) :=  \frac{w(\partial S)}{\text{vol}(S)}\,,
\end{equation}
and for $k\geq 2$, we set
$$\phi_k( \cG) := \min_{\substack{S_1,...,S_k\\ \mbox{\tiny nonempty and disjoint}}} 
\max_{i=1,...,k}\phi(S_i).
$$

\begin{example}{\rm Consider the examples in Figure \ref{fig:hyper}, and set $S = \{1,2\}$, $\bar{S} = \{3,4\}$. Then one has the following computations: {\bf (a)} $\varphi(S) = \varphi(\bar{S}) = 2/3$; {\bf (b)} $\varphi(S) = 2/5$, $\varphi(\bar{S}) = 1/2$; {\bf (c)} $\varphi(S) = \varphi(\bar{S}) = 3/5$.

}
\end{example}

It is now possible to derive an analogue of the ``easy direction'' of the Cheeger inequality of order~$k$ (see Proposition \ref{p:easycheeger}) for hypergraphs. To this end, we take inspiration from \cite[Theorems~3.3 and~4.1]{B21}, where a bound is established for the hypergraph analogue of $\tilde{\varphi}(G)$ (see~\eqref{e:tilde}) in the case $k=2$.

\begin{proposition}\label{prop:cheegerhyp}
    Let $\cG =(V,\cE,w)$ be a weighted hypergraph as above. Then, for all $k=2,\ldots,N$,
    \begin{equation}\label{eq:easyCheeger}
      \mu_k \le \frac{2 \big( r(\cG)-1\big)^2}{cr(\cG)-1} \, \varphi_k (\cG)  \,,
    \end{equation}
    where $r(\cG)$ and $cr(\cG)$ are called, respectively, the {\em rank} and {\em co-rank} of $\cG$ and correspond to the maximum and the minimum of the cardinalities of the edges $e \in \cE$.
\end{proposition}
Note that the inequality \eqref{eq:easyCheeger} reduces to \eqref{eq:cheegereasydir} in the graph setting, where $r(\cG)=cr(\cG)=2$. We outline the main steps of the proof, which extends to the case of hypergraphs the strategy used in the proof of Proposition \ref{p:easycheeger}.

\begin{proof}[Proof of Proposition \ref{prop:cheegerhyp}]
    Let $S_1,...,S_k \subseteq V$ be non--empty disjoint subsets of $V$. For $\ell =1,\ldots,k$, let ${\bf f}_\ell\coloneq (\text{vol}(S_\ell))^{-\frac{1}{2}}{\bf 1}_{S_\ell}$, where ${\bf 1}_{S_\ell} \coloneq (\ind_{S_\ell}(v_1),\ldots,\ind_{S_\ell}(v_N))$. For all not identically zero vectors $(\alpha_1,...,\alpha_k)$, we have
    \begin{equation*}
    \begin{split}
        R_\cG \bigg( \sum_{\ell=1}^k \alpha_\ell { \bf f_\ell} \bigg) &= \frac{\sum_{e \in \cE} \sum_{\substack{v_i \sim v_j\\ v_i,v_j \in e}}\frac{w(e)}{|e|-1} \bigg( \sum_{\ell=1}^k\frac{\alpha_\ell}{\sqrt{\text{vol}(S_\ell)}} \big(\ind_{S_\ell}(v_i)-\ind_{S_\ell}(v_j) \big)\bigg)^2}{\sum_{v_i \in V}d(v_i) \bigg( \sum_{\ell=1}^{k}\frac{\alpha_\ell}{\sqrt{\text{vol}(S_\ell)}} \ind_{S_\ell}(v_i)\bigg)^2}\\
        &\le \frac{2 \sum_{\ell=1}^k \frac{\alpha_\ell^2}{\text{vol}(S_\ell)} \sum_{e \in \cE} \sum_{\substack{v_i \sim v_j\\ v_i,v_j \in e}} \frac{w(e)}{|e|-1}\big(\ind_{S_\ell}(v_i)-\ind_{S_\ell}(v_j) \big)^2}{\sum_{\ell=1}^k \alpha_\ell^2}\,,
    \end{split}
    \end{equation*}
    where the last inequality is a consequence of the fact that $S_1,\ldots,S_k$ are pairwise disjoint, hence the sum over $\ell$ contains at most two non-zero terms. Observe indeed that for any $\ell=1,\ldots,k$, the term $\big(\ind_{S_\ell}(v_i)-\ind_{S_\ell}(v_j)\big)$ does not vanish if and only if exactly one of the vertices $v_i$ and $v_j$ belongs to $S_\ell$. As a consequence, we can rewrite
    \begin{equation*}
    \begin{split}
        \sum_{v_i \sim v_j} \sum_{e \in \cE_{ij}}\big(\ind_{S_\ell}(v_i)-\ind_{S_\ell}(v_j) \big)^2 =\sum_{e \in \partial S_\ell} \sum_{\substack{v_i \sim v_j\\ v_i,v_j \in e}}\big(\ind_{S_\ell}(v_i)-\ind_{S_\ell}(v_j) \big)^2 = \sum_{e \in \partial S_\ell} |e \cap S_\ell | \, | e \cap \widebar{S_\ell} |\,,
    \end{split}
    \end{equation*}
    and, for any $e \in \partial S_\ell$, we have $|e \cap S_\ell | \, | e \cap \widebar{S_\ell} | \le (|e|-1)^2 \le (r(\cG)-1)^2$. Using the bound $|e|-1 \ge cr(\cG)-1$, we eventually obtain that
    \begin{equation*}
    \begin{split}
        R_\cG \bigg( \sum_{\ell=1}^k \alpha_\ell { \bf f_\ell} \bigg) &\le \frac{ 2 (r(\cG)-1)^2}{cr(\cG)-1} \frac{\sum_{\ell=1}^k \alpha_\ell^2 \varphi(S_\ell)}{\sum_{\ell=1}^k \alpha_\ell^2} \le \frac{ 2 (r(\cG)-1)^2}{cr(\cG)-1} \max_{\ell=1,\ldots,k} \varphi (S_\ell)\,.
    \end{split}
    \end{equation*}
The conclusion now follows from \eqref{eq:eingenLhyper}, following the same route as in the graph setting.
\end{proof}

As a consequence of Proposition \ref{prop:cheegerhyp}, we state an estimate analogous to\eqref{e:simple} for hypergraphs, that one can deduce from arguments similar to those in the proof of Proposition~\ref{prop:simple}.

\begin{proposition} Let the above assumptions prevail, and fix integers $1\leq m\leq N$. Consider a collection $B_1,...,B_m$ of nonempty disjoint subsets of $V$, and assume that
\begin{equation}\label{e:orderinghyp}
\phi(B_1)\le \phi(B_2)\le \cdots \le \phi(B_m). 
\end{equation}
Then, for all $k=1,...,m$,
\begin{equation}\label{e:simplehyp}
 \mu_k \leq \frac{2 \big(r(\cG)-1\big)^2}{ cr(\cG)-1}\, \frac{\sum_{i=k}^m \mathrm{vol} (B_i) \, \phi(B_i)}{\sum_{i=k}^m \mathrm{vol}(B_i)}  . 
\end{equation}
\end{proposition}
}

\subsection{Irreducibility}
{The aim of this section is to state a generalization of Theorem \ref{th:mainSPECTRUMgraph} to the setting of homogeneous sums of a generic order $d\geq 2$. To this end, we consider a sequence $\left\{(V_n,E_n,q_n) : n\geq 1\right\}$ such that, for a fixed $d \ge 2$: (i) $\{V_n\}$ verifies \eqref{eq:propV}, (ii) $\{E_n\}$ satisfies the subsequent requirements \eqref{item:VE1}--\eqref{item:VE3}, and (iii) the symmetric coefficients $\{q_n\}$ verify \eqref{e:varinfty}. Following the convention \eqref{e:weights}, we continue to use the notation $w_n(\cdot) = q_n(\cdot)^2$. 

\smallskip

For each $n\geq 1$, the triple $(V_n,E_n,q_n)$ is canonically associated with the weighted hypergraph $\cG_n = (V_n, \mathcal{E}_n, w_n)$ such that\footnote{Due to the symmetry of $E_n$, the way in which the elements of a given $d$-subset $\{v_1,...,v_d\}$ are enumerated is immaterial. } 
\begin{eqnarray}\label{e:can1}
&& \mathcal{E}_n = \{ \{v_1,...,v_d\}\subset V_n : (v_1,...,v_d)\in E_n\},\quad \mbox{and}\\ \label{e:can2}
  && w_n(\{v_1,...,v_d\}) = w_n(v_1,...,v_d) = q_n(v_1,...,v_d)^2.  
\end{eqnarray}
 Note that, by construction, each $\mathcal{G}_n$ is $d$-uniform, and consequently $r(\cG_n)=cr(\cG_n)=d$ (see Proposition \ref{prop:cheegerhyp}). Also, one has that $|E_n| = d!|\mathcal{E}_n|$. For $n\ge 1$, the Laplace spectrum of $\cG_n$ is written $0 = \mu_1^{(n)}\leq \mu_2^{(n)}\le \cdots \le \mu_{N_n}^{(n)}\leq 2$.

\smallskip 

As before, we denote by $\{Z_n\} = \{Z_n({\bf X}) : n\geq 1\} $ the sequence of homogeneous sums defined in \eqref{eq:defpolchaos}, and we observe that the second relation in \eqref{eq:secmomZn} implies that
\begin{equation}\label{e:hypvariance}
\mathbb{E}[Z_n^2] = d!^2 w_n(\cE_n),
\end{equation}
where we have used the notation \eqref{eq:volumeF} in the case $\mathcal{F} = \mathcal{E}_n$.

\medskip
In order to generalize Lemma \ref{def:irreducibility22} 
in the context of hypergraphs, 
we define the set of edges entirely contained in a nonempty subset $S \subseteq V$ as follows:
\begin{equation}\label{eq:edgesinS}
   \begin{split}
        \cE_n (S,S) &:= \big\lbrace \, e \in \cE_n : e \subset S \, \big\rbrace.
   \end{split}
\end{equation}

The following statement is the exact analogous of Lemma \ref{def:irreducibility22}. The proof is left to the reader.

\begin{lemma}[Reducibility and hypergraphs]\label{def:irreducibility2}
 Let the above assumptions and conventions prevail. Then, the sequence $\{Z_n\}$ is reducible in the sense of Definition \ref{def:irreducibility} if and only if there exist partitions $\{B_1,\ldots,B_{m_n}\} = \{B^{(n)}_1,\ldots,B^{(n)}_{m_n}\}$ of the vertex sets $V_n$, such that, as $n \to \infty$,
    \begin{enumerate}[label={\rm (\roman*)}]
    \item \label{item:irr1hyp} $m_n\to \infty$;
    \item \label{item:irr2hyo} $\sum_{i=1}^{m_n} w_n( \cE(B_i,B_i)) \sim w_n(\mathcal{E}_n)$;
    \item \label{item:irr3hyp} $\max_{i=1,...,m_n} w_n( \cE(B_i,B_i)) = o\big( w_n(\mathcal{E}_n)\big)$.  
\end{enumerate}
\end{lemma}

\begin{definition}\label{def:irreducibileHG}{\rm Let $\cG_n = (V_n, \mathcal{E}_n, w_n)$, $n\geq 1$, be a sequence of weighted hypergraphs such that $|V_n|,\, |\cE_n|,\,  w_n(\cE_n)\to \infty$. Consider a sequence $\{Z_n\}$ defined as in \eqref{eq:defpolchaos} such that \eqref{e:can1}--\eqref{e:can2} are verified. We say that $\{Z_n\}$ is a {\em sequence of homogeneous sums associated with} $\{\cG_n\}$. As before, we say that $\{\cG_n\}$ {\em generates an irreducible CLT}, if $\{Z_n\}$ verifies an irreducible CLT in the sense of Definition \ref{def:irreducibility}. If $\{Z_n\}$ is reducible, we will say that $\{\cG_n\}$ {\em is reducible} or, more precisely, that $\{\cG_n\}$ {\em generates a reducible CLT}. 
 }
\end{definition}

\begin{remark}{\rm A natural open question is whether, for $d \geq 3$, one can obtain a spectral characterization of Condition~\ref{item:cond2thm3.1} in Theorem~\ref{t:unidejong}, analogous to the necessary and sufficient condition~\eqref{e:degreeo} established in the case $d=2$, but now based on the spectral analysis of hypergraphs developed in the previous section. We refer the reader to the breakthrough work by Herry, Malicet, and Poly~\cite{HeMaPolyLaw} for significant progress in this direction.

}
\end{remark}

We state the main result of the section, which establishes the sufficient condition for irreducibility.

\begin{theorem}\label{th:mainresulthyper}
Fix $d\geq 2$ and let $\{Z_n\}$ be a sequence of homogeneous sum as in \eqref{eq:defpolchaos}. Let $\cG_n=(V_n,\cE_n,w_n)$, $n\geq 1$, be the sequence of weighted hypergraphs associated to $\{Z_n\}$ via \eqref{e:can1}--\eqref{e:can2}. If there exists $k \ge 2$ such that, as $n\to\infty$,
    \begin{equation}
        \label{eq:liminfhyp}
        \liminf_{n} \mu_k^{(n)} > 0\,,
    \end{equation}
then $\{Z_n\}$ is irreducible in the sense of Definition \ref{def:irreducibility}.
\end{theorem}
The proof of Theorem \ref{th:mainresulthyper} is provided in Appendix \ref{appendix}: it consists of a technical generalization of the proof of Theorem \ref{th:mainSPECTRUMgraph}.

\smallskip

\begin{example}[Rook-like hypergraphs]\label{ex:rooklike} {\rm For every $d\geq 3$, we will now build an example of a sequence of homogeneous sums $\{Z_n\}$, of order $d\geq 3$, and such that (i) they verify an irreducible CLT, (ii) irreducibility follows from Theorem \ref{th:mainresulthyper} via the fact that the adjacency matrices \eqref{eq:adjacencymatrix} of the associated weighted hypergraphs (via \eqref{e:can1}--\eqref{e:can2}) coincide with those of the Rook's graph discussed in Example \ref{item:completegraph}-\ref{item:ci} of Section \ref{subsec:examples}. To this end, for every $n > d$ large enough, define $\cG_n = (V_n, \cE_n, w_n)$ to be the hypergraph such that $V_n = [n]^2$, $\cE_n$ is given by all $d$-subset with the form
$$
e = \{(a,b_1),(a,b_2),..., (a,b_d)\} \quad \mbox{or}\quad e = \{(b_1,a),(b_2,a),..., (b_d,a)\}
$$
with $a\in [n]$ and $b_1,...,b_d$ distinct, and
we choose a constant weight (for convenience, we fix the weight so that it yields the same adjacency matrix of the Rook's graph)
$$
w_n(e) \equiv (d-1)  \binom{n-2}{ d-2}^{-1}, \quad e\in \mathcal{E}_n.
$$
It is easily seen that two vertices $v_1= (a_1,b_1), \, v_2=(a_2,b_2)$ are adjacent in $\cG_n$ if and only if they are adjacent in the Rook's graph and, in this case, they are both contained in exactly $\binom{n-2}{d-2}$ edges. The choice of $w_n$ then ensures that the adjacency matrix of $\cG_n$, as given in \eqref{eq:adjacencymatrix}, coincides with that of the Rook's graph. Writing $0 = \mu_1^{(n)}\leq \mu_2^{(n)} \leq \cdots \leq \mu_{n^2}^{(n)}$ for the normalized Laplace eigenvalues of $\cG_n$, one therefore has that $\mu_k^{(n)}\longrightarrow \frac12$ for all $k\geq 2$, and Theorem \ref{th:mainresulthyper} implies that any associated sequence of homogeneous sums
$\{Z_n\}$ (via \eqref{e:can1}--\eqref{e:can2}) is irreducible. We observe that
$$
|\cE_n| = 2 n\binom{n}{d} \asymp n^{d+1}, \quad n\to\infty. 
$$
We now define $\{Z_n\}$ according to \eqref{eq:defpolchaos}, with 
$$
q_n(v_1,...,v_d) = \begin{cases}
\sqrt{d-1}\binom{n-2}{d-2}^{-1/2}, & \{v_1,...,v_d\}\in \mathcal{E}_n \\
0,  & \text{otherwise. }
\end{cases}
$$
In this way, $\mathbb{E}[Z_n^2]\asymp n,$ and a direct computation shows that the numerical sequences defined in \eqref{e:qn} scale as $O(n^{-1})$ for every $r=1,...,d-1$, and thus $\{Z_n\}$ verifies an irreducible CLT.}
    
\end{example}

\begin{example}[3-uniform hypergraphs generating an irreducible CLT]\label{e:3uni}{\rm We present the following example as an application of Theorem \ref{th:mainresulthyper}. Consider the vertex set 
\begin{equation*}
    V_n = \{ (a,b) : 1 \le a \ne b \le n \} \subseteq \N \times \N
\end{equation*}
in such a way that $N_n=|V_n|=n(n-1)$. For every $n$, we define the symmetric and non-diagonal set \( \cE_n \subseteq (V_n)^3 \) consisting of all ordered triples \( (v_1, v_2, v_3) \), where each \( v_i = (a_i, b_i)\) is an element of $V_n$, satisfying the following conditions:
\begin{itemize}
  \item For every pair \( \ell \ne k \), the entries \( v_\ell \) and \( v_k \) share exactly one coordinate:  
  \( |\{a_\ell, b_\ell\} \cap \{a_k, b_k\}| = 1 \).
  \item The number of distinct labels among \( \{a_1, b_1, a_2, b_2, a_3, b_3\} \) equals 3.
\end{itemize}
Note that $E_n$ is the disjoint union of eight sets $E_n^{(1)},\ldots,E_n^{(8)}$, each one with a similar structure and of the same size as
\begin{equation*}
    E_n^{(1)}= \Big\lbrace \,\big( (a,b)\,, \, (b,c)\,, \, (c,a)\big) : 1 \le a \ne b \ne c \ne a \le n\,\Big\rbrace\,.
\end{equation*}
Thus, \( |E_n| = 8n(n-1)(n-2) \), and the hypergraph \( \cG_n = (V_n, \cE_n) \), obtained from $\{E_n\}$ by setting $q_n\equiv 1$ and using \eqref{e:can1}, is unweighted and 3-uniform. For $n\geq 3$, we consider the homogeneous sum
$$
Z_n := \sum_{v_1,v_2,v_3 \in V_n} \ind_{E_n}(v_1,v_2,v_3) \, X_{v_1}  X_{v_2} X_{v_3},
$$
whose variance is commensurate to $n^3$, as $n\to \infty$. To determine the irreducibility of $\{Z_n\}$, we observe that $\cG_n=(V_n,\cE_n)$ has adjacency matrix
\begin{equation*}
    (A_{\cG_n})_{ij} = \sum_{e \in \cE_{ij}} \frac{1}{|e|-1}=\frac{1}{2} \, \big|  \{ \,e \in \cE_n : v_i,v_j \in e \, \} \big| \,, \quad v_i, \, v_j \in V_n\,, \ i \ne j\,.
\end{equation*}
In particular, when $v_i$ and $v_j$ are adjacent, say $v_i=(a,b)$ and $v_j=(b,c)$, they are contained in exactly two edges, which have $(a,c)$ and $(c,a)$ as their third vertices. This implies that $|\cE_{ij}|=2$, thus $(A_{\cG_n})_{ij} = \ind_{ \{v_i \sim v_j\}}$ for every $v_i,\, v_j \in V_n$, where the adjacency relation $\sim$ is given by
\begin{equation*}
    v_i \sim v_j \ \iff \ v_i=(a,b)\,, \, v_j=(a',b') \ \text{ and either } \
    \begin{cases}
        a=a'\,, \ b \ne b'\,, \text{ or}\\
        b=b'\,, \ a \ne a'\,, \text{ or}\\
        a=b'\,, \ a' \ne b\,, \text{ or}\\
        a'=b\,, \ a \ne b'\,.\\
    \end{cases}
\end{equation*}
One can verify that this equivalence relation induces a graph structure on the vertex set \( V_n \), corresponding to the union of two graphs: the Rook’s graph (see Example~\ref{e:ridux}-\eqref{it:ridux-rook}) with the diagonal removed, and a second graph—also isomorphic to the diagonal-free Rook’s graph—obtained by permuting the two coordinates of each vertex. A direct computation yields that the degree of each vertex is $d_n = 4(n-2)$, and the spectrum of $A_{\cG_n}$ is given by the integers $\{ \, 4(n-2), \, (2n-8), \, 0, \, -4\} $ with multiplicities $1, \, (n-1), \, n(n-1)/2$ and $-1+(n-1)(n-2)/2$. As a consequence, the Laplace spectrum of $\cG_n$ is given by $\{0, \, n/(2(n-2)), \, 1, \, (n-1)/(n-2) \} $, with the same respective multiplicities. Since, for all fixed $k\geq 2$, one has that $\lim_{n\to\infty}\mu_k^{(n)} = \frac12$, we deduce from Theorem~\ref{th:mainresulthyper} that $\{Z_n\}$ is irreducible. Finally, we remark that this example will be revisited in  the subsequent section, where we will show that the hyperedge sets $\{\cE_n\}$  can be realized as a special cases of a {\em fractional Cartesian product}, as defined in Example~\ref{e:fcp}. As such, the asymptotic normality of the (normalized) sequence \( \{Z_n\} \) will follow directly from Proposition~\ref{cor:fracprod}.
}
\end{example}

}

\section{Irreducibility via combinatorial dimensions}\label{sec:sparsity}
{ 
The aim of this section is to prove Theorem \ref{th:irreducombdim}. 
The proof of Part~\ref{item:partB} (see Section \ref{ss:proofb}) is based on an explicit construction, described in full detail in Example~\ref{e:fcp} below. The arguments used in this part are purely combinatorial and do not rely on the spectral analysis developed in the previous sections. We note, however, that one special instance of Example~\ref{e:fcp} was already addressed in Section~\ref{sec:Hypergraphs} by means of hypergraph techniques (see Example~\ref{e:3uni}).

From now on, we let the assumptions and notation in the statement of Theorem \ref{th:irreducombdim} prevail. 

\subsection{Proof of Part \ref{item:partA} of Theorem \ref{th:irreducombdim}}\label{ss:proofa}

Assume that $\{E_n\}$ has combinatorial dimension $1<\alpha\leq d$. We have to show that it is not possible to find a sequence $\Pi_n \coloneq  \{ B_1,\ldots,B_{m_n} \}$, $n\ge 1$, with $B_i=B_i(n)$, such that each $\Pi_n$ is a partition of $V_n$, and, as 
$n\to \infty$, 
\begin{eqnarray}
 \max_{i=1,...,m_n} \Big| E_n \cap (\underbrace{B_i \times \cdots \times B_i}_\text{$d$ times}) \Big| &=& o ( |E_n|), \label{e:pre1} \\
\sum_{i=1}^{m_n} \Big| E_n \cap (\underbrace{B_i \times \cdots \times B_i}_\text{$d$ times}) \Big| &\asymp& |E_n| \label{e:pre2}\,,
\end{eqnarray}
(note that, if \eqref{e:pre1} and \eqref{e:pre2} are both verified, then, necessarily, $m_n\to\infty$).

\noindent We reason by contradiction, and assume that there exists a sequence of partitions $\{\Pi_n \}$ such that the two properties \eqref{e:pre1} and \eqref{e:pre2} are satisfied. Without loss of generality, we can assume that $|B_1|\geq |B_2|\geq \cdots \geq |B_{m_n}|$, $n\geq 1$. For all $\beta >0$, we set
$$
K_n(\beta) := \max\left\{s\leq m_n : |B_s|\geq \beta |V_n| \right\}, 
$$
with $\max \emptyset :=0$. Since each $\Pi_n$ is a partition of $V_n$ we have that
$$
|V_n| = \sum_{s=1}^{m_n} |B_s|\geq \sum_{s=1}^{K_n(\beta)} |B_s^{(n)}|\geq K_n(\beta) \beta |V_n|,
$$
where we have used the convention $\sum_{s=1}^0 := 0$, thus yielding the bound
\begin{equation}\label{e:kbound}
K_n(\beta) \leq \frac{1}{\beta}, \quad \mbox{for all }\, \beta>0.
\end{equation}
In the sequel, we write $M(\beta) := \lceil \beta^{-1} \rceil +1$, $\beta>0$. We will also use the following estimate, valid for every $n\geq 1$ and every $t\leq m_n$:
\begin{equation}\label{e:sorey}
\sum_{s=t}^{m_n} \big| B_s\big|^{\alpha}\leq \big| B_t\big|^{\alpha-1} \sum_{s=t}^{m_n} \big| B_s\big| \leq \big| B_t\big|^{\alpha-1} \, |V_n|,
\end{equation}
where the first inequality uses the fact that $\alpha\in (1,d]$ and that the sequence $s\mapsto \big|B_s\big|$ is decreasing. Since \eqref{e:pre2} is in order and $\{E_n\}$ has combinatorial dimension $\alpha$, using \eqref{e:alphasymp} for $J_n=E_n$ one infers that there exist a finite constant $\Gamma>0$ and an integer $n_0\geq 1$ such that, for all $n\geq n_0$,
\begin{eqnarray*}
 \Gamma &>& \frac{|V_n|^\alpha}{\sum_{s=1}^{m_n} | E_n \cap (B_s\times \cdots \times B_s)|} \geq  \frac{|V_n|^\alpha}{c\sum_{s=1}^{m_n} \big| B_s\big|^{\alpha}} \geq \frac{|V_n|^{\alpha-1} }{c \big| B_1\big|^{\alpha-1}},
\end{eqnarray*}
where the constant $c$ is the one appearing on the right-hand side of \eqref{e:req1} (choosing $J_n=E_n$), and the last estimate exploits \eqref{e:sorey} in the case $t=1$. This implies that, for $n\geq n_0$, $ \big|B_1\big| \geq b|V_n| $,
where $b := (c\, \Gamma)^{-\frac{1}{\alpha-1}}$.
Now, assumption \eqref{e:pre1} combined once again with \eqref{e:alphasymp} (with $J_n=E_n$) yields that, as $n\to \infty$,
$$
\frac{1}{|V_n|^\alpha} \sum_{s=1}^{M(b) -1 } | E_n \cap (B_s\times \cdots \times B_s)|
\longrightarrow 0. $$
From this we infer that there exists an integer $n_1\geq n_0$ such that, for all $n\geq n_1$
\begin{eqnarray*}
 \Gamma &>& \frac{|V_n|^\alpha}{\sum_{s=M(b)}^{m_n} | E_n \cap (B_s\times \cdots \times B_s)|} \geq  \frac{|V_n|^\alpha}{c\sum_{s=M(b)}^{m_n} \big| B_s\big|^{\alpha}} \geq \frac{|V_n|^{\alpha-1} }{c \big| B_{M(b)}\big|^{\alpha-1}},
\end{eqnarray*}
where the last inequality uses \eqref{e:sorey} in the case $t=M(b)$. These relations imply that, for $n\geq n_1$, one necessarily has $\big| B_{M(b)} | \geq b |V_n|$, which is absurd, since it would yield
$$
\frac1b < M(b) \leq K_n(b) \leq \frac1b,
$$
where we used \eqref{e:kbound}. The proof of Part \ref{item:partA} is concluded.

\subsection{Proof of Part~\ref{item:partB} of Theorem \ref{th:irreducombdim}}\label{ss:proofb}

The following example shows that, at least for $d\geq 3$, one can easily build examples of sequences $\{E_n \}$ that have combinatorial dimensions strictly between $1$ and $d$.

\begin{example}[Fractional Cartesian products]\label{e:fcp}{\rm The following construction is a variation of the definition of \emph{fractional Cartesian products}, as discussed in \cite[Chapters XII-XIII]{bleibook} and \cite{bleijanson, DPKPTRF, NPR10}. Fix $d\geq 3$ and $ b = 1,...,d-1$. For every $n>b$, we define $V_n:= \{ {\bf v} = (v_1,...v_b)\in [n]^b : v_s\neq v_t, \, s\neq t\}$.  In what follows, generic elements $({\bf v}_1,...,{\bf v}_d)$ of the Cartesian product $\underbrace{V_n\times \cdots \times V_n}_\text{$d$ times}$ will be written 
\begin{equation}\label{e:uglyduck}
({\bf v}_1,...,{\bf v}_d) = (v_{(1,1)},v_{(1,2)},...,v_{(1,b)}; v_{(2,1)},...,v_{(2,b)};...;v_{(d,1)},  ..., v_{(d,b)}),
\end{equation}
 where $(v_{(s,1)},v_{(s,2)},...,v_{(s,b)}) = {\bf v}_s \in V_n$, $s=1,...,d$. We now fix a partition $S= \{S_1,...,S_d\}$ of the index set $I:= \{(1,1), (1,2),..., (d,b)\}$ with the following properties: 
 \begin{enumerate}[label=(\alph*)]
     \item \label{i:conn1} $|S_i| = b$, $i=1,...,d$;
     \item \label{i:conn2} for every $\ell = 1,...,d$, each $S_i$ contains at most one index of the form $(\ell, s)$, for $s=1,...,b$.
 \end{enumerate}
 We will also say that the partition $S$ is \emph{connected} if there do not exist partitions $\{B_1, B_2\}$ and $\{C_1,C_2\}$ of $[d]$ such that, for $i=1,2$
 $$
 \{(\ell, s) : \ell\in B_i\} = \bigcup_{j\in C_i} S_j.
 $$ 
 See Figure \ref{fig:rectangles} for examples of such partitions.

\begin{figure}[h]
\centering
\resizebox{0.5\textwidth}{!}{
\begin{tikzpicture}
    \draw (0,0.5) rectangle (2,3.5);
    \node at (1,3.8) {$b$};
    \node at (1,-0.5) {(a)};
    \node at (-0.3,2) {$d$};
    \fill[orange!40] (0,3.5) rectangle (1,2.5) node[midway] {\textcolor{black}{$S_3$}};
    \fill[blue!20] (1,3.5) rectangle (2,1.5) node[midway] {\textcolor{black}{$S_1$}};
    \fill[pink!50] (0,2.5) rectangle (1,0.5) node[midway] {\textcolor{black}{$S_2$}};
    \fill[orange!40] (1,0.5) rectangle (2,1.5) node[midway] {\textcolor{black}{$S_3$}};

    \draw (3.5,0) rectangle (5.5,4);
    \node at (4.5,4.3) {$b$};
    \node at (4.5,-0.5) {(b)};
    \node at (3.2,2) {$d$};
    \fill[orange!40] (3.5,2) rectangle (4.5,4) node[midway] {\textcolor{black}{$S_1$}};
    \fill[blue!20] (4.5,2) rectangle (5.5,4) node[midway] {\textcolor{black}{$S_2$}};
    \fill[pink!50] (3.5,0) rectangle (4.5,2) node[midway] {\textcolor{black}{$S_3$}};
    \fill[green!20] (4.5,0) rectangle (5.5,2) node[midway] {\textcolor{black}{$S_4$}};

    \draw (7,0) rectangle (9,4);
    \node at (8,4.3) {$b$};
    \node at (8,-0.5) {(c)};
    \node at (6.7,2) {$d$};
    \fill[green!20] (7,3) rectangle (8,4) node[midway] {\textcolor{black}{$S_4$}};
    \fill[orange!40] (8,2) rectangle (9,4) node[midway] {\textcolor{black}{$S_1$}};
    \fill[blue!20] (7,1) rectangle (8,3) node[midway] {\textcolor{black}{$S_2$}};
    \fill[pink!50] (8,0) rectangle (9,2) node[midway] {\textcolor{black}{$S_3$}};
    \fill[green!20] (7,0) rectangle (8,1) node[midway] {\textcolor{black}{$S_4$}};
\end{tikzpicture}}
\captionsetup{font=scriptsize}
\caption{ Some examples of partitions $S$ of the index set $I$. Case (a): for $d=3$ and $b=2$ there are only four partitions $S=\{S_1,S_2,S_3\}$ verifying \ref{i:conn1} and \ref{i:conn2} and all are connected; in the example, we have $S_1 = \{(1,2), (2,2)\}$, $S_2 = \{(2,1), (3,1)\}$ and $S_3 = \{(3,2), (1,1)\}$. Case (b): for $d=4$ and $b=2$, an example of disconnected partition $S=\{S_1,S_2,S_3, S_4\}$ verifying \ref{i:conn1} and \ref{i:conn2} is given by $S_1 = \{(1,1), (2,1)\}$, $S_2 = \{(1,2), (2,2)\}$, $S_3 = \{(3,1), (4,1)\}$ and $S_4 = \{(3,2), (4,2)\}$ (for which $B_1 = C_1 = \{1,2\}$ and $B_2 = C_2 = \{3,4\}$). Case (c):  for $d=4$ and $b=2$, an example of a connected partition is given by $S_1 = \{(1,2), (2,2)\}$, $S_2 = \{(2,1), (3,1)\}$, $S_3 = \{(3,2), (4,2)\}$ and $S_4 = \{(1,1), (4,1)\}$.}
\label{fig:rectangles}
\end{figure}
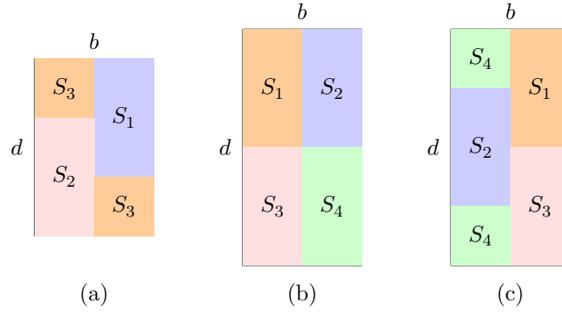

\medskip

\noindent For every $n>d$ we define the set $E_n^0 \subset V_n\times \cdots \times V_n$ as follows:
\begin{equation}\label{e:D0}
E^0_n := \{ ({\bf v}_1,...,{\bf v}_d) : v_{(s, u)} = v_{(t, v)} \mbox{ iff $(s,u), \, (t,v)\in S_i$ for some $i=1,...,d$} \},   
\end{equation}
where we used the notation \eqref{e:uglyduck}. Finally, we define $E_n$ to be the collection of those $({\bf v}_1,...,{\bf v}_d)$ such that $({\bf v}_{\sigma(1)},...,{\bf v}_{\sigma(d)}) \in E^0_n$, for some permutation $\sigma$ of $[d]$. In this way, each $E_n$ automatically satisfies the properties \eqref{item:VE1} and \eqref{item:VE2} (symmetry), introduced at the beginning of Section \ref{ss:prelim}. It is easily seen that a connected partition $S$ with the properties \ref{i:conn1} and \ref{i:conn2} listed above always exists; moreover, standard combinatorial considerations yield that such an $S$ can always be chosen so that each $E_n$ also verifies \eqref{item:VE3} from the beginning of Section \ref{ss:prelim}, that is: $E_n$ is non-diagonal. We claim that the sequence $\{E_n : n>d\}$ has combinatorial dimension $\alpha = \frac{d}{b}$, as per Definition \ref{d:cdm}. To see this, we first observe that, by definition, as $n\to\infty$, one has that $|V_n|\asymp n^b$, and
$$
|E_n^0| \asymp |E_n| \asymp n^{d} \asymp |V_n|^{\frac{d}{b}},
$$
proving that $\eqref{e:req2}$ is satisfied for $J_n = E_n$. On the other hand (e.g. by the triangle inequality), to prove \eqref{e:req2} for $J_n=E_n$ it is sufficient to show that
\begin{equation}\label{e:cdfp}
| E^0_n \cap (A_1\times \cdots \times A_d)| \leq c \, \max_{i=1,...,d} |A_i|^{\frac{d}{b}},
\end{equation}
for some absolute constant $c$. It turns out that such a bound is a direct consequence of a classical estimate by Finner, stated in \cite[Theorem 2.1]{F92} (see also \cite[Proposition 2.8]{BCCT08}). We refer to Appendix \ref{ss:combproof} for a complete proof of \eqref{e:cdfp}.}
\end{example}

\medskip

The following statement shows that the sequence of homogeneous sums associated with the irreducible sequence $\{E_n\}$ verifies a CLT, thus concluding the proof of Part \ref{item:partB} of Theorem \ref{th:irreducombdim}.

\begin{proposition}[Fractional products yield irreducible CLTs] \label{cor:fracprod} For $d\geq 3$, let $\{E_n : n>k\}$ be the sequence of sets constructed in Example {\rm \ref{e:fcp}}. Assume in addition that the underlying partition $S = \{S_1,...,S_k\}$ is connected, and verifies Properties {\rm  \ref{i:conn1}} and {\rm \ref{i:conn2}} in Example {\rm \ref{e:fcp}}. For $n>d$, we consider the sequence $\{Z_n\}$ defined as in \eqref{eq:defpolchaos} for $q_n\equiv 1$. Then, as $n\to\infty$, $\mathbb{E}[Z_n^2] = d! |E_n| \asymp n^d$, and the sequence $\widetilde{Z}_n = \frac{Z_n}{\sqrt{d!|E_n|}}$ converges in distribution to a standard Gaussian random variable $N$. 
\end{proposition}
\begin{proof} To show that $\widetilde{Z}_n$ converges in distribution to $N$, we can directly use \cite[Proof of Proposition 6.6 and Proposition 6.8]{NPR10b} to deduce that, for every thrice differentiable bounded function $h : \R\to \R$ with bounded derivatives,
\begin{equation}\label{e:gesine}
| \mathbb{E}[h(\widetilde{Z}_n) ] - \mathbb{E}[h(N) ] |\leq \frac{C}{|V_n|^{\frac{1}{2b}}}, 
\end{equation}
for some absolute constant $C$, so that the conclusion follows from the fact that $|V_n|\to \infty$.    
\end{proof}

{ \begin{remark}{\rm The estimate \eqref{e:gesine} follows from \cite[formula (6.72)]{NPR10b}, implying that the quantity $| \mathbb{E}[h(\widetilde{Z}_n) ] - \mathbb{E}[h(N) ] |$ is bounded by a multiple of the sum
$$
\frac{|E_n^\#|^{1/2} }{|E_n|} + 
\frac{\max_{{\bf i}\in V_n} |E^*_{n, {\bf i}}|^{1/4} }{|E_n|^{1/4}},
$$
where $E^*_{n, {\bf i}} := \{ ({\bf v}_1,...,{\bf v}_d)\in E_n : {\bf i} = {\bf v}_\ell \mbox{ for one } \ell\in [d]\}$, and $E_n^\# \subset E_n\times E_n$ is defined as the collection of all pairs $(F, G) := (({\bf f}_1,...,{\bf f}_d) , ({\bf g}_1,...,{\bf g}_d)) \in E_n \times E_n$ such that: (a) $F$ and $G$ have no elements in common, and (b) there exists $p\in [d-1]$, as well as distinct integers $\ell_1,.., \ell_p\in [d]$ such that replacing each ${\bf f}_{\ell_i}$ in $F$ with the corresponding ${\bf g}_{\ell_i}$ in $G$ (and vice versa) results in a (possibly different) element of $E_n \times E_n$. The claimed estimate \eqref{e:gesine} then follows from \cite[formula (6.70)]{NPR10b}, implying that, for some absolute constant $C$, one has that $|E_n^\#| \leq C |V_n|^{2\alpha - \frac{1}{b}}$ and $\max_{{\bf i}\in V_n} |E^*_{n, {\bf i}}|\leq C |V_n|^{\alpha-1}$. 
}
\end{remark}}
}
\section{Irreducibility and sparsity: ad hoc construction for $d=2$}\label{sec:d=2}
This section provides the proof of Theorem~\ref{th:irreducibilityk=2}, under the assumptions stated in Subsection~\ref{sss:adhoc}.
In particular, we prove a refined version of Theorem~\ref{th:irreducibilityk=2} (see Theorem~\ref{th:mainth2d} below). We not only show that any partition $B_1, \ldots, B_{m_n}$ of $V_n$ satisfying \ref{redCLTcond1} and \ref{redCLTcond2} cannot satisfy \ref{redCLTcond3} (this already implies irreducibility; see Definition~\ref{def:irreducibility} and Remark~\ref{r:reducible}-\ref{item:point1}), but we also provide a quantitative bound contradicting \ref{redCLTcond3}.
In more detail, we provide a lower bound away from zero for the main contribution to the second moment coming from one of the sets in the partition (see inequality \eqref{eq:contrBox0}), which is in conflict with \ref{redCLTcond3}. 
\medskip

We briefly recall the notation of Section \ref{sss:adhoc}. For $n \in \N$, we set $V_n= [n]^2$, fix $\beta \in (0,1]$ and for $a,b \in \{1,\ldots,n\}$ consider the subsets
$\bbS_{\mathsf{v}}(a)
\subseteq \{a\} \times 
\{1,\ldots, n\}$
and $\bbS_{\mathsf{h}}(b) \subseteq \{1,\ldots, n\} \times \{b\}$
with $|\bbS_{\mathsf{v}}(a)| = |\bbS_{\mathsf{h}}(b)| = \lfloor \beta n \rfloor$.
\smallskip

We recall the equivalence relations $\stackrel{\mathsf{h}}{\sim}$ and $\stackrel{\mathsf{v}}{\sim}$ on $V_n$: for any $v_1,v_2 \in V_n$
\begin{gather}\label{eq:relationH+}
	v_1 \stackrel{\mathsf{h}}{\sim} v_2 \quad  \iff \quad 
 \text{for some $b$ we have } \ v_1,v_2 \in \bbS_{\mathsf{h}}(b) \ \text{ with } \ v_1 \ne v_2 \,, \\
	\label{eq:relationV+}
	v_1 \stackrel{\mathsf{v}}{\sim} v_2  \quad\iff \quad \text{for some $a$ we have } \ v_1,v_2 \in \bbS_{\mathsf{v}}(a) \ \text{ with } \ v_1 \ne v_2 \,,
\end{gather}
and we write $v_1 \sim v_2$ if and only if $ v_1 \stackrel{\mathsf{h}}{\sim} v_2$ or $
    v_1 \stackrel{\mathsf{v}}{\sim} v_2$.
We set $E_n \coloneq \big\lbrace (v_1,v_2) \in V_n \times V_n : v_1 \sim v_2 \big\rbrace$
and we consider the homogeneous sum
\begin{equation}
    \label{eq:polch6}
    Z_n= \sum_{v_1,v_2 \in V_n} \ind_{E_n}(v_1,v_2) \, X_{v_1}X_{v_2} \,,
\end{equation}
where $\{X_v\}_{v \in \bigcup_n V_n}$ is a family of i.i.d. standard Gaussian random variables. We recall that $Z_n$ is centered with second moment 
\begin{equation*}
  \mathbb{E}[Z_n^2]= 2|E_n|=
  2 \bigg\{ \sum_{b=1}^n |\bbS_{\mathsf{h}}(b)|
  \, (|\bbS_{\mathsf{h}}(b)|-1) +
  \sum_{a=1}^n |\bbS_{\mathsf{v}}(a)|
  \, (|\bbS_{\mathsf{v}}(a)|-1) \bigg\}
  \underset{n\to\infty}{\sim} 
  4 \beta^2 n^3\,.
\end{equation*}
    
\smallskip

The following result establishes irreducibility for the sequence $\{Z_n\}$.

\begin{theorem}[Irreducibility and quantitative bounds]\label{th:mainth2d}
Fix any $\beta \in (\frac{1}{2},1]$. Then
\begin{enumerate}
\renewcommand{\labelenumi}{\upshape(\arabic{enumi})}
    \item\label{item:boxes} Any partition $\{B_\ell = B_\ell(n,m_n)\}_{ \ell=1,\ldots,m_n}$, which satisfies as $n\to\infty$
\begin{equation}
	\label{eq:condmainth0}
	\sum_{\ell=1}^{m_n} \sigma_n^2(B_\ell) \ = \ ({4}\beta^2 + o(1)) \, n^3 
\end{equation}
must also satisfy
\begin{equation}\label{eq:contrBox0}
	\max_{\ell=1,\ldots, m_n}
	\sigma_n^2(B_{\ell})  \ \ge \ \big( {2} \beta^2(2\beta-1) + o(1) \big) \, {n^3} \,.
	\end{equation}
    \item\label{item:irredCLT}
    {As a consequence}, $\{Z_n\}$ satisfies an irreducible CLT in the sense of Definition \ref{def:irreducibility}.
\end{enumerate}
\end{theorem}
Let us prove point \eqref{item:irredCLT}, while the (longer) proof of the quantitative bound is presented separately in Subsection~\ref{subsec:proofpunto1}.

\begin{proof}[Proof of Theorem \ref{th:mainth2d}-\eqref{item:irredCLT}] 
Assuming \eqref{item:boxes} in Theorem \ref{th:mainth2d}, the conditions \ref{redCLTcond1}, \ref{redCLTcond2} and \ref{redCLTcond3} in Definition \ref{def:irreducibility} cannot hold together, thus $\{Z_n\}$ is irreducible. We only need to prove the convergence in distribution towards $N \sim \mathcal{N}(0,1)$ by applying the Fourth Moment Theorem (see \cite[Proposition 1.6]{NPR10} and \cite[Theorem 4.2]{CSZ17b}). Recall $\widetilde{Z}_n \coloneq Z_n/\sqrt{2|E_n|}$.
Since $\lim_{n\to\infty} \mathbb{E} \big[\widetilde{Z}_n^2\big] =1$, we only need to prove that $\lim_{n \to \infty} \mathbb{E}[\widetilde{Z}_n^4]=3$. We have
\begin{align}
    \mathbb{E}[\widetilde{Z}_n^4] &= \frac{1}{4|E_n|^2} \sum_{v_1,\ldots,v_8\in V_n} \ind_{E_n}(v_1,v_2)\ind_{E_n}(v_3,v_4)\ind_{E_n}(v_5,v_6)\ind_{E_n}(v_7,v_8) \, \mathbb{E} \bigg[ \prod_{i=1}^8 X_{v_i}\bigg] \notag\\
    & =\frac{4}{|E_n|^2} \sum_{A,B,C,D\subseteq V_n} q_n(A)q_n(B)q_n(C)q_n(D) \, \mathbb{E} \big[ X_AX_BX_CX_D \big]\,, \label{eq:momq}
\end{align}
where for convenience we rearranged the sum over (\emph{unordered}) subsets $A=\{v_1,v_2\}$, $B=\{v_3,v_4\}$, $C=\{v_5,v_6\}$, $D=\{v_7,v_8\}$ (recall that each vertex $v_i$ is an \emph{ordered} couple in the square $[n]^2$) and for each of these subsets, say $A \subseteq V_n$, we set $q_n(A)\coloneq \ind_{E_n}(v_1,v_2)$ and $X_A \coloneq X_{v_1}X_{v_2}$.
 
Recall that the $X_v$'s are centered and independent and that all pairs $A,B,C,D \subseteq V_n$ contain two \emph{distinct} vertices (the diagonal points of $V_n \times V_n$ are not in $E_n$). Therefore, the only non--zero contribution to $\mathbb{E} [X_A\, X_B\, X_C \,X_D]=\mathbb{E}[X_{v_1}X_{v_2}X_{v_3}X_{v_4}X_{v_5}X_{v_6}X_{v_7}X_{v_8}]$ is given by those terms where the individual $X_v$'s match either in pairs or in quadruples. We outline the possible cases below, recalling that $2|E_n| \sim 4 \beta^2 n^3$ as $n\to\infty$.
\begin{enumerate}
	\item The $X_v$'s match in quadruples, i.e.\ $A=B=C=D$, and the contribution to \eqref{eq:momq} is negligible as $n\to\infty$:
\begin{equation*}
\begin{split}
    4|E_n|^{-2}\sum_{A \subseteq V_n} q_n(A)^4 \mathbb{E}[X_A^4]&=4|E_n|^{-2}\sum_{\{v_1,v_2\} \subseteq V_n}\ind_{E_n}(v_1,v_2) \mathbb{E}[X_{v_1}^4X_{v_2}^4]\\
    &=18|E_n|^{-1} = O\big(n^{-3}\big)\,.
\end{split}
\end{equation*}

\item The $X_v$'s match in pairs, however $A,B,C,D\subseteq V_n$ do not pair up two by two, for instance when $A=\{v_1,v_2\}$, $B=\{v_1,v_3\}$, $C=\{v_3,v_4\}$ and $D=\{v_2,v_4\}$. It is simple to see that their contribution to \eqref{eq:momq} is always either $O\big(n^{-1}\big)$ or $O\big(n^{-2}\big)$. To give a brief idea, let us consider the term when $\sim$ is always $\stackrel{\mathsf{v}}{\sim}$, thus the indicator function of $E_n$ imposes that $v_1=(a_1,b_1)$, $v_2=(a_1,b_2)$, $v_3=(a_1,b_3)$ and $v_4=(a_1,b_4)$. The corresponding term is then negligible for large $n$, indeed
\begin{equation*}
\begin{split}
    &4|E_n|^{-2}\sum_{v_1,\ldots,v_4 \in V_n} \ind_{E_n}(v_1,v_2)\ind_{E_n}(v_1,v_3)\ind_{E_n}(v_3,v_4)\ind_{E_n}(v_2,v_4) \le O\big(n^{-1}\big)\,,
\end{split}
\end{equation*}
since the sum above can be bounded by the sum over the five degrees of freedom $a_1,b_1,b_2,b_3,b_4 \in \{1,\ldots,n\}$. The other cases work similarly.
\smallskip

\item  The $X_v$'s match in pairs and $A,B,C,D\subseteq V_n$ pair up two by two, i.e.\ either $A=B$ and $C=D$, or $A=C$ and $B=D$, or $A=D$ and $C=B$. If two distinct couples $A,B \subset V_n$ differ from each other but have a common element $v\in V_n$, the corresponding term can be treated similarly as in the previous case, thus giving a negligible contribution. Therefore, consider the case where all distinct couples $A,B\subseteq V_n$ are also disjoint: this is the only non--negligible contribution and gives exactly
\begin{equation*}
\begin{split}
3\Bigg(4|E_n|^{-2}\sum_{\substack{A,B\subseteq V_n\\ A \cap B=\emptyset}}  q(A)^2 \, q(B)^2 \mathbb{E}\big[X_A^2 X_B^2\big] \Bigg) &=3 \, {\big(4|E_n|^{-2}\big)} \Bigg( \sum_{\{v_1,v_2\}\subseteq V_n}  \ind_{E_n}(v_1,v_2) \Bigg)^2 =3\,.
\end{split}
\end{equation*}
\end{enumerate}
\end{proof}

\subsection{Proof of Theorem \ref{th:mainth2d}-\eqref{item:boxes}} \label{subsec:proofpunto1} We divide the proof into three steps, which we first outline below.
\subsubsection{Strategy of the proof}
\label{sec:strategy}
We assume that \eqref{eq:condmainth0} holds. We fix $\eta \in (0,1)$
small
and we take $n$ large enough so that (the factor $\frac{1}{2}$ is for later convenience)
\begin{equation}	\label{eq:condmainth}
	\sum_{\ell=1}^{m_n} \sigma_n^2(B_\ell) \ge ({4}\beta^2 -\tfrac{1}{2}\,\eta) \, n^3 \,.
\end{equation}
We are going to show that there is $\bar{\ell} = \bar\ell_{n,\eta} \in \{1,\ldots,m_n\}$
such that
\begin{equation}\label{eq:contrBox}
	\sigma_n^2(B_{\bar{\ell}})  \ge  2 \big( \beta^2 (2\beta-1) - {11}\sqrt{\eta} + O(\tfrac{1}{n}) \big) \, {n^3} \,.
\end{equation}
(The factor $11$ multiplyng $\sqrt{\eta}$ is immaterial, but it will be transparent to carry out explicit computations.)
Since we can take $\eta > 0$ as small as we wish, this proves \eqref{eq:contrBox0}.

It remains to prove \eqref{eq:contrBox}.
Given a subset $A \subseteq V_n$, we denote its ``rows'' by
	\begin{equation}
		\label{eq:Brow}
		A(\cdot,b) \coloneq A \cap \big( \{1,\ldots,n\} \times \{b\} \big)
		\qquad \text{for} \quad b \in \{1,\ldots,n\}\,,
	\end{equation}
and similarly we denote its ``columns'' by
	\begin{equation}
		\label{eq:Bcolumn}
		A(a, \cdot) \coloneq A \cap \big( \{a\} \times \{1,\ldots,n\}  \big)
		\qquad \text{for} \quad a \in \{1,\ldots,n\}\,.
	\end{equation}
To help explanations, we refer to the labels $\ell \in \{1,\ldots, m_n\}$
of the partition $\{B_\ell\}$ as \emph{colors}.
We are going to prove the following three steps.

\begin{enumerate}
\item \emph{Almost each row and column has a ``dominant'' color},
almost filling up
$\bbS_h(b)$ or $\bbS_v(a)$:
\begin{gather}
	\notag
	\exists I,I' \subseteq \{1,\ldots, n\} \quad \text{such that} \\
	\label{eq:elleb}
		|I| \ge \big(1- {\tfrac{1}{4}}\sqrt{\eta}\big)\,n \quad\ \text{and} \quad\
		\forall b \in I \ \ \exists \ell_b \colon
		\quad 
		| B_{\ell_b}(\cdot,b)\cap \bbS_{\mathsf{h}}(b)| \ge 
		(\beta-2\sqrt{\eta})\, n \,,
	\\
	\label{eq:ellea'}
		|I'| \ge \big(1- {\tfrac{1}{4}}\sqrt{\eta}\big)\,n \quad\ \text{and} \quad\
	\forall a \in I' \ \, \exists \ell_a' \colon
	\quad | B_{\ell_a'}(a,\cdot)\cap \bbS_{\mathsf{v}}(a)| \ge 
	(\beta-2\sqrt{\eta}) \, n \,.
\end{gather}

\item \emph{There is a color $\bar\ell$ which is dominant for a
positive fraction of rows}:
\begin{equation}\label{eq:row-dominant}
	\exists \bar\ell \colon 
	\quad
	| \{ b \in I \colon \ \ell_b = \bar\ell\, \} |  \ge 
	(\beta'-7\sqrt{\eta}) \, n
 \qquad \text{where} \quad 
 \beta' := 2\beta-1 > 0\,.
\end{equation}

\item \emph{Color $\bar\ell$ fulfils relation \eqref{eq:contrBox}.}
\end{enumerate}

\subsubsection{Step 1} 
By symmetry, we only prove \eqref{eq:elleb}. We argue in three parts.

\medskip
\noindent
\emph{Part A.}
Recalling \eqref{eq:sigmaBdef}, for $B \subseteq V_n$
we define $\sigma_{n,h}^2(B)$ and $\sigma_{n,v}^2(B)$ by
\begin{equation} \label{eq:sigma2hv}
	\sigma_n^2(B) = \underbrace{2\sum_{v_1,v_2 \in B} 
	\ind_{v_1 \overset{\mathsf{h}}{\sim} v_2}}_{\sigma_{n,\mathsf{h}}^2(B)}
	\ + \ \underbrace{2\sum_{v_1,v_2 \in  B} 
	\ind_{v_1 \overset{\mathsf{v}}{\sim} v_2}}_{\sigma_{n,\mathsf{v}}^2(B)} \,.
\end{equation}
We prove in this part that
\begin{equation}
	\label{eq:step1h}
	\sum_{\ell=1}^{m_n} \sigma_{n,\mathsf{h}}^2(B_\ell) \ge ({2}\beta^2- \tfrac{1}{2}\eta) \, n^3 \,,
\end{equation}

Note that
by \eqref{eq:relationH+}
and \eqref{eq:relationV+}
we can write for $B \subseteq V_n$:
\begin{equation} \label{eq:sigmah}
	\sigma_{n,\mathsf{h}}^2(B) = {2}\sum_{b=1}^n |B(\cdot,b) \cap \bbS_{\mathsf{h}}(b)| \, \big(|B(\cdot,b) \cap \bbS_{\mathsf{h}}(b)| -1 \big) 
\end{equation}
and
\begin{equation*}
    \sigma_{n,\mathsf{v}}^2(B) = {2}\sum_{a=1}^n |B(a,\cdot) \cap \bbS_{\mathsf{v}}(a)| \, \big( |B(a,\cdot) \cap \bbS_{\mathsf{v}}(a)| -1 \big)\,,
\end{equation*}
in particular
\begin{equation} \label{eq:Theta}
	\sum_{\ell=1}^{m_n} \sigma_{n,\mathsf{h}}^2(B_\ell) = \sum_{b=1}^n \Theta_{n,h}(b)
	\quad \text{where} \quad
	\Theta_{n,\mathsf{h}}(b) \coloneq {2}\sum_{\ell=1}^{m_n} |B_\ell(\cdot,b) \cap \bbS_{\mathsf{h}}(b)| \, \big(|B_\ell(\cdot,b) \cap \bbS_{\mathsf{h}}(b)| -1 \big) .
\end{equation}
Since $(B_\ell(\cdot,b))_{\ell = 1, \ldots, m_n}$ is a partition of $\{1,\ldots, n\} \times \{b\}$,
we can bound
\begin{equation}\label{eq:boundtheta}
	\Theta_{n,\mathsf{h}}(b) = {2}\sum_{\ell=1}^{m_n} 
	\sum_{v_1 \ne v_2 \in B_\ell(\cdot, b) \cap \bbS_{\mathsf{h}}(b)} 1
	\le {2}\sum_{v_1 \ne v_2 \in  \bbS_{\mathsf{h}}(b)}
	1 = {2}\lfloor \beta n \rfloor\big(\lfloor \beta n \rfloor -1\big)  \le
 {2\beta^2 \,n^2}\,,
\end{equation}
which yields $\sum_{\ell=1}^{m_n} \sigma_{n,\mathsf{h}}^2(B_\ell)  \le {2}\sum_{b=1}^n {\beta^2 \,n^2}
\le {2\beta^2 \,n^3}$. The same arguments apply to $\sigma_{n,\mathsf{v}}^2$, hence
\begin{equation} \label{eq:bosu}
	\sum_{\ell=1}^{m_n} \sigma_{n,\mathsf{v}}^2(B_\ell) 
	\le {2\beta^2 \,n^3}\,.
\end{equation}
Finally, relation \eqref{eq:step1h} follows by
$\sigma_{n,\mathsf{h}}^2(B_\ell) = \sigma_{n}^2(B_\ell) - \sigma_{n,\mathsf{v}}^2(B_\ell)$ applying
\eqref{eq:condmainth} and \eqref{eq:bosu}.

\medskip
\noindent
\emph{Part B.}
We now show that $\Theta_{n,\mathsf{h}}(b)$ from \eqref{eq:Theta}
is close to its maximum ${2 \beta^2\,n^2}$ for most values of $b$:
more precisely, setting
\begin{equation} \label{eq:defI}
	I \coloneq \bigg{\lbrace} \, b \in \{1,\ldots,n\} :  
	\ \Theta_{n,\mathsf{h}}(b) > {2}(\beta^2 -\sqrt{\eta}) \, {n^2} \, \bigg{\rbrace}\,,
\end{equation}
we show that
\begin{equation} \label{eq:boundI}
	|I| \ge \big(1- {\tfrac{1}{4}}\sqrt{\eta}\big) \,n\,.
\end{equation}

To this purpose, we can write
\begin{equation*}
\begin{split}
	\sum_{\ell=1}^{m_n} \sigma_{n,\mathsf{h}}^2(B_\ell) 
	= \sum_{b \in I} \Theta_{n,\mathsf{h}}(b)  + \sum_{b \in I^c} \Theta_{n,\mathsf{h}}(b)
	& \le  |I| \, {2\beta^2
 \,n^2} + (n-|I|) \, {2}(\beta^2 -\sqrt{\eta}) \, {n^2} \,,
\end{split}
\end{equation*}
which can be rewritten as
\begin{equation*}
\begin{split}
	 \sqrt{\eta} \, {2n^2} \, |I|
	\ge \sum_{\ell=1}^{m_n} \sigma_{n,h}^2(B_\ell) - {2}(\beta^2 -\sqrt{\eta}) \,{ n^3 }\,.
\end{split}
\end{equation*}
Plugging in the bound \eqref{eq:step1h}, we obtain
precisely \eqref{eq:boundI}.

\medskip
\noindent
\emph{Part C.}
Given $b \in I$,
see \eqref{eq:defI}, let
$\ell_b$ be a color for which $| B_\ell(\cdot,b)\cap \bbS_{\mathsf{h}}(b)| $ is maximal:
\begin{equation*}
	\ell_b \coloneq \arg \max \big{\lbrace} | B_\ell(\cdot,b)\cap \bbS_{\mathsf{h}}(b)| 
	\colon\ 1 \le \ell \le m_n \big{\rbrace}\,.
\end{equation*}
We  show that
\begin{equation}\label{eq:bBbound}
	\forall b \in I \colon \qquad |B_{\ell_b}(\cdot,b)\cap \bbS_{\mathsf{h}}(b)| \ge 
 (\beta -2\sqrt{\eta}) \, n \,,
\end{equation}
which completes the proof of \eqref{eq:elleb}.

To prove \eqref{eq:bBbound}, we recall from \eqref{eq:Theta} that,
setting $k_\ell \coloneq | B_\ell(\cdot,b)\cap \bbS_{\mathsf{h}}(b)|$,
we have
\begin{equation} \label{eq:Thetakappa}
\begin{split}
	\Theta_{n,\mathsf{h}}(b) = {2}\sum_{\ell=1}^{m_n} { k_\ell (k_\ell -1)}
	\qquad 
 \text{with} \quad
	k_\ell \ge 0 \quad
  \text{such that} \quad
	\sum_{\ell=1}^{m_n} k_\ell =
 | \bbS_{\mathsf{h}}(b)| = \lfloor \beta \, n \rfloor \,.
\end{split}
\end{equation}
The function $(k_1,\ldots,k_{m}) \mapsto {2}\sum_{\ell=1}^{m} { k_\ell (k_\ell -1)}$ 
with the above constraints is maximised for $m=1$ and $k_1 = \lfloor \beta \, n \rfloor$, where it takes
the value~${2 \lfloor \beta \, n \rfloor (\lfloor \beta \, n \rfloor-1)} \sim {2\beta^2\,n^2}$.
Since $\Theta_{n,\mathsf{h}}(b)$ is close to ${2 \beta^2\,n^2}$ for $b\in I$,
see \eqref{eq:defI},
this explains intuitively why $k_{\ell_b} = |B_{\ell_b}(\cdot,b)\cap \bbS_{\mathsf{h}}(b)| = \max_{\ell=1,\ldots,m_n} k_\ell$ 
should be close to~$\beta \, n$,
so that \eqref{eq:bBbound} holds.
To make this precise, we simply bound
\begin{equation} \label{eq:Thetafin}
\begin{split}
	\Theta_{n,\mathsf{h}}(b)	&= {2 k_{\ell_b} (k_{\ell_b} -1)} + 
	{2}\sum_{\ell \ne \ell_b} { k_\ell (k_\ell -1)}
	\\
    &\le {2 k_{\ell_b}^2} + {2 k_{\ell_b}}\, 
	\sum_{\ell \ne \ell_b}  k_\ell
	= {2 k_{\ell_b}^2} + {2 k_{\ell_b}}\, (
 \lfloor \beta \, n \rfloor -k_{\ell_b})
	\le {2 k_{\ell_b}} \, 
  \beta \, n \,.
	\end{split}
\end{equation}
Since $\Theta_{n,\mathsf{h}}(b) \ge {2}(\beta-\sqrt{\eta}) \, {n^2}$ for $b \in I$ and $\beta > \frac{1}{2}$,
we see that $k_{\ell_b}$ fulfils \eqref{eq:bBbound}.\qed

\subsubsection{Step~2}

We prove \eqref{eq:row-dominant} by contradiction: we assume that
\begin{equation}\label{eq:contra}
	\forall\ell \colon 
	\quad
	| \{ b \in I \colon \ \ell_b = \ell\, \} |  <
	(\beta'-7\sqrt{\eta}) \, n \qquad \text{with} \quad\beta' := 2\beta-1 \,,
\end{equation}
and we deduce a contradiction with \eqref{eq:ellea'}, namely
\begin{equation} \label{eq:contra-fin}
	\exists a \in I'\ \ 
 \forall\ell \colon \quad |B_\ell(a,\cdot) | <  (\beta-2\sqrt{\eta}) \, n \,.
\end{equation}

Recalling \eqref{eq:elleb},
let $D$ be the union of all
rows of boxes with dominant colors:
\begin{equation*}
	D := \bigcup_{b \in I}
 \big\{ B_{\ell_b}(\cdot, b) \cap \bbS_{\mathsf{h}}(b) \big\} \,.
\end{equation*}
By assumption \eqref{eq:contra},
each color $\ell$ appears in less
than $(\beta' - {7}\sqrt{\eta})\, n$
rows of $D$, therefore
\begin{equation} \label{eq:contra1}
	\forall a \ \ 
 \forall \ell \colon \quad
	| B_\ell(a,\cdot) \cap D | < (\beta'-7\sqrt{\eta}) \, n \,.
\end{equation}
We claim that we can obtain the following bound for a suitable $a \in I'$:
\begin{equation} \label{eq:claim-step-2}
    \exists a \in I' \colon \ \ 
	\forall \ell \colon \quad
	|B_\ell(a,\cdot) \cap D^c| 
 \le (1-\beta + 5\sqrt{\eta}) \, n \,.
\end{equation}
Summing \eqref{eq:contra1} and \eqref{eq:claim-step-2} we obtain \eqref{eq:contra-fin} (note that $\beta' + (1 - \beta) = \beta$).

\smallskip

It only remains to prove \eqref{eq:claim-step-2}.
We observe that by \eqref{eq:elleb}
\begin{equation} \label{eq:carD}
	|D| \ge |I| \, (\beta -2\sqrt{\eta}) \, n
	\ge {\big(1-\tfrac{1}{4}\sqrt{\eta}\big)}
 \, (\beta - 2\sqrt{\eta}) \, n^2 \ge (\beta -3\sqrt{\eta}) \, n^2 \,.
\end{equation}
We then define
\begin{equation} \label{eq:defJ}
	J := \{a \in \{1,\ldots, n\} \colon \ |D(a,\cdot)| > (\beta -5 \sqrt{\eta}) \, n\}
\end{equation}
so that, by Lemma~\ref{th:comb}
below, we have
\begin{equation} \label{eq:Jbound}
	|J| \ge 2\sqrt{\eta} \, n \,.
\end{equation}
Since ${|I'| \ge \big(1-\tfrac{1}{4}\sqrt{\eta}\big)n> (1-\sqrt{\eta})\,n}$,
see \eqref{eq:ellea'},
we have $|I'|+|J| > n$ and therefore $I'\cap J \ne \emptyset$.
We then select any $a \in I' \cap J$
and note that
$|B_\ell(a,\cdot) \cap D^c|
\le |D(a,\cdot)^c|
= n - |D(a,\cdot)|$,
hence the bound
\eqref{eq:claim-step-2} follows
by the definition \eqref{eq:defJ}
of~$J$.\qed

\medskip

It remains to prove the
following elementary lemma
(recall that $V_n = \{1,\ldots,n\}^2$).

\begin{lemma} \label{th:comb}
If $C \subseteq V_n$ satisfies
$|C| \ge u \, n^2$ for some $u \in (0,1)$, then for any $u' < u$
\begin{gather*}
	\big| \big\{ a \in \{1,\ldots,n\} \colon\ |C(a,\cdot)| > u' \, n \big\} \big| \ge 
	\frac{u-u'}{1-u'}
\, n \ge (u-u') \, n \,.
\end{gather*}
\end{lemma}

\begin{proof}
Setting $J := \big\{ a \in \{1,\ldots,n\} \colon\ |C(a,\cdot)| > 
u' \, n \big\}$ we 
can bound
\begin{equation*}
\begin{split}
	|C| = \sum_{a \in J} |C(a,\cdot)| + \sum_{a \in J^c} |C(a,\cdot)| 
	& \le |J| \, n + (n-|J|) \, 
 u' \, n \\
	& = (1- u') \, n \, |J|
	+ u' \, n^2 \,,
\end{split}
\end{equation*}
that is $|J| \ge \frac{|C| - 
u' \, n^2}{(1 - u') \, n}$.
Plugging in $|C| \ge u\, n^2$ completes the proof.
\end{proof}

\subsubsection{Step 3}
We finally prove \eqref{eq:contrBox}. 
From \eqref{eq:sigma2hv} and \eqref{eq:sigmah} we can write
\begin{equation*}
\begin{split}
	\sigma_n^2(B_{\bar\ell})
	\ge \sigma_{n,\mathsf{h}}^2(B_{\bar\ell})
={2} \sum_{b = 1}^{n} |B_{\bar\ell}(\cdot,b) \cap \bbS_{\mathsf{h}}(b)|\, \big( |B_{\bar\ell}(\cdot,b) \cap \bbS_{\mathsf{h}}(b)| -1\big)
\end{split}
\end{equation*}
Restricting the sum to the set
$B := \{b \in I \colon \ell_b = \bar\ell\}$, recalling \eqref{eq:elleb} and \eqref{eq:row-dominant}, we obtain
\begin{equation} \label{eq:estrow}
\begin{split}
	\sigma_{n,h}^2(B_{\bar\ell})
	&\ge {2}(\beta' - 7 \sqrt{\eta}) \, n \, (\beta -{2}\sqrt{\eta}) \, n \, \big((\beta -{2}\sqrt{\eta}) \, n-1\big)\\
	&\ge {2} (\beta' - 7 \sqrt{\eta}) 
 \, (\beta - {2}\sqrt{\eta} +
 O(\tfrac{1}{n}))^2
 \, {n^3} \ge {2} \big( \beta' \beta^2 - {11}\sqrt{\eta} + O(\tfrac{1}{n}) \big) \, {n^3} \,,
\end{split}
\end{equation}
which completes the proof.\qed

\begin{remark}{\rm 
In \eqref{eq:estrow} we only estimated the ``horizontal'' contribution to the variance $\sigma_{n,\mathsf{h}}^2(B_{\bar\ell})$.
Of course, by symmetry, a version of \eqref{eq:row-dominant} holds for columns, for some color $\bar\ell'$,
hence
an estimate like \eqref{eq:estrow} 
also holds for $\sigma_{n,\mathsf{v}}^2(B_{\bar\ell'})$.
Note that both $B_{\bar\ell}$
and $B_{\bar\ell'}$ have cardinality at least $(\beta'-7\sqrt{\eta})\,(\beta-2\sqrt{\eta}) \, n^2$ by \eqref{eq:elleb} and \eqref{eq:row-dominant}, hence they must overlap when $\beta \, \beta' = \beta(2\beta - 1) > \frac{1}{2}$,
i.e.\ for $\beta > \frac{1+\sqrt{5}}{4} \simeq 0.81$:
in this case 
$\bar\ell = \bar\ell'$, hence
we can improve our final bound
\eqref{eq:contrBox}
by a factor~$2$:
\begin{equation}
    \sigma_n^2(B_{\bar\ell})
	= \sigma_{n,\mathsf{h}}^2(B_{\bar\ell})
 + \sigma_{n,\mathsf{v}}^2(B_{\bar\ell})
 \ge {4}\big( \beta' \beta^2 - {11}\sqrt{\eta} + O(\tfrac{1}{n}) \big) \, {n^3 }\,.
\end{equation}}
\end{remark}

\appendix
\section{Proofs of some technical results}\label{appendix}

\subsection{Proof of Theorem \ref{t:probamethod}}\label{ss:proofproba}

Let $\Delta(\alpha,n)$ denote the maximal degree in the Erd\"os-Renyi random graph $G(n,p_n)$, where $p_n = n^{\alpha-2}$. By Theorem \ref{t:peres} (including the subsequent discussion) and Remark~\ref{r:afterdejong}-\ref{item:rempoint2}, it is sufficient to show that, for $\alpha$ as in the statement, there exists $0<\epsilon<\alpha/2$ such that
$$
\mathbb{P}[\Delta(\alpha,n)\geq n^{\alpha/2-\epsilon}] \longrightarrow 0, \quad n\to\infty.
$$
One has that 
$$
\mathbb{P}[\Delta(\alpha,n)\geq n^{\alpha/2-\epsilon}]\leq n\mathbb{P}[B(n-1,p_n)\geq n^{\alpha/2-\epsilon}]\leq n\mathbb{P}[B(n,p_n)\geq n^{\alpha/2-\epsilon}],$$
where $B(k,p)$ denotes a binomial random variable with parameters $(k,p)$, and the conclusion follows from an application of the multiplicative Chernoff inequality: according e.g. to \cite[Problem 1.6-(e)]{DuPanconesi} we can bound $\mathbb{P}(B(k,p) \ge t \, kp) \le 2^{-t \, kp}$ provided $t > 2 \rme$, hence  for $0<\epsilon < \min\{\alpha/2, 1-\alpha/2\}$ 
$$
\mathbb{P}[B(n,p_n)\geq n^{\alpha/2-\epsilon}]\leq 2^{-(n^{1- \epsilon - \alpha/2})n^{\alpha-1}} = {2^{-n^{\alpha/2-\epsilon}}} = o(n^{-1}), \quad n\to \infty,
$$
thus yielding the desired conclusion.

\subsection{Proof of Theorem \ref{th:k-irr+}}

\smallskip

To prove Theorem \ref{th:mainresulthyper} we follow very closely the proof of Theorem~\ref{th:mainSPECTRUMgraph}.
By contradiction we assume that \eqref{e:mucca} holds for some $k\geq 2$ and that $\{G_n\}$ is partially reducibile, i.e.\
there exists $\rho \in (0,1]$,
$V'_n \subseteq V_n$ satisfying
\eqref{eq:V'rho} and
partitions $B_1,\ldots, B_{m_n}$ of $V'_n$ such that
the three properties \ref{item:red1''}, \ref{item:red2''} and \ref{item:red3''} of partial reducibility hold. Ordering the sets as in \eqref{e:ordering}, we can apply \eqref{e:simple} 
which yields \eqref{eq:bound-muk},
that we copy for convenience:
\begin{equation} \label{eq:bound-muk-new}
    \frac{\mu_k^{(n)}}{2}\leq \frac{\sum_{i=k}^{m_n} q^{(n)}_i  \phi(B_i)}{\sum_{i=k}^{m_n} q^{(n)}_i}
\qquad \text{where again} \quad
q^{(n)}_i := \frac{\text{vol}(B_i)}{\text{vol}(V_n)}
= \frac{\text{vol}(B_i)}{2|\cE_n|}\,.
\end{equation}
This time the vector $(q^{(n)}_i)_{1 \le i \le m_n}$ needs not be a probability, however we still have
\begin{equation}
    \sum_{i=1}^{m_n} q^{(n)}_i
    = \frac{\text{vol}(V'_n)}{\text{vol}(V_n)} \to \rho > 0 \,.
\end{equation}
To complete the proof, it is then sufficient to show that
\begin{equation}
    \sum_{i=1}^{m_n} q^{(n)}_i \phi(B_i) = o(1) \qquad \text{and} \qquad
    \max_{1\le i \le m_n} q^{(n)}_i = o(1)\,,
\end{equation}
(so that $\sum_{i=k}^{m_n} q^{(n)}_i \to \rho > 0$).
Applying
\eqref{eq:BBbar} and \eqref{eq:pbound-}, it remains to show that
\begin{equation} \label{eq:newgoal}
\sum_{i=1}^{m_n} E(B_i, \widebar{B_i} )= o(|\cE_n|) \,.
\end{equation}
To this purpose, we need to modify \eqref{eq:parti} because 
$B_1,\ldots, B_{m_n}$ is only a partition of $V'_n$. To obtain a partition of the full set of vertices $V_n$, we define $B_{m_n+1} := \widebar{V'_n}$.
Arguing as in \eqref{eq:parti}, we can now write
\begin{equation}\label{eq:parti-new}
    \frac{1}{2} 
\sum_{i=1}^{m_n+1} E(B_i, \widebar{B_i})
= \frac{1}{2} \sum_{\substack{i,j=1\\ i\neq j}}^{m_n+1} 
E(B_i, B_j)
= |\cE_n| - \sum_{i=1}^{m_n+1} E(B_i,B_i) 
\end{equation}
and isolating the terms $i=m_n+1$
in the first and last sums we obtain
\begin{equation*}
\begin{split}
\frac{1}{2} \sum_{i=1}^{m_n} E(B_i, \widebar{B_i}) 
& = \bigg\{ |\cE_n| - \frac{1}{2}E(V'_n, \widebar{V'_n}) - E(\widebar{V'_n}, \widebar{V'_n}) \bigg\} - \sum_{i=1}^{m_n} E(B_i,B_i) \\
& = \bigg\{ |\cE_n| - E(\widebar{V'_n}, \widebar{V'_n}) 
- o(|\cE_n|) \bigg\} -
E(V'_n, V'_n) - o(|\cE_n|) \,,
\end{split}
\end{equation*}
where we applied \eqref{eq:V'rho} and the assumption of partial reducibility.
We finally observe that, plainly,
$|\cE_n| = E(V'_n, V'_n) + E(\widebar{V'_n}, \widebar{V'_n})
+ E(V_n, \widebar{V'_n})
= E(V'_n, V'_n) + E(\widebar{V'_n}, \widebar{V'_n}) + o(|\cE_n|)$
which completes the proof.

\subsection{Proof of Theorem \ref{th:mainresulthyper}}

We first introduce and recall the following notation. For $\ell \le k$ and $S_1,\ldots,S_\ell \subseteq V_n$ disjoint and nonempty subsets, we denote
\begin{eqnarray*}
 && \cE(S_1\,,\ldots\,, S_\ell) \coloneq  \left\{ \, e \in \cE_n : e \subset  S_1 \cup \ldots \cup S_\ell : \right. \\ 
&&\quad\quad\quad\quad\quad\quad\quad\quad\quad\quad\quad\quad\quad\quad\left.  \exists \, v_1,\ldots,v_\ell \in e\,  \text{ with } \, v_1 \in S_1, \ldots, v_\ell \in S_\ell \,\right\}\,,
\end{eqnarray*}
as the set of edges with elements in $S_1 \cup \ldots \cup S_\ell$ and at least one of them in each $S_i$, for $i=1,\ldots,\ell$.
In particular, we recall
\begin{equation*}
\cE_n (S,S) := \big\lbrace \, e \in \cE_n : e \subset S \, \big\rbrace\,,
\end{equation*}
while 
\begin{equation}\label{eq:deltaS}
   \cE(S,\bar S) = \big\lbrace \, e \in \cE_n : \exists \, v_1,v_2 \in e \text{ with } v_1 \in S, \, v_2 \in \bar S \, \big\rbrace = \partial S\,.
\end{equation}

\medskip
We now prove Theorem \ref{th:mainresulthyper}. The arguments follow the same guidelines as the proof of Theorem~\ref{th:mainSPECTRUMgraph}. However, the more complex structure of hypergraphs requires additional combinatorial details, which we briefly outline below.

We assume that \eqref{eq:liminfhyp} holds for some $k\geq 2$ and that there exists a sequence of partitions $B_1,\ldots, B_{m_n}$, $n\ge 1$, such that the three properties \ref{item:irr1hyp}, \ref{item:irr2hyo}, \ref{item:irr3hyp} are verified. Moreover, without loss of generality, the relation \eqref{e:orderinghyp} holds for $n\ge 1$.
Recall that in this setting $r(\cG)=cr(\cG)=d$, hence \eqref{e:simplehyp} yields
\begin{equation} \label{eq:bound-mukhyp}
    \frac{\mu_k^{(n)}}{2(d-1)}\leq \frac{\sum_{i=k}^{m_n} q^{(n)}_i  \phi(B_i)}{\sum_{i=k}^{m_n} q^{(n)}_i}\,,
\quad \text{where } \
q^{(n)}_i := \frac{\mathrm{vol}\big(B_i\big)}{d \, w_n(\cE_n)}\,.
\end{equation}
To conclude the proof, it suffices to show that the right-hand side of the previous inequality converges necessarily to zero. 

Since the subsets $B_i$'s form a partition of $V_n$, it is possible to express the set $\partial B_i$, $i \in \{1,\ldots,m_n\}$, in terms of a disjoint union:
\begin{equation*}\label{eq:union1}
    \partial B_i= \cE(B_i,\widebar{B_i}) = \bigcup_{\alpha = 1}^{d-1} \ \bigcup_{\substack{1 \le i_1 < \cdots < i_\alpha \le m_n\\ i_1,\ldots,i_\alpha \ne i}} \cE(B_i,B_{i_1},\ldots,B_{i_\alpha} )\,,
\end{equation*}
(recall \eqref{eq:edgesinS} and \eqref{eq:deltaS}).
Moreover, note that $\cE_n \ \setminus \ \bigcup_{i=1}^{m_n} \cE(B_i,B_i )\,=\,
    \bigcup_{i=1}^{m_n} \cE(B_i,\widebar{B_i})$, where the union in the right--hand side is not disjoint, yet with explicit cardinality 
\begin{equation*}\label{eq:union3}
    \bigg| \, \bigcup_{i=1}^{m_n} \cE(B_i,\widebar{B_i})\,\bigg| = \sum_{\alpha = 1}^{d-1} \frac{1}{\alpha+1}\,\sum_{\substack{1 \le i_1 < \cdots < i_\alpha \le m_n\\ i_1,\ldots,i_\alpha \ne i}} \big| \cE(B_i,B_{i_1},\ldots,B_{i_\alpha} )\big|\,,
\end{equation*}
where the factor $\frac{1}{\alpha+1}$ compensates for overcounting: indeed, each edge that
intersects exactly \(\alpha+1\) blocks is counted once for every choice of the
distinguished \(B_i\) among them, and hence appears \(\alpha+1\) times in
the sum.
As a consequence, by \ref{item:irr2hyo} we have
\begin{equation}\label{eq:partihyp}
\begin{split}
\frac{1}{d}\, 
\sum_{i=1}^{m_n} w_n\big( \cE(B_i, \widebar{B_i}) \big)&
= \frac{1}{d}\, 
\sum_{i=1}^{m_n}\,\sum_{\alpha = 1}^{d-1} \,\sum_{\substack{1 \le i_1 < \cdots < i_\alpha \le m_n\\ i_1,\ldots,i_\alpha \ne i}} w_n\big( \cE(B_i,B_{i_1},\ldots,B_{i_\alpha} )\big)
 \\
&\le
\sum_{i=1}^{m_n}\,\sum_{\alpha = 1}^{d-1} \frac{1}{\alpha+1}\,\sum_{\substack{1 \le i_1 < \cdots < i_\alpha \le m_n\\ i_1,\ldots,i_\alpha \ne i}} w_n\big(\cE(B_i,B_{i_1},\ldots,B_{i_\alpha} )\big)\\
&= w_n(\cE_n) - \sum_{i=1}^{m_n} w_n\big( \cE(B_i,B_i)\big) = o\big(w_n(\cE_n)\big)
\end{split}
\end{equation}
and, then
\begin{equation} \label{eq:BBbarhyp}
    \sum_{i=1}^{m_n} \frac{w_n\big(\cE(B_i, \widebar{B_i}\big)}{d \, w_n(\cE_n)} = \sum_{i=1}^{m_n} \frac{w_n\big(\partial B_i \big)}{d \, w_n(\cE_n)} 
= \sum_{i=1}^{m_n} q^{(n)}_i \phi(B_i)
= o(1) \,,
\end{equation}
which already shows that the numerator in the bound
\eqref{eq:bound-mukhyp} for $\mu^{(n)}_k$ vanishes.
It remains to prove that the
denominator is bounded away from zero. We still have
\begin{equation*}
\begin{split}
    \sum_{i=1}^{m_n} q_i^{(n)} &= \frac{\sum_{i=1}^{m_n} \mathrm{vol}\big(B_i\big)}{d \, w_n(\cE_n)} = \frac{\sum_{i=1}^{m_n} \sum_{v \in B_i } \sum_{e \in \cE_v}w_n(e)}{d \, w_n(\cE_n)}
    = \frac{ \sum_{v \in V_n } \sum_{e \in \cE_v}w_n(e)}{d \, w_n(\cE_n)}=1\,,
\end{split}
\end{equation*}
(see \eqref{eq:volumeS}). 
Therefore, we just need to prove that $\sum_{i=1}^{k-1} q^{(n)}_i \to 0$, which is implied by 
\begin{equation} \label{eq:pbound}
\begin{split}
    q^{(n)}_i &=
\frac{\mathrm{vol}\big(B_i\big)}{d \, w_n(\cE_n)} =  \frac{\sum_{v \in B_i } \sum_{e \in \cE_v}w_n(e)}{d \, w_n(\cE_n)} \le \frac{ w_n\big(\cE(B_i,B_i\big)}{w_n(\cE_n)} +\frac{(d-1)\,w_n\big(\cE(B_i,\widebar{B_i})\big)}{d \, w_n(\cE_n)}= o(1) \,,
\end{split}
\end{equation}
uniformly for~$i \in \{1,\ldots, m_n\}$, where we applied \ref{item:irr3hyp} and \eqref{eq:BBbarhyp}.

\subsection{Proof of \eqref{e:cdfp}}\label{ss:combproof} We adopt the notation and assumptions of Example \ref{e:fcp}; also, given a permutation $\rho$ of $[b]$ and $A\subseteq V_n$, we write $A^\rho$ to denote the class of all $(v_1,...,v_b)\in V_n$ such that $(v_{\rho(1)},..., v_{\rho(b)})\in A$.  For every $s=1,...,d$, write $L_s$ to denote the set of those $\ell =1,...,d$ such that $\big|S_\ell \cap \{ (s,1),...,(s,b)\} \big| = 1$ (note that the size of the previous intersection is either zero or one, by construction). We stress that $|L_s| = b$, for $s=1,...,d$, and that each $\ell = 1,...,d$ is contained in exactly $b$ distinct sets $L_s$. Without loss of generality, we always label the elements of $L_s$ in such a way that, if $L_s = \{\ell_1,..., \ell_b\}$, then $\ell_1<\ell_2<\cdots < \ell_b$. For $s=1,...,d$, we denote by $$\pi_s : [n]^{d} \to [n]^{L_s} : {\bf i}= (i_1,...,i_d) \mapsto \pi_s({\bf i})=(i_{\ell_1},...,i_{\ell_b}),$$ where $\{\ell_1,...,\ell_b\}= L_s$. 
One can easily show that there exist permutations $\rho_1,..., \rho_d$ of $[b]$ such that 
\begin{equation}\label{e:hpi}
| E^0_n \cap (A_1\times \cdots \times A_d)| = \int_{[n] ^d} \prod_{s=1}^d {\bf 1}_{A_s^{\rho_s}} ( \pi_s ({\bf i} ))\, \nu_d({\rm d}{\bf i}) =  \int_{[n] ^d} \prod_{s=1}^d {\bf 1}_{A_s^{\rho_s}} ( \pi_s ({\bf i} ))^{1/b} \, \nu_d({\rm d}{\bf i}) ,
\end{equation}
where, for $t\geq 1$, $\nu_t$ stands for the counting measure on $[n]^t$. We can now directly apply \cite[Theorem 2.1]{F92} and deduce that the right-hand side of \eqref{e:hpi} is bounded by 
$$
\prod_{s=1}^d \left( \int_{[n] ^b} {\bf 1}_{A_s^{\rho_s}} ( {\bf i} ) \, \nu_b({\rm d}{\bf i}) \right)^{1/b}.
$$
Since $[n] ^b$ is a symmetric set, one has that
$$
\int_{[n] ^b} {\bf 1}_{A_s^{\rho_s}} ( {\bf i} ) \, \nu_b({\rm d}{\bf i}) = |A_s|,
$$
and the conclusion follows immediately. \qed

\section{Cartesian products}\label{appendixB}

In this appendix, we recall some basic definitions and properties of Cartesian products of regular graphs. See e.g. \cite[Chapters 4, 5 and 33]{ProductHandbook} and \cite[Section 1.4]{BrHaBook} for a full picture. Fix $d\geq 2$, and let $G = (V, \mathcal{E})$ be an undirected (loop-free) $d$-regular graph such that $|V| = N$. We denote by $ \lambda_1\geq \cdots \geq \lambda_N$ the eigenvalues of the adjacency matrix of $G$, and by $0 = \mu_1\leq \cdots \leq \mu_N\leq 2$ the eigenvalues of the corresponding normalized Laplacian; see Section \ref{ss:generalgraph} for details. Fix $m\geq 2$: throughout the paper, we use the symbols
$$ G^{\cpow m} = \underbrace{G\,  \raisebox{-0.25ex}{\scalebox{1.8}{$\square$}} \cdots \raisebox{-0.25ex}{\scalebox{1.8}{$\square$}} \, G}_{m\,\,{\rm times}} $$
to denote the {\em $m$th Cartesian product} of $G$. We recall that $G^{\cpow m}$ is the graph whose vertices are given by the set
$$
V^m = \{v =(v_1,...,v_m) : v_i\in V\} 
$$
and such that $v=(v_1,...,v_m)\sim w=(w_1,...,w_m)$ if and only if there exists $j\in [m]$ such that $v_i = w_i$ for all $i\neq j$ and $\{v_j , w_j\}\in \mathcal{E}$ (that is, if and only if $v_j,w_j$ are adjacent in $G$). For an arbitrary $(i_1,...,i_m)\in [N]^m$, we introduce the notation
\begin{equation}\label{e:cartspec}
    \Lambda(i_1,...,i_m):= \lambda_{i_1}+\cdots +\lambda_{i_m}, \quad \mbox{and}\quad M(i_1,...,i_m) := \mu_{i_1}+\cdots + \mu_{i_m}.
\end{equation}
The following facts are used in several parts of the paper, and can be routinely checked.
\begin{enumerate}
    \item $G^{\cpow m}$ is $(dm)$-regular and, consequently, the number of edges in $G^{\cpow m}$ is $2^{-1} dm N^m$.
    \item The spectrum of the adjacency matrix of $G^{\cpow m}$ is given by the set
    $$
     \left\{\Lambda(i_1,...,i_m) : (i_1,...,i_m)\in [N]^m\right\},
    $$
    where we have used the notation introduced in the first part of \eqref{e:cartspec}.
    \item The spectrum of the normalized Laplacian associated with $G^{\cpow m}$ is
    \[
    \left\{\tfrac{1}{m}\, M(i_1,\dots,i_m) : (i_1,\dots,i_m)\in [N]^m\right\},
    \]
    where we have used the notation appearing in the second part of \eqref{e:cartspec}. To see this, recall that, since $G^{\cpow m}$ is $(dm)$-regular and its adjacency eigenvalues are the sums
    $\Lambda(i_1,\dots,i_m)$ defined above, the associated normalized Laplacian eigenvalues are
    \[
    1 - \frac{\Lambda(i_1,\dots,i_m)}{dm}
    = \frac{1}{m}\left( \mu_{i_1} + \cdots + \mu_{i_m}\right),
    \]
    which yields the above formula. Note that the factor $1/m$ ensures that all eigenvalues 
    of the normalized Laplacian of $G^{\cpow m}$ lie in $[0,2]$, as they should.
\end{enumerate}

\bibliographystyle{abbrv}
\bibliography{bibliografia}

\begin{thebibliography}{10}

\bibitem{Alo86}
N.~Alon.
\newblock Eigenvalues and expanders.
\newblock {\em Combinatorica}, 6:83--96, 1986.

\bibitem{AM85}
N.~Alon and V.~D. Milman.
\newblock $\lambda_1$, isoperimetric inequalities for graphs, and superconcentrators.
\newblock {\em J. Comb. Theory, Ser. B}, 38(1):73--88, 1985.

\bibitem{AngstPolyEJP}
J.~Angst and G.~Poly.
\newblock Fluctuations in {Salem}-{Zygmund} almost sure central limit theorem.
\newblock {\em Electron. J. Probab.}, 28:40, 2023.
\newblock Id/No 44.

\bibitem{AnjosNeto}
M.~F. Anjos and J.~Neto.
\newblock Spectral bounds for graph partitioning with prescribed partition sizes.
\newblock {\em Discrete Appl. Math.}, 269:200--210, 2019.

\bibitem{aroraFlowEmbedding}
S.~Arora, S.~Rao, and U.~Vazirani.
\newblock Expander flows, geometric embeddings and graph partitioning.
\newblock {\em J. ACM}, 56(2), Apr. 2009.

\bibitem{baldirinott}
P.~Baldi and Y.~Rinott.
\newblock On normal approximations of distributions in terms of dependency graphs.
\newblock {\em Ann. Probab.}, 17(4):1646--1650, 1989.

\bibitem{B21}
A.~Banerjee.
\newblock On the spectrum of hypergraphs.
\newblock {\em Linear Algebra and its Applications}, 614:82--110, 2021.

\bibitem{BCCT08}
J.~Bennett, A.~Carbery, M.~Christ, and T.~Tao.
\newblock The brascamp–lieb inequalities: Finiteness, structure and extremals.
\newblock {\em Geom. Funct. Anal.}, 17:1343--1415, 2008.

\bibitem{BDMM22}
B.~Bhattacharya, S.~Das, S.~Mukherjee, and S.~Mukherjee.
\newblock Fluctuations of quadratic chaos.
\newblock {\em Commun. Math. Phys.}, 405(10):51, 2024.
\newblock Id/No 237.

\bibitem{bleiHarmonic1}
R.~Blei.
\newblock Combinatorial dimension and certain norms in harmonic analysis.
\newblock {\em Am. J. Math.}, 106:847--887, 1984.

\bibitem{bleibook}
R.~Blei.
\newblock {\em Analysis in integer and fractional dimensions}, volume~71 of {\em Camb. Stud. Adv. Math.}
\newblock Cambridge: Cambridge University Press, 2001.

\bibitem{Blei2011Survey}
R.~Blei.
\newblock Measurements of interdependence.
\newblock {\em Lith. Math. J.}, 51(2):141--154, 2011.

\bibitem{bleijanson}
R.~Blei and S.~Janson.
\newblock Rademacher chaos: tail estimates versus limit theorems.
\newblock {\em Ark. Mat.}, 42(1):13--29, 2004.

\bibitem{BleiPeresScchmerl}
R.~Blei, Y.~Peres, and J.~Schmerl.
\newblock Fractional products of sets.
\newblock {\em Random Struct. Algorithms}, 6(1):113--119, 1995.

\bibitem{blei1979}
R.~C. Blei.
\newblock Fractional cartesian products of sets.
\newblock {\em Ann. Inst. Fourier}, 29(2):79--105, 1979.

\bibitem{bleiharmonic2}
R.~C. Blei.
\newblock {\em Fractional dimensions and bounded fractional forms}, volume 331 of {\em Mem. Am. Math. Soc.}
\newblock Providence, RI: American Mathematical Society (AMS), 1985.

\bibitem{BleiSRandom}
R.~C. Blei and T.~W. K\"orner.
\newblock Combinatorial dimension and random sets.
\newblock {\em Israel J. Math.}, 47(1):65--74, 1984.

\bibitem{bollobas}
B.~Bollob{\'a}s.
\newblock Random graphs.
\newblock London-Orlando etc.: {Academic} {Press} ({Harcourt} {Brace} {Jovanovich}, {Publishers}). {XVI}, 447 p. hbk: {{\textsterling}} 52.00; {\$} 58.50; pbk: {{\textsterling}} 26.00; {\$} 29.95 (1985)., 1985.

\bibitem{BrHaBook}
A.~E. Brouwer and W.~H. Haemers.
\newblock {\em Spectra of graphs}.
\newblock Universitext. Springer, New York, 2012.

\bibitem{CC22}
F.~Caravenna and F.~Cottini.
\newblock Gaussian limits for subcritical chaos.
\newblock {\em Electron. J. Probab.}, 27:1--35, 2022.

\bibitem{CSZ1}
F.~Caravenna, R.~Sun, and N.~Zygouras.
\newblock Polynomial chaos and scaling limits of disordered systems.
\newblock {\em J. Eur. Math. Soc. (JEMS)}, 19(1):1--65, 2017.

\bibitem{CSZ17b}
F.~Caravenna, R.~Sun, and N.~Zygouras.
\newblock Universality in marginally relevant disordered systems.
\newblock {\em Ann. Appl. Probab.}, 27:3050--3112, 2017.

\bibitem{CSZ3}
F.~Caravenna, R.~Sun, and N.~Zygouras.
\newblock The critical 2d stochastic heat flow.
\newblock {\em Invent. Math.}, 233(1):325--460, 2023.

\bibitem{Cha08}
S.~Chatterjee.
\newblock {A new method of normal approximation}.
\newblock {\em The Annals of Probability}, 36(4):1584 -- 1610, 2008.

\bibitem{Che71}
J.~Cheeger.
\newblock A lower bound for the smallest eigenvalue of the laplacian.
\newblock In {\em Problems in Analysis}, pages 195--200. Princeton University Press, Princeton, 1971.

\bibitem{ChungSpectralGraph}
F.~R.~K. Chung.
\newblock {\em Spectral graph theory}, volume~92 of {\em CBMS Regional Conference Series in Mathematics}.
\newblock Conference Board of the Mathematical Sciences, Washington, DC; by the American Mathematical Society, Providence, RI, 1997.

\bibitem{LarryGraphs}
N.~Cook, L.~Goldstein, and T.~Johnson.
\newblock Size biased couplings and the spectral gap for random regular graphs.
\newblock {\em Ann. Probab.}, 46(1):72--125, 2018.

\bibitem{DeServedio2}
A.~De, I.~Diakonikolas, and R.~A. Servedio.
\newblock Deterministic approximate counting for juntas of degree-2 polynomial threshold functions.
\newblock In {\em 2014 IEEE 29th Conference on Computational Complexity (CCC)}, pages 229--240, 2014.

\bibitem{DeServedio1}
A.~De and R.~A. Servedio.
\newblock Efficient deterministic approximate counting for low-degree polynomial threshold functions.
\newblock In {\em Proceedings of the Forty-Sixth Annual ACM Symposium on Theory of Computing}, STOC '14, page 832–841, New York, NY, USA, 2014. Association for Computing Machinery.

\bibitem{dJ87}
P.~de~Jong.
\newblock A central limit theorem for generalized quadratic forms.
\newblock {\em Probab. Theory Relat. Fields}, 75:261--277, 1987.

\bibitem{dJ90}
P.~de~Jong.
\newblock A central limit theorem for generalized multilinear forms.
\newblock {\em J. Multivariate Anal.}, 34:275--289, 1990.

\bibitem{dejongRSA}
P.~de~Jong.
\newblock A central limit theorem with applications to random hypergraphs.
\newblock {\em Random Struct. Algorithms}, 8(2):105--120, 1996.

\bibitem{DeyaNourdinCMP}
A.~Deya, S.~Noreddine, and I.~Nourdin.
\newblock Fourth moment theorem and {{\(q\)}}-{Brownian} chaos.
\newblock {\em Commun. Math. Phys.}, 321(1):113--134, 2013.

\bibitem{DeyaNourdinBernoulli}
A.~Deya and I.~Nourdin.
\newblock Invariance principles for homogeneous sums of free random variables.
\newblock {\em Bernoulli}, 20(2):586--603, 2014.

\bibitem{DierickxNPRCMP}
G.~Dierickx, I.~Nourdin, G.~Peccati, and M.~Rossi.
\newblock Small scale {CLTs} for the nodal length of monochromatic waves.
\newblock {\em Commun. Math. Phys.}, 397(1):1--36, 2023.

\bibitem{DKP22}
D{\"o}bler, Kasprzak, and Peccati.
\newblock Title missing.
\newblock {\em Probab. Theory Relat. Fields}, 2022.
\newblock Dettagli da completare.

\bibitem{DoblerSPL}
C.~D{\"o}bler.
\newblock The {Berry}-{Esseen} bound in de {Jong}'s {CLT}.
\newblock {\em Stat. Probab. Lett.}, 215:8, 2024.
\newblock Id/No 110244.

\bibitem{DPKPTRF}
C.~D{\"o}bler, M.~Kasprzak, and G.~Peccati.
\newblock The multivariate functional de {Jong} {CLT}.
\newblock {\em Probab. Theory Relat. Fields}, 184(1-2):367--399, 2022.

\bibitem{DoblerKrokowski}
C.~D\"obler and K.~Krokowski.
\newblock On the fourth moment condition for {R}ademacher chaos.
\newblock {\em Ann. Inst. Henri Poincar\'e{} Probab. Stat.}, 55(1):61--97, 2019.

\bibitem{DPejp}
C.~D{\"o}bler and G.~Peccati.
\newblock Quantitative de {Jong} theorems in any dimension.
\newblock {\em Electron. J. Probab.}, 22:35, 2017.
\newblock Id/No 2.

\bibitem{DuPanconesi}
D.~P. Dubhashi and A.~Panconesi.
\newblock {\em Concentration of measure for the analysis of randomized algorithms.}
\newblock Cambridge: Cambridge University Press, 2009.

\bibitem{efronstein}
B.~Efron and C.~Stein.
\newblock The jackknife estimate of variance.
\newblock {\em Ann. Stat.}, 9:586--596, 1981.

\bibitem{F92}
H.~Finner.
\newblock A generalization of h{\"o}lder's inequality and some probability inequalities.
\newblock {\em Ann. Probab.}, 20(4):1893--1901, 1992.

\bibitem{FriedmanSecondEig}
J.~Friedman.
\newblock On the second eigenvalue and random walks in random {{\(d\)}}-regular graphs.
\newblock {\em Combinatorica}, 11(4):331--362, 1991.

\bibitem{Garban2011}
C.~Garban.
\newblock Oded {S}chramm's contributions to noise sensitivity.
\newblock {\em Ann. Probab.}, 39(5):1702--1767, 2011.

\bibitem{GPS10}
C.~Garban, G.~Pete, and O.~Schramm.
\newblock The {F}ourier spectrum of critical percolation.
\newblock {\em Acta Math.}, 205(1):19--104, 2010.

\bibitem{GodsilRoy}
C.~Godsil and G.~Royle.
\newblock {\em Algebraic graph theory}, volume 207 of {\em Graduate Texts in Mathematics}.
\newblock Springer-Verlag, New York, 2001.

\bibitem{ProductHandbook}
R.~Hammack, W.~Imrich, and S.~Klav{\v{z}}ar.
\newblock {\em Handbook of product graphs}.
\newblock Discrete Math. Appl. (Boca Raton). Boca Raton, FL: CRC Press, 2nd ed. edition, 2011.

\bibitem{HeMaPolyLaw}
R.~Herry, D.~Malicet, and G.~Poly.
\newblock Limit distributions for polynomials with independent and identically distributed entries, 2024.

\bibitem{Hoeffding}
W.~Hoeffding.
\newblock A class of statistics with asymptotically normal distribution.
\newblock {\em Ann. Math. Stat.}, 19:293--325, 1948.

\bibitem{hlw}
S.~Hoory, N.~Linial, and A.~Wigderson.
\newblock Expander graphs and their applications.
\newblock {\em Bull. Amer. Math. Soc.}, 43(4):439--561, 2006.

\bibitem{JansonBook}
S.~Janson.
\newblock {\em Gaussian {H}ilbert spaces}, volume 129 of {\em Cambridge Tracts in Mathematics}.
\newblock Cambridge University Press, Cambridge, 1997.

\bibitem{kallengerg}
O.~Kallenberg.
\newblock {\em Foundations of modern probability}.
\newblock Probab. Appl. New York, NY: Springer, 1997.

\bibitem{KNPSaop}
T.~Kemp, I.~Nourdin, G.~Peccati, and R.~Speicher.
\newblock Wigner chaos and the fourth moment.
\newblock {\em Ann. Probab.}, 40(4):1577--1635, 2012.

\bibitem{KoikeJOTP}
Y.~Koike.
\newblock High-dimensional central limit theorems for homogeneous sums.
\newblock {\em J. Theor. Probab.}, 36(1):1--45, 2023.

\bibitem{korolyuk}
V.~S. Korolyuk and Y.~V. Borovskikh.
\newblock {\em Theory of {{\(U\)}}-statistics. {Updated} and transl. from the {Russian} by {P}. {V}. {Malyshev} and {D}. {V}. {Malyshev}}, volume 273 of {\em Math. Appl., Dordr.}
\newblock Dordrecht: Kluwer Academic Publishers, 1994.

\bibitem{largesudakov}
M.~Krivelevich and B.~Sudakov.
\newblock The largest eigenvalue of sparse random graphs.
\newblock {\em Comb. Probab. Comput.}, 12(1):61--72, 2003.

\bibitem{chessboardgraphs}
R.~Laskar and C.~Wallis.
\newblock Chessboard graphs, related designs, and domination parameters.
\newblock {\em J. Stat. Plann. Inference}, 76(1-2):285--294, 1999.

\bibitem{LPY}
G.~Last, G.~Peccati, and D.~Yogeshwaran.
\newblock Phase transitions and noise sensitivity on the {Poisson} space via stopping sets and decision trees.
\newblock {\em Random Struct. Algorithms}, 63(2):457--511, 2023.

\bibitem{LGT14}
J.~Lee, S.~O. Gharan, and L.~Trevisan.
\newblock Multiway spectral partitioning and higher-order cheeger inequalities.
\newblock {\em J. ACM}, 61(6), 2014.

\bibitem{Vempala12}
A.~Louis, P.~Raghavendra, P.~Tetali, and S.~Vempala.
\newblock Many sparse cuts via higher eigenvalues.
\newblock In {\em Proceedings of the Forty-Fourth Annual ACM Symposium on Theory of Computing}, STOC '12, page 1131–1140, New York, NY, USA, 2012. Association for Computing Machinery.

\bibitem{LuboSurvey}
A.~Lubotzky.
\newblock Expander graphs in pure and applied mathematics.
\newblock {\em Bull. Amer. Math. Soc. (N.S.)}, 49(1):113--162, 2012.

\bibitem{MandelbaumTaqqu83}
A.~Mandelbaum and M.~S. Taqqu.
\newblock Invariance principle for symmetric statistics.
\newblock {\em Ann. Stat.}, 12:483--496, 1984.

\bibitem{MOO10}
E.~Mossel, R.~O'Donnell, and K.~Oleszkiewicz.
\newblock Noise stability of functions with low influences: invariance and optimality.
\newblock {\em Ann. Math. (2)}, 171(1):295--341, 2010.

\bibitem{rammurtysurvey}
M.~R. Murty.
\newblock Ramanujan graphs.
\newblock {\em J. Ramanujan Math. Soc.}, 18(1):33--52, 2003.

\bibitem{NouPecALEA}
I.~Nourdin and G.~Peccati.
\newblock Universal {Gaussian} fluctuations of non-{Hermitian} matrix ensembles: from weak convergence to almost sure {CLTs}.
\newblock {\em ALEA, Lat. Am. J. Probab. Math. Stat.}, 7:341--375, 2010.

\bibitem{NP12}
I.~Nourdin and G.~Peccati.
\newblock {\em Normal Approximations with Malliavin Calculus: From Stein's Method to Universality}.
\newblock Cambridge University Press, 2012.

\bibitem{NPPS}
I.~Nourdin, G.~Peccati, G.~Poly, and R.~Simone.
\newblock Classical and free fourth moment theorems: universality and thresholds.
\newblock {\em J. Theor. Probab.}, 29(2):653--680, 2016.

\bibitem{NPR10}
I.~Nourdin, G.~Peccati, and G.~Reinert.
\newblock Invariance principles for homogeneous sums: universality of gaussian wiener chaos.
\newblock {\em Ann. Probab.}, 38:1947--1985, 2010.

\bibitem{NPR10b}
I.~Nourdin, G.~Peccati, and G.~Reinert.
\newblock Stein's method and stochastic analysis of rademacher functionals.
\newblock {\em Electron. J. Probab.}, 15:1703--1742, 2010.

\bibitem{NouPecRossi19}
I.~Nourdin, G.~Peccati, and M.~Rossi.
\newblock Nodal statistics of planar random waves.
\newblock {\em Commun. Math. Phys.}, 369(1):99--151, 2019.

\bibitem{NouPecSwanJFA}
I.~Nourdin, G.~Peccati, and Y.~Swan.
\newblock Entropy and the fourth moment phenomenon.
\newblock {\em J. Funct. Anal.}, 266(5):3170--3207, 2014.

\bibitem{NuaBook}
D.~Nualart.
\newblock {\em The {M}alliavin calculus and related topics}.
\newblock Probability and its Applications (New York). Springer-Verlag, Berlin, second edition, 2006.

\bibitem{NP05}
D.~Nualart and G.~Peccati.
\newblock Central limit theorems for sequences of multiple stochastic integrals.
\newblock {\em Ann. Probab.}, 33(1):177--193, 2005.

\bibitem{ODonnellBook}
R.~O'Donnell.
\newblock {\em Analysis of {Boolean} functions}.
\newblock Cambridge: Cambridge University Press, 2014.

\bibitem{PecReitz}
G.~Peccati and M.~Reitzner, editors.
\newblock {\em Stochastic analysis for {Poisson} point processes. {Malliavin} calculus, {Wiener}-{It{\^o}} chaos expansions and stochastic geometry}, volume~7 of {\em Bocconi Springer Ser.}
\newblock Milano: Bocconi University Press; Cham: Springer, 2016.

\bibitem{PecTaqBook}
G.~Peccati and M.~Taqqu.
\newblock {\em Wiener chaos: {Moments}, cumulants and diagrams. {A} survey with computer implementation}, volume~1 of {\em Bocconi Springer Ser.}
\newblock Milano: Bocconi University Press; Milano: Springer, 2011.

\bibitem{PecVidottoJMP}
G.~Peccati and A.~Vidotto.
\newblock Gaussian random measures generated by {Berry}'s nodal sets.
\newblock {\em J. Stat. Phys.}, 178(4):996--1027, 2020.

\bibitem{PuderInventiones}
D.~Puder.
\newblock Expansion of random graphs: new proofs, new results.
\newblock {\em Invent. Math.}, 201(3):845--908, 2015.

\bibitem{R79}
V.~I. Rotar'.
\newblock Limit theorems for polylinear forms.
\newblock {\em J. Multivariate Anal.}, 9:511--530, 1979.

\bibitem{SSPHypergraphs}
S.~Saha, K.~Sharma, and S.~Panda.
\newblock On the laplacian spectrum of k-uniform hypergraphs.
\newblock {\em Linear Algebra and its Applications}, 655:1--27, 2022.

\bibitem{serfling}
R.~J. Serfling.
\newblock {\em Approximation theorems of mathematical statistics}.
\newblock Wiley Ser. Probab. Math. Stat. John Wiley \& Sons, Hoboken, NJ, 1980.

\bibitem{SimoneALEA}
R.~Simone.
\newblock Universality of free homogeneous sums in every dimension.
\newblock {\em ALEA, Lat. Am. J. Probab. Math. Stat.}, 12(1):213--244, 2015.

\bibitem{SJ89}
A.~Sinclair and M.~Jerrum.
\newblock Approximate counting, uniform generation and rapidly mixing markov chains.
\newblock {\em Inf. Comput.}, 82(1):93--133, 1989.

\bibitem{trevisanLN}
L.~Trevisan.
\newblock Lecture notes on graph partitioning, expanders and spectral methods.
\newblock Lecture notes, 2016.
\newblock Link: \url{https://lucatrevisan.github.io/books/expanders-2016.pdf}.

\end{thebibliography}

\end{document}